\documentclass[print,wide,eng]{draft}
\usepackage{placeins}
\crefname{enumi}{Assumption}{Assumptions}
\Crefname{enumi}{Assumption}{Assumptions}
\hypersetup{pdftitle={From Continuous to Fully Discrete Carleman Estimates},pdfauthor={Qi Lu and Yu Wang}}
\title{\bf From Continuous to Fully Discrete Carleman Estimates:\\
A Transfer Principle for Parabolic Schemes}
\author{Qi L\"u\thanks{School of Mathematics, Sichuan University, Chengdu 610064, China. E-mail: \texttt{lu@scu.edu.cn}. Qi L\"u is supported by the National Natural Science Foundation of China under grant 12025105.}
\and
Yu Wang\thanks{School of Mathematics, Southwest Jiaotong University, Chengdu 611756, China. E-mail: \texttt{yuwangmath@163.com}. Yu Wang is supported by the National Natural Science Foundation of China under grant 12401589 and by Sichuan Science and Technology Program under grant 2026NSFSC0777.}}
\date{}

\hypersetup{linkcolor=black,citecolor=black,urlcolor=black}

\begin{document}
\maketitle
\begin{abstract}
Direct approaches to discrete parabolic Carleman estimates often rely on weighted identities tailored to the discrete operator. We prove an abstract transfer theorem that derives such estimates for fully discrete backward Euler schemes with conforming spatial reconstructions from a continuous Carleman inequality together with suitable weighted approximation and compatibility properties. For a class of time-independent positive self-adjoint operators, the key construction assigns to each discrete state a shifted continuous auxiliary problem whose discrete solution is exactly the prescribed state. Weighted space-time error estimates, uniform over the full discrete load space, then transfer the inequality while preserving the residual and observation terms of the original scheme. This separates the PDE-specific Carleman analysis from the numerical verification required for each discretization. We verify the hypotheses for the heat equation discretized by consistent-mass $P_1$ finite elements on locally graded meshes in two and three dimensions and by standard Cartesian finite differences in any fixed dimension. For the heat equation, the transferred estimate is valid for Carleman parameters up to the scale $\min\{h^{-4/5},(\delta t)^{-2/5}\}$. For bounded real-valued space-time potentials, the resulting estimates yield relaxed observability and approximate null controls with uniformly bounded cost and exponentially small terminal errors under independent spatial and temporal refinement; weak accumulation points of the reconstructed controls are continuous null controls.
\end{abstract}



\section{Introduction}\label{sec:intro}

Partial differential equations with distributed controls arise in many applications where the state is to be steered to a prescribed target by acting only on a small subregion of the domain. For parabolic equations, controllability and the cost of the control are classically characterized by observability estimates for the adjoint equation, which in turn follow from global Carleman estimates. When the problem is discretized for numerical computation, however, these properties need not survive uniformly: a scheme that is consistent and stable may still lose uniform controllability as the mesh and the time step are refined. The present paper addresses this issue for fully discrete parabolic schemes by transferring a continuous Carleman estimate directly to the discrete setting.

This introduction is organized as follows. In \Cref{intro:heat} we formulate the numerical control problem and state the discrete control bounds that motivate the paper. In \Cref{intro:transfer}, we recall the continuous observability and Carleman estimates from which these bounds will be derived, and explain the obstructions to a direct perturbation argument. The transfer mechanism is then illustrated on an elliptic model in \Cref{sec:elliptic-transfer}, and the relations to previous work are summarized in \cref{intro:comparison}.

\subsection{The numerical control problem}\label{intro:heat}

Let $\Omega\subset\mathbb R^d$ be a bounded convex polytope, and let $\omega\Subset\Omega$ be a nonempty open set.
Fix $T>0$, $M\geq0$, and  $b\in L^\infty((0,T)\times\Omega;\mathbb R)$ with $\|b\|_{L^\infty}\leq M$.
Consider the controlled heat equation:
\begin{align}\label{intro:heat-model}
\begin{cases}
\partial_t y-\Delta y+b(t,x)y=\mathbf1_\omega u,& \text{ in } (0,T)\times\Omega,\\
y=0,& \text{ on } (0,T)\times\partial\Omega,\\
y(0)=y_0\in L^2(\Omega).
\end{cases}
\end{align}
The null controllability result~\cite{Fursikov1996} provides, for every $y_0$, a control $u\in L^2((0,T)\times\omega)$ such that
$ y(T)=0$.
The control satisfies the cost bound
\begin{align}\label{intro:continuous-cost}
\|u\|_{L^2((0,T)\times\omega)}\leq C\|y_0\|_{L^2(\Omega)}.
\end{align}
Here $C$ depends only on $\Omega,\omega,T,M$.
Thus the target is reached with a control cost bounded linearly in the size of the initial datum.

Given the above null controllability result for \eqref{intro:heat-model}, a natural question is how to numerically construct a control that drives the system state to zero. To address this question, one can follow two approaches: either formulate a control construction for the continuous equation and then discretize it, or first discretize the state equation and then construct controls for the resulting discrete system; see \cite[Section~1]{Zuazua2005}. In this paper, we adopt the latter, discretize-then-control approach. A key issue, however, is that convergence of the forward scheme alone does not guarantee that the associated controls remain bounded under refinement. Therefore, the numerical control problem is to keep the control cost bounded while simultaneously driving the terminal error to zero as the spatial mesh and the time step are refined.

To make this discussion concrete, we focus on dimensions $d=2,3$ and let $V_h\subset H_0^1(\Omega)$ be the conforming $P_1$ space on a uniformly shape-regular mesh satisfying the grading condition \cref{eq:grading}, with maximum element diameter $h$. Define $A_h:V_h\to V_h$ by
\begin{align*}
(A_hz_h,v_h)_{L^2(\Omega)}=(\nabla z_h,\nabla v_h)_{L^2(\Omega)},
\qquad z_h,v_h\in V_h.
\end{align*}
For an integer $N\geq2$, write $\delta t=T/N$, and set
$b^n=(\delta t)^{-1}\int_{(n-1)\delta t}^{n\delta t}b(t,\cdot)\,dt$.
Given $y^0=y_{0,h}\in V_h$ and controls $u^n\in L^2(\omega)$, the backward-Euler scheme seeks $y^n\in V_h$ satisfying
\begin{align}\label{intro:discrete-model}
\biggl(\frac{y^n-y^{n-1}}{\delta t},v_h\biggr)_{L^2(\Omega)}
+(A_hy^n,v_h)_{L^2(\Omega)}+(b^ny^n,v_h)_{L^2(\Omega)}
=\int_\omega u^nv_h\,dx,
\qquad v_h\in V_h,
\end{align}
for $1\leq n\leq N$.

For sufficiently small $h$ and $\delta t$, the results in \cref{cor:scales,prop:physicalC} provide controls satisfying
\begin{align}\label{intro:control-bounds}
\delta t\sum_{n=1}^N\|u^n\|_{L^2(\omega)}^2
&\leq C\|y_{0,h}\|_{L^2(\Omega)}^2.
\end{align}
Here and in the terminal estimate below, the constants may also depend on the fixed mesh regularity and grading bounds, but are independent of $h$, $\delta t$, and $y_{0,h}$.
This is the discrete analogue of the control-cost bound in \cref{intro:continuous-cost}: the control norm is bounded linearly by the norm of the initial datum, uniformly under spatial and temporal refinement. Thus the control cost remains uniformly bounded.

Moreover, the same controls also satisfy the terminal error estimate
\begin{align}\label{intro:terminal-error}
\|y^N\|_{L^2(\Omega)}
&\leq C\exp\!\big[-c\min\{h^{-4/5},(\delta t)^{-2/5}\}\big]
\|y_{0,h}\|_{L^2(\Omega)}.
\end{align}
Here $c>0$ is also independent of $h$ and $\delta t$.
Thus, for a family of initial data bounded in $L^2(\Omega)$, the terminal estimate gives
\begin{align*}
\lim_{h,\delta t\to0}\|y^N\|_{L^2(\Omega)}=0,
\qquad N=T/\delta t,
\end{align*}
with an exponentially small bound under independent spatial and temporal refinement.

As a consequence of \cref{intro:control-bounds}, for approximations of a fixed initial datum, the piecewise constant control reconstructions remain bounded in $L^2((0,T)\times\omega)$ and therefore admit weakly convergent subsequences as $h,\delta t\to0$.
Together with consistency of the scheme, \cref{intro:terminal-error} ensures that every such weak limit is a null control for \cref{intro:heat-model} (see \cref{prop:control-convergence} for more details).

\subsection{From continuous to discrete Carleman estimates}\label{intro:transfer}

To derive control bounds for the discrete problem, it is natural to work with the adjoint equation and to seek a discrete counterpart of the continuous observability inequality. Classically, the continuous inequality is deduced from a global parabolic Carleman estimate. A natural attempt is therefore to apply the continuous Carleman estimate directly to a reconstruction of the discrete state and to absorb the consistency error as a small perturbation. However, for fully discrete schemes this direct perturbation argument is not sufficient. The obstruction is twofold: the continuous estimate does not control the highest discrete modes uniformly, and consistency alone does not make the perturbation small over the full discrete space. These difficulties motivate the transfer principle developed in this paper, in which each discrete state is embedded into a shifted continuous auxiliary problem and the continuous Carleman inequality is transferred back to the original scheme through weighted numerical error estimates.

Write $Pp=-\partial_t p-\Delta p+b(t,x)p$ for the continuous adjoint operator. By duality~\cite{Lions1988}, the null controllability of \eqref{intro:heat-model} and the control cost \cref{intro:continuous-cost} are equivalent to an observability estimate for the backward equation
\begin{align}\label{intro:continuous-adjoint}
\begin{cases}
Pp=0,& \text{ in }(0,T)\times\Omega,\\
p=0,& \text{ on }(0,T)\times\partial\Omega,\\
p(T)=p_T\in L^2(\Omega).
\end{cases}
\end{align}
Specifically, every solution satisfies
\begin{align}\label{intro:continuous-observability}
\|p(0)\|_{L^2(\Omega)}^2
\leq C\int_0^T\|p(t)\|_{L^2(\omega)}^2\,dt.
\end{align}

At the continuous level, observability is classically obtained from a global parabolic Carleman estimate~\cite{Fursikov1996}. Let $\alpha<0$ and $\xi>0$ be standard parabolic weight functions. For all sufficiently large $s$ and all sufficiently regular functions $p$ vanishing on $(0,T)\times\partial\Omega$, the gradient and zeroth-order terms satisfy
\begin{align}\label{intro:classical-parabolic-carleman}
&\int_0^T\!\int_\Omega e^{2s\alpha}
\bigl(s\xi|\nabla p|^2+s^3\xi^3|p|^2\bigr)\,dx\,dt
\leq C\int_0^T\!\int_\Omega e^{2s\alpha}|Pp|^2\,dx\,dt
+Cs^3\int_0^T\!\int_\omega e^{2s\alpha}\xi^3|p|^2\,dx\,dt.
\end{align}
This, together with an energy estimate,  yields \cref{intro:continuous-observability} and hence null controllability by duality.

These observations suggest a direct route to numerical control. A consistent discretization approximates the continuous operator on smooth functions, so one might try to apply \cref{intro:classical-parabolic-carleman} to a reconstruction of a discrete state and absorb the consistency error as a small perturbation. To recover \cref{intro:control-bounds,intro:terminal-error}, the error must be absorbable over the full discrete space and for a range of $s$ growing as the mesh is refined. As already indicated, two obstructions to this approach are explained in \Cref{sec:elliptic-transfer}.

To overcome these obstructions, we prove an abstract transfer theorem, \cref{thm:main}, for fully discrete schemes satisfying weighted approximation and compatibility assumptions. It yields a Carleman estimate over the full discrete space and applies, in particular, to the finite-element and finite-difference schemes verified in this paper.

To illustrate this general result concretely, we return to the finite-element scheme in \Cref{intro:heat} and define the discrete adjoint operator by
\begin{align*}
(P_{h,\delta t}p)^n
:=\frac{p^n-p^{n+1}}{\delta t}+A_hp^n+\Pi_h(b^{n+1}p^n),
\qquad 0\leq n<N,
\end{align*}
where $\Pi_h$ is the $L^2(\Omega)$ projection onto $V_h$.
We use regularized weight functions with matching time endpoints and write $\alpha^n(x)=\alpha(n\delta t,x)$ and $\xi^n=\xi(n\delta t)$, with $\max_{[0,T]}\xi=1$. The weight conditions and normalization are specified in \cref{eq:qclass}.
For sufficiently small $h,\delta t$, \cref{thm:main,cor:scales} give, for every sequence $p^n\in V_h$ with $p^N=p^0$,
\begin{align}\label{intro:parabolic-discrete-carleman}
&\delta t\sum_{n=0}^{N-1}\int_\Omega e^{2s\alpha^n}
\bigl(s\xi^n|\nabla p^n|^2+s^3(\xi^n)^3|p^n|^2\bigr)\,dx
\notag
\\
&
\leq C\delta t\sum_{n=0}^{N-1}\int_\Omega e^{2s\alpha^n}
|(P_{h,\delta t}p)^n|^2\,dx
+Cs^3\delta t\sum_{n=0}^{N-1}\int_\omega e^{2s\alpha^n}(\xi^n)^3|p^n|^2\,dx,
\end{align}
provided $s\min\limits_{[0,T]}\xi\geq s_0$ and
$
s_0\leq s\leq\kappa\min\{h^{-4/5},(\delta t)^{-2/5}\}$.
Here $C,\kappa,s_0>0$ are independent of $h,\delta t,s$ and the allowed weights; $s_0$ is chosen sufficiently large depending on $M$ to absorb the potential. The two terms on the right are the residual and observation of the original scheme, just as in \cref{intro:classical-parabolic-carleman}.

For a prescribed terminal state $p_T\in V_h$, the discrete adjoint of \cref{intro:discrete-model} reads
\begin{align}\label{intro:discrete-adjoint}
(P_{h,\delta t}p)^n=0,\qquad 0\leq n<N,
\quad p^N=p_T.
\end{align}
To apply \cref{intro:parabolic-discrete-carleman}, we cut off $p$ near the time endpoints. Since $P_{h,\delta t}p=0$, the residual of the cutoff sequence contains only cutoff differences. Choosing the weights at the admissible mesh scale and combining \cref{intro:parabolic-discrete-carleman} with a discrete energy estimate yields
\begin{align}\label{intro:observability}
\|p^0\|_{L^2(\Omega)}^2
\leq C\Bigl(
{\delta t}\sum_{n=0}^{N-1}\|p^n\|_{L^2(\omega)}^2
+e^{-2c\min\{h^{-4/5},{(\delta t)}^{-2/5}\}}\|p^N\|_{L^2(\Omega)}^2
\Bigr),
\end{align}
where $C,c$ are independent of $h$ and ${\delta t}$. Compared with \cref{intro:continuous-observability}, this estimate has an exponentially small terminal term arising from the endpoint cutoff, hence the name \emph{relaxed observability}.
Minimizing a quadratic functional of the terminal adjoint datum then gives the control bounds in \cref{intro:control-bounds,intro:terminal-error} (see \cref{prop:physicalO,prop:physicalC}).
With suitable filtering of the terminal adjoint data, the remainder can be absorbed, yielding uniform observability  (see \cref{thm:filtered-observability}).

The heat-equation example above illustrates the general transfer framework developed in this paper. In summary, our main contributions are as follows.
\begin{enumerate}[itemsep=3pt]
\item We establish an abstract framework for transferring continuous Carleman estimates to fully discrete schemes through weighted approximation estimates uniform over the full discrete space.
\item The transferred inequality retains the residual and observation of the original scheme and allows distinct physical and reconstruction masses and a small weighted stiffness defect.
\item We verify the heat-equation framework for locally graded consistent-mass $P_1$ finite elements in dimensions two and three and for standard Cartesian finite differences in any fixed dimension.
\item The resulting relaxed observability estimates yield approximate null controls with uniformly bounded cost, exponentially small terminal errors, and continuous null controls as weak accumulation points, with no coupling between the spatial and temporal refinement parameters.
\end{enumerate}

\subsection{The transfer mechanism: an elliptic model}\label{sec:elliptic-transfer}

We use the Poisson operator as a model to isolate the two obstructions to a direct perturbation argument and to motivate the construction that overcomes them.

Fix $1<p<\infty$ in the Dirichlet regularity range of the Laplacian. The classical $W^{1,p}$ perturbation argument associated with Meyers' theory~\cite{Meyers1963} allows $\operatorname{div}(\epsilon\nabla v)$ to be absorbed whenever $C(p)\|\epsilon\|_{L^\infty}$ is sufficiently small. In the Carleman setting, however, the analogous threshold depends on the large parameter: the argument in \cite[Section~6.1, (6.5)]{Zuazua2007} requires $Cs\|\epsilon\|_{L^\infty}^2<1$. Thus a growing range of $s$ requires a correspondingly smaller perturbation.

Consistency alone, however, does not provide such smallness uniformly over a discrete space. A simple example illustrates this obstruction. For centered differences on $(0,1)$ with $h=1/J$, the continuous and discrete eigenvalues associated with the sine mode $v_j(x)=\sin(j\pi x)$ are~(see \cite[Section~4.2]{Zuazua2005})
\begin{align*}
\lambda_j=(j\pi)^2,\qquad
\lambda_{h,j}=4h^{-2}\sin^2(j\pi h/2),\qquad 1\leq j<J.
\end{align*}
Writing $\epsilon_{h,j}=\lambda_{h,j}/\lambda_j-1$ for the relative error, we have
\begin{align*}
-\Delta v_j=\lambda_{h,j}v_j+\operatorname{div}(\epsilon_{h,j}\nabla v_j).
\end{align*}
For each fixed $j$, $\lim\limits_{J\to\infty}\epsilon_{1/J,j}=0$, whereas the highest mode satisfies $\lim\limits_{J\to\infty}|\epsilon_{1/J,J-1}|=1-4/\pi^2>0$. Hence the preceding absorption argument cannot yield an admissible Carleman range that grows under mesh refinement.

This obstruction motivates a different comparison. Take $\Omega=(0,1)^2$ and the finite-element space $V_h$, operator $A_h$, and mesh assumptions of \Cref{intro:heat}. Let $\rho$ be the spatial weight supplied by \cref{prop:continuousCE} for a nonempty open set $\omega_0\Subset\omega$, and set $\theta=e^{s\rho}$. Given any $u_h\in V_h$, put $f_h=A_hu_h$ and solve
\begin{align}\label{intro:poisson-auxiliary}
(-\Delta+Ks^2)U=f_h+Ks^2u_h\quad\text{in }\Omega,
\qquad U=0\quad\text{on }\partial\Omega.
\end{align}
The constant $K$ is fixed sufficiently large, independently of $h$ and $s$, to provide weighted coercivity. By the definition of $A_h$, the prescribed state $u_h$ is exactly the Galerkin solution of this auxiliary problem.

With $e=U-u_h$, the static case of \cref{lem:spatialEuler}, with $m=1$ and $q\equiv s$, gives
\begin{align}\label{intro:weighted-error}
\|\theta\nabla e\|_{L^2(\Omega)}
+h^{-1}\|\theta e\|_{L^2(\Omega)}
\leq Ch\|\theta(f_h+Ks^2u_h)\|_{L^2(\Omega)},
\end{align}
provided $hs$ is sufficiently small. The constant is uniform over all $u_h\in V_h$, including mesh-scale oscillations. The essential cancellation is
\begin{align}\label{intro:poisson-residual}
-\Delta U=f_h-Ks^2e.
\end{align}
Consequently, applying the continuous Carleman estimate to $U$ introduces the original discrete residual $f_h$ and a term controlled by the comparison error. Moreover, estimate \cref{intro:weighted-error} controls the difference between the continuous and discrete energy and observation terms. The most restrictive state-error contribution is of order $h^4s^5$, giving the sufficient range $s_0\leq s\leq ch^{-4/5}$. For $h\leq1$, choosing $c>0$ sufficiently small after fixing $K$ and the estimate constants absorbs the state error and bounds the residual coefficient by $Cs^{-1}$, yielding
\begin{align}\label{intro:discrete-carleman}
\|\theta\nabla u_h\|_{L^2(\Omega)}^2
+s^2\|\theta u_h\|_{L^2(\Omega)}^2
&\leq C\Bigl(s^{-1}\|\theta A_hu_h\|_{L^2(\Omega)}^2
+\|\theta\nabla u_h\|_{L^2(\omega)}^2
+s^2\|\theta u_h\|_{L^2(\omega)}^2\Bigr),
\end{align}
where $C,c,s_0>0$ are independent of $h,s,u_h$, and $c$ is sufficiently small. The continuous elliptic Carleman estimate follows from \cref{prop:continuousCE} with $q\equiv s$ and a time-independent function, after division by $s$; its observation on $\omega_0$ is bounded by the displayed observation on $\omega$.

The role of the auxiliary problem is therefore to embed each discrete state as the Galerkin solution of a continuous problem to which the Carleman estimate applies. Weighted numerical error estimates then transfer the inequality to the original scheme. The parabolic argument combines the temporal and spatial comparisons in \Cref{sec:time,sec:error} with the absorption argument in \Cref{sec:transfer}.

\subsection{Relation to previous work}\label{intro:comparison}

A first approach to uniform numerical controllability consists in deriving Carleman estimates directly at the discrete level, by developing weighted identities and estimating remainders within the discrete calculus.
Elliptic difference estimates were established in \cite{Boyer2010,Boyer2010a}, followed by semi-discrete parabolic, time-discrete, and fully discrete estimates in \cite{Boyer2014,BH,GCH}.
By contrast, we derive the discrete Carleman estimate from its continuous counterpart through weighted approximation estimates that hold uniformly over the full discrete space.
This separation of the PDE-specific Carleman analysis from the numerical estimates required for each discretization yields a common transfer mechanism for the finite-element and finite-difference realizations considered here.

A second line of work, rather than deriving discrete Carleman estimates directly, addresses the high-frequency obstruction through spectral truncation, filtering, and multilevel constructions (see \cite{Zuazua2006ICM}). In particular, Boyer, Hubert, and Le Rousseau~\cite{BHLR11} developed an abstract framework for spatial Galerkin approximations and temporal discretization based on a uniform partial spectral inequality and the Lebeau--Robbiano strategy. Under their spectral and discretization assumptions,   \cite[Theorem~3.5]{BHLR11} provides uniformly bounded controls that annihilate a low-frequency spectral projection exactly and leave an exponentially small full terminal state. The resulting control conclusions are therefore closely related, but the underlying estimates differ: their argument uses a partial spectral inequality, whereas ours transfers the Carleman inequality itself through weighted numerical approximation properties. In the autonomous backward-Euler setting, their results also allow for independent refinement after reducing the usable spectral range (see \cref{rem:spectral-comparison}).

In contrast to the spectral approach, in our framework the transferred estimate itself determines a sufficient threshold for filtering. In \Cref{sec:filtering}, the remainder in \cref{intro:observability} is absorbed on terminal spaces whose discrete Rayleigh quotient is sufficiently small relative to the available Carleman scale. This restriction can be realized by compatible nested $P_1$ coarse spaces and Cartesian coarse-grid interpolation, without a discrete eigendecomposition. The corresponding controls simultaneously annihilate the retained terminal projection, have cost bounded uniformly in the discretization parameters, and leave an exponentially small full terminal state for the bounded space-time potentials considered here.

Beyond these Carleman-based and spectral approaches, complementary methods draw on other continuous control and stability arguments. Boyer and Hern\'andez-Santamar\'ia~\cite{BoyerHernandezSantamaria2026} adapt the moments method to time-discrete boundary-controlled parabolic systems by constructing and estimating discrete biorthogonal families.
The penalized Hilbert Uniqueness Method (HUM), which turns relaxed observability estimates into numerically computable controls, has also been studied systematically in the discrete setting (see \cite[Sections~4--6]{BHLR11} and \cite{Boyer2013}).

Discrete Carleman estimates have also found applications in numerical inverse problems.
For semi-discrete parabolic equations, they yield stability estimates for inverse source and coefficient problems~\cite{Lecaros2025a}.

Finally, we outline the structure of the paper. The abstract framework and transfer theorem are stated in \Cref{sec:interfaces}. Weighted temporal and spatial comparisons are proved in \Cref{sec:time,sec:error}, and the transfer is completed in \Cref{sec:transfer}. The finite-element and finite-difference realizations are verified in \Cref{sec:realizations}. Relaxed and filtered observability, control, and weak convergence are treated in \Cref{sec:control}. Numerical illustration follows in \Cref{sec:numerics}. The appendices verify the continuous analytic assumptions, the difference interfaces, the consistency statements used in the control limit, and the auxiliary filtering estimates.

\section{Abstract framework and transfer theorem}\label{sec:interfaces}

This section develops the abstract framework that allows us to transfer a continuous Carleman estimate to a fully discrete scheme. We first collect the continuous analytic assumptions, then introduce the general fully discrete scheme together with the reconstruction and transfer operators that link it to the continuous problem, and finally state the numerical assumptions ensuring weighted approximation, stiffness compatibility, and observation compatibility. Under these hypotheses, the main transfer theorem yields a Carleman estimate over the full discrete space while preserving the residual and observation terms of the original scheme. Its proof is given in Section~\ref{sec:transfer}, after the necessary temporal and spatial comparison estimates have been established in Sections~\ref{sec:time} and~\ref{sec:error}.

\subsection{Continuous analytic assumptions and Carleman estimate}

Fix $d\geq1$, let $\Omega\subset\mathbb R^d$ be a bounded Lipschitz domain, set $H=L^2(\Omega)$, and fix an integer $m\geq1$.

Fix $T>0$. In the auxiliary estimates we impose matching time endpoints in order to cancel temporal boundary terms and to solve both the forward and adjoint auxiliary problems. The original control and adjoint systems have prescribed initial or terminal data. The auxiliary estimate is applied to the adjoint only after a time cutoff. Endpoint values of functions in $H^1(0,T;H)$ are understood as traces in $H$.
Set $\mathcal Q=(0,T)\times\Omega$ and, for a measurable subset $E\subset\Omega$, write $\mathcal Q_E=(0,T)\times E$.

Fix $\beta>0$, $q_c\geq1$, and $\rho\in C^\infty(\overline\Omega)$ with $\rho\leq-\rho_*$ on $\overline\Omega$ for some fixed $\rho_*>0$.
Let $q\in C^2([0,T])$ satisfy $q^{(j)}(0)=q^{(j)}(T)$ for $j=0,1,2$, and extend $q$ periodically when needed. Assume
\begin{align}\label{eq:qclass}
1\leq q_c\leq q(t)\leq q_{\max},\qquad |q'|\leq\beta q^2,
\qquad |q''|\leq\beta^2q^3.
\end{align}
One may take $q_{\max}=\max_{t\in[0,T]}q(t)$. In \cref{eq:qclass}, it need only be an upper bound.
Set $\theta(t,x)=e^{q(t)\rho(x)}$.
The regularized weights in \cref{intro:parabolic-discrete-carleman} are $\xi=q/s$ and $\alpha=\xi\rho$, with $s=\max_{[0,T]}q$. Thus $\max_{[0,T]}\xi=1$, $s\xi=q$, and $e^{s\alpha}=\theta$.
Fix $\tau_*>0$. An allowed weight has the form $\sigma=q^\tau\theta$ or $\sigma=q^\tau\theta^{-1}$, where $\tau$ is a fixed real exponent with $|\tau|\leq\tau_*$.

Write $D^jv$ for the collection of all spatial derivatives of $v$ of order $j$, with the Euclidean norm on its components. At each fixed time $t$, define the spatial norms
\begin{align}\label{eq:graphnorms}
\mathcal N_{\sigma,\Omega}(v)^2
&=\sum_{j=0}^m\|q(t)^{m-j}\sigma(t,\cdot)D^jv\|_{L^2(\Omega)}^2,
\qquad v\in H^m(\Omega),\notag\\
\mathcal R_{\sigma,\Omega}(v)
&=\sum_{j=0}^{2m}\|q(t)^{2m-j}\sigma(t,\cdot)D^jv\|_{L^2(\Omega)},
\qquad v\in H^{2m}(\Omega).
\end{align}
For fixed $t$ and weights, these norms are equivalent to the usual $H^m(\Omega)$ and $H^{2m}(\Omega)$ norms, respectively.
For time-dependent functions, define the space-time norms by
\begin{align}
\label{eq:spacetimegraphnorms}
\mathcal N_{\sigma,\mathcal Q}(v)^2
&=\sum_{j=0}^m\|q^{m-j}\sigma D^jv\|_{L^2(\mathcal Q)}^2,
\qquad
\mathcal R_{\sigma,\mathcal Q}(v)
=\sum_{j=0}^{2m}\|q^{2m-j}\sigma D^jv\|_{L^2(\mathcal Q)}.
\end{align}
These are norms on $L^2(0,T;H^m(\Omega))$ and $L^2(0,T;H^{2m}(\Omega))$, respectively.
All these definitions also apply with $\sigma=1$.

We now collect the continuous analytic assumptions used throughout the paper. They are required to hold with uniform positive constants for every $T>0$ and every allowed weight. The symbols $c,C>0$ denote generic constants that may change from line to line.
\begin{enumerate}[label=\textbf{(A\arabic*)},ref=(A\arabic*)]
\item\label{ass:form} \emph{Form and graph regularity.}
The space $V$ is a closed subspace of $H^m(\Omega)$, dense in $H$.
The real symmetric form $a:V\times V\to\mathbb R$ satisfies the boundedness estimate
\begin{align*}
|a(v,w)|\leq C\|v\|_{H^m(\Omega)}\|w\|_{H^m(\Omega)}
\end{align*}
and the coercivity estimate
\begin{align*}
a(v,v)\geq c\|v\|_{H^m(\Omega)}^2.
\end{align*}
Its associated positive self-adjoint operator $A$ has $D(A)\hookrightarrow H^{2m}(\Omega)$ and satisfies
\begin{align}\label{eq:ellreg}
\|v\|_{H^{2m}(\Omega)}\leq C\|Av\|_{L^2(\Omega)},\quad v\in D(A).
\end{align}
In particular, $D(A^{1/2})=V$ and $\|A^{1/2}v\|_{L^2(\Omega)}^2=a(v,v)$.
\item\label{ass:localization} \emph{Weight compatibility and commutator bound.} Multiplication by every allowed weight $\sigma$ and its inverse $\sigma^{-1}$ preserves both $V$ and $D(A)$.
Write $[A,\sigma]u=A(\sigma u)-\sigma Au$. At each fixed time $t$, for every allowed weight $\sigma$ and every $v\in D(A)$,
\begin{align}\label{eq:commutator}
\|[A,\sigma]\sigma^{-1}v\|_{L^2(\Omega)}
\leq C\sum_{j=0}^{2m-1}q^{2m-j}\|D^jv\|_{L^2(\Omega)}.
\end{align}
\item\label{ass:weighted} \emph{Weighted form bounds.} At each fixed time, for every allowed weight $\sigma$ and all $v,w\in V$,
\begin{align}
|a(v,w)|&\leq C\mathcal N_{\sigma,\Omega}(v)\mathcal N_{\sigma^{-1},\Omega}(w),\label{eq:unshiftedformbound}
\end{align}
and
\begin{align}
a(v,\sigma^2v)&\geq c\mathcal N_{\sigma,\Omega}(v)^2-C\|q^m\sigma v\|_{L^2(\Omega)}^2.\label{eq:unshiftedcoercive}
\end{align}
\item\label{ass:carleman} \emph{Continuous Carleman estimate.} Write $P_\epsilon=\epsilon\partial_t+A$, $\epsilon\in\{-1,1\}$, and fix a nonempty open set $\omega_0\Subset\Omega$. Fix $c_0>0$, $c_j\geq0$ for $1\leq j\leq m$, and $o\geq0$. Let $\mathcal D_j$ be fixed spatial differential expressions of order at most $j$, with bounded coefficients.
For each $\epsilon\in\{-1,1\}$ and every $v\in H^1(0,T;H)\cap L^2(0,T;D(A))$ with $v(0)=v(T)$,
\begin{align}\label{eq:continuousabstract}
\mathcal E_q(v)&:=\|q^{c_0/2}\theta v\|_{L^2(\mathcal Q)}^2+
\sum_{j=1}^m\|q^{c_j/2}\theta\mathcal D_jv\|_{L^2(\mathcal Q)}^2
\leq C\bigl(\|\theta P_\epsilon v\|_{L^2(\mathcal Q)}^2+
\|q^{o/2}\theta v\|_{L^2(\mathcal Q_{\omega_0})}^2\bigr).
\end{align}
\end{enumerate}

Unless stated otherwise, all constants depend only on $d,m,\Omega,\omega_0,\rho,\beta,q_c,\tau_*$, the Carleman exponents $c_0,\ldots,c_m,o$, and the fixed bounds in the hypotheses used. Additional parameter dependence is stated explicitly.

\subsection{Discrete scheme and transfer operators}

For each discretization parameter $h$, we fix the following discrete data: a finite-dimensional real Hilbert space $X_h$ with inner product $(\cdot,\cdot)_{X_h}$ and induced norm $\|\cdot\|_{X_h}$, a linear spatial operator $A_h:X_h\to X_h$, a real normed load space $\mathfrak F_h$ with load map $\Gamma_h:\mathfrak F_h\to X_h$, and a Hilbert control space $U_h$ with linear control operator $B_h:U_h\to X_h$.
The adjoint $B_h^*:X_h\to U_h$ is taken with respect to the prescribed unweighted inner products on $X_h$ and $U_h$.

To compare this discrete scheme with the continuous problem, we introduce a linear conforming reconstruction $R_h:X_h\to V$ and define its adjoint $J_h:H\to X_h$ by
\begin{align}\label{eq:massinterface}
(J_hv,z)_{X_h}=(v,R_hz)_H,\qquad v\in H,\quad z\in X_h.
\end{align}
The two compositions
\begin{align}\label{eq:transfer-compositions}
Q_h=R_hJ_h:H\to V,\qquad G_h=J_hR_h:X_h\to X_h
\end{align}
will play a central role in the transfer of weighted estimates: $Q_h$ maps continuous data to the reconstructed discrete space, while $G_h$ encodes the induced discrete mass. We emphasize that neither $Q_h^2=Q_h$ nor $G_h=\mathrm{Id}_{X_h}$ is assumed; the reconstruction need not be orthogonal.
All scheme, reconstruction, load, and control maps are independent of time.

The load and control spaces $\mathfrak F_h$ and $U_h$ need not be finite dimensional. The weighted norms on $X_h$, $\mathfrak F_h$, and $U_h$ are part of the discretization data; their concrete definitions are given in \cref{sec:p1,sec:fd}. At each fixed time, for every $r>0$ and all $z\in X_h$, $f\in\mathfrak F_h$, $w\in U_h$, we require the homogeneity
\begin{align*}
\|z\|_{X_h,r\sigma}&=r\|z\|_{X_h,\sigma},&
\|f\|_{\mathfrak F_h,r\sigma}&=r\|f\|_{\mathfrak F_h,\sigma},&
\|w\|_{U_h,r\sigma}&=r\|w\|_{U_h,\sigma},
\end{align*}
so that the weighted norms scale consistently with the weight parameter.

\subsection{Numerical assumptions}

We write $a\asymp b$ if $cb\leq a\leq Cb$ for positive constants independent of the varying parameters.

We now state the numerical assumptions under which the transfer theorem will be proved. They are the discrete analogues of the continuous analytic assumptions and are required to hold uniformly over the prescribed family of discretizations. Specifically, we assume that there exist a fixed threshold $\delta_{\rm num}>0$ and uniform positive constants such that \cref{ass:interfaces,ass:loads,Papprox,Ptest,ass:observation} hold for every discretization in the prescribed family, every $q$ satisfying \cref{eq:qclass} with $hq_{\max}\leq\delta_{\rm num}$, every associated allowed weight, and every $t\in[0,T]$.

\begin{enumerate}[label=\textbf{(N\arabic*)},ref=(N\arabic*)]
\item\label{ass:interfaces} \emph{Reconstruction and stiffness compatibility.} The reconstruction $R_h$ is injective, and $A_h$ is self-adjoint and nonnegative in the inner product. For all $z,w\in X_h$,
\begin{align}\label{eq:reconstruction-norms}
\|R_hz\|_{L^2(\Omega)}&\asymp\|z\|_{X_h},\qquad
\|\sigma R_hz\|_{L^2(\Omega)}\asymp\|z\|_{X_h,\sigma}.
\end{align}
The discrete stiffness satisfies
\begin{align}\label{eq:stiffnessinterface}
(A_hz,w)_{X_h}=a(R_hz,R_hw)+\mathfrak b_h(z,w),
\end{align}
where $\mathfrak b_h$ is a time-independent symmetric form satisfying, for some $\eta_A(h)\geq0$ and every allowed weight,
\begin{align}\label{eq:stiffnessdefect}
|\mathfrak b_h(z,w)|\leq\eta_A(h)
\mathcal N_{\sigma,\Omega}(R_hz)\mathcal N_{\sigma^{-1},\Omega}(R_hw).
\end{align}
Exact stiffness compatibility is the case $\eta_A=0$.
\item\label{ass:loads} \emph{Stable load maps.} There is a lift $\mathcal L_h:\mathfrak F_h\to H$ with $J_h\mathcal L_h=\Gamma_h$ satisfying
\begin{align}\label{eq:loadinterface}
\|\sigma\mathcal L_hf\|_{L^2(\Omega)}\leq C\|f\|_{\mathfrak F_h,\sigma}.
\end{align}
The load map has a right inverse $\mathcal I_h:X_h\to\mathfrak F_h$, with $\Gamma_h\mathcal I_h=\mathrm{Id}_{X_h}$, satisfying
\begin{align}\label{eq:stateload}
\|\mathcal I_hz\|_{\mathfrak F_h,\sigma}\leq C\|z\|_{X_h,\sigma}.
\end{align}
\item\label{Papprox} \emph{Weighted approximation.} For prescribed $a_E(h),a_0(h)\geq0$, every $v\in D(A)$ satisfies
\begin{align}\label{eq:approximation}
\mathcal N_{\sigma,\Omega}(v-Q_hv)\leq a_E(h)\mathcal R_{\sigma,\Omega}(v).
\end{align}
The weighted $L^2$ approximation satisfies
\begin{align}\label{eq:approximation-l2}
\|\sigma(v-Q_hv)\|_{L^2(\Omega)}\leq a_0(h)\mathcal R_{\sigma,\Omega}(v).
\end{align}
\item\label{Ptest} \emph{Weighted test and superapproximation.} There is a differentiable operator family $\mathscr{T}_{\sigma,h}(t):X_h\to X_h$, self-adjoint and positive definite in the inner product, with $\mathscr{T}_{\sigma,h}(0)=\mathscr{T}_{\sigma,h}(T)$, such that
\begin{align}
(z,\mathscr{T}_{\sigma,h}z)_{X_h}&\asymp\|\sigma R_hz\|_{L^2(\Omega)}^2.\label{eq:test-norm}
\end{align}
Its temporal derivative satisfies
\begin{align}
|(z,\mathscr{T}_{\sigma,h}'z)_{X_h}|&\leq C\|q\sigma R_hz\|_{L^2(\Omega)}^2,\label{eq:testtime}
\end{align}
and the spatial superapproximation estimate is
\begin{align}
\mathcal N_{\sigma^{-1},\Omega}(R_h\mathscr{T}_{\sigma,h}z-\sigma^2R_hz)
&\leq Chq_{\max}\mathcal N_{\sigma,\Omega}(R_hz).\label{eq:testsuper}
\end{align}
\item\label{ass:observation} \emph{Observation compatibility.} For every $z\in X_h$,
\begin{align}\label{eq:observationinterface}
\|\sigma R_hz\|_{L^2(\omega_0)}\leq C\|B_h^*z\|_{U_h,\sigma}.
\end{align}
\end{enumerate}

\Cref{ass:form,ass:localization,ass:weighted} provide the regularity and weighted stability needed for the shifted continuous auxiliary problem, while \cref{ass:carleman} supplies the continuous Carleman estimate. On the numerical side, \cref{ass:interfaces,ass:loads} ensure that reconstructed states and lifted loads are compatible with the discrete scheme, whereas \cref{Papprox,Ptest} provide the weighted approximation and test estimates required for the spatial comparison. Finally, \cref{ass:observation} controls the continuous observation of the reconstruction by the discrete observation. These ingredients are combined in the temporal and spatial comparison estimates of \cref{sec:time,sec:error}, and the transfer is completed in \cref{sec:transfer}.

\subsection{Fully discrete transfer theorem}

We now discretize the time interval and introduce the discrete norms and difference operators needed to state the transfer theorem. Fix an integer $N\geq2$ and discretize $[0,T]$ with uniform time step ${\delta t}=T/N$ and nodes $t_n=n{\delta t}$, $0\leq n\leq N$. Put $q_n=q(t_n)$, $\theta_n=\theta(t_n,\cdot)$, and $\sigma_n=\sigma(t_n,\cdot)$. Superscripts on state and load sequences denote time levels; the subscript $h$ denotes spatial discretization. Products of weights and sequences are evaluated nodewise.

For a normed space $\mathcal H$ and a specified finite set $\mathcal J$ of time indices, let $\ell_{\delta t}^2(\mathcal H)=\mathcal H^{\mathcal J}$ with norm
\begin{align*}
\|v\|_{\ell_{\delta t}^2(\mathcal H)}^2
:={\delta t}\sum_{n\in\mathcal J}\|v^n\|_{\mathcal H}^2.
\end{align*}
If $\mathcal H$ is a Hilbert space, the corresponding inner product is
\begin{align*}
(v,w)_{\ell_{\delta t}^2(\mathcal H)}
:={\delta t}\sum_{n\in\mathcal J}(v^n,w^n)_{\mathcal H}.
\end{align*}
For prescribed weighted norms, set
\begin{align*}
\|v\|_{\ell_{\delta t}^2(\mathcal H),\sigma}^2
:={\delta t}\sum_{n\in\mathcal J}\|v^n\|_{\mathcal H,\sigma_n}^2.
\end{align*}

To match the temporal endpoint condition in \cref{ass:carleman}, take $\mathcal J=\{0,\ldots,N-1\}$ and extend sequences by $v^{n+N}=v^n$ for $n\in\mathbb Z$, with indices interpreted modulo $N$. Define the oriented difference quotients
\begin{align*}
D_+v^n:=\frac{v^n-v^{n-1}}{{\delta t}},\qquad
D_-v^n:=\frac{v^n-v^{n+1}}{{\delta t}}.
\end{align*}
The notation $D_\pm$ denotes either choice, corresponding to $\pm\partial_t$ in the continuous equations. Periodic summation by parts gives $D_+^*=D_-$ with respect to the $\ell_{\delta t}^2(\mathcal H)$ inner product.

Next, we introduce discrete counterparts of the continuous graph norms. For $v\in\ell_{\delta t}^2(H^m(\Omega))$ and $v\in\ell_{\delta t}^2(H^{2m}(\Omega))$, respectively, define
\begin{align*}
\mathcal N_{\sigma,{\delta t}}(v)^2
&=\sum_{j=0}^m\|q^{m-j}\sigma D^jv\|_{\ell_{\delta t}^2(L^2(\Omega))}^2,\qquad
\mathcal R_{\sigma,{\delta t}}(v)
=\sum_{j=0}^{2m}\|q^{2m-j}\sigma D^jv\|_{\ell_{\delta t}^2(L^2(\Omega))}.
\end{align*}
These definitions also apply with $\sigma=1$.

For $z\in\ell_{\delta t}^2(X_h)$, we define the discrete Carleman energy
\begin{align}\label{eq:discreteenergy}
\mathcal{E}_{q,h,{\delta t}}(z)&:=\|q^{c_0/2}z\|_{\ell_{\delta t}^2(X_h),\theta}^2+
\sum_{j=1}^m\|q^{c_j/2}\theta\mathcal D_jR_hz\|_{\ell_{\delta t}^2(L^2(\Omega))}^2.
\end{align}
Using the Carleman exponents in \cref{ass:carleman}, set
\begin{align}
d_*&=\max\{0,\max_{1\leq j\leq m}(c_j-2(m-j))\},\qquad
b_*=\max\{c_0,o,4m\}.\label{eq:transfer-exponents}
\end{align}
These exponents determine the weight powers that enter the smallness conditions below.

We are now ready to state the main transfer theorem.

\begin{theorem}\label{thm:main}
Assume
\cref{ass:form,ass:localization,ass:weighted,ass:carleman} and \cref{ass:interfaces,ass:loads,Papprox,Ptest,ass:observation},
and
\begin{equation}\label{eq:weightcoverage}
\tau_*\ge\max\{d_*/2,b_*/2\}.
\end{equation}
Then there exist positive constants
$\delta_x,\delta_{\mathrm{time}},\eta_*,c_*,C$, with
$\delta_x\le\delta_{\rm num}$, such that the following holds.

Put
\begin{equation}
\varepsilon_E=a_E(h)+\eta_A(h)+\sqrt{\delta t},
\qquad
\varepsilon_0=a_0(h)+a_E(h)^2+\eta_A(h)+\delta t.
\label{eq:error-scales}
\end{equation}
Assume that
\[
a_E\le1,\qquad
\eta_A\le\eta_*,\qquad
\delta t\le1,\qquad
hq_{\max}\le\delta_x,\qquad
\delta t q_{\max}^2\le\delta_{\mathrm{time}},
\]
and that the defect quantities satisfy
\begin{align}
\Lambda_f
&:=\varepsilon_E^2q_{\max}^{d_*}
+\varepsilon_0^2q_{\max}^{b_*}\le1,
\label{eq:transfer-defects}\\
\Lambda_z
&:=\varepsilon_E^2\max\{1,q_{\max}^{4m+d_*-c_0}\}
+\varepsilon_0^2\max\{1,q_{\max}^{4m+b_*-c_0}\}\le c_*.
\label{eq:transfer-state-defect}
\end{align}
Then, for either choice of sign, every pair
$(z,f)\in\ell_{\delta t}^2(X_h)\times\ell_{\delta t}^2(\mathfrak F_h)$
satisfying
\begin{equation}
D_\pm z^n+A_hz^n=\Gamma_hf^n,
\qquad n=0,\ldots,N-1,
\label{eq:state-load-equation}
\end{equation}
obeys
\begin{equation}\label{eq:finalCE}
\mathcal{E}_{q,h,\delta t}(z)
\le C\bigl(
\|f\|_{\ell_{\delta t}^2(\mathfrak F_h),\theta}^2
+\|q^{o/2}B_h^*z\|_{\ell_{\delta t}^2(U_h),\theta}^2
\bigr).
\end{equation}
In particular, every $z\in\ell_{\delta t}^2(X_h)$ satisfies
\begin{equation}\label{eq:rawCE}
\mathcal{E}_{q,h,\delta t}(z)
\le C\bigl(
\|D_\pm z+A_hz\|_{\ell_{\delta t}^2(X_h),\theta}^2
+\|q^{o/2}B_h^*z\|_{\ell_{\delta t}^2(U_h),\theta}^2
\bigr).
\end{equation}
\end{theorem}

The condition $\Lambda_f\le1$ is a normalization. Any fixed upper bound changes only the constant in \cref{eq:finalCE}. Only the state defect $\Lambda_z$ must be small.

\begin{remark}
The finite upper bound in \cref{eq:qclass} permits $hq_{\max}\ll1$ and ${\delta t}q_{\max}^2\ll1$. The admissible bound grows under refinement. Thus we use a regularized temporal factor rather than an endpoint-singular one. \Cref{sec:continuous} adapts the existing continuous weighted identities to functions with matching time endpoints, so that temporal boundary terms cancel. In \cref{sec:control}, the regularization is chosen at the admissible mesh scale and a cutoff gives physical-time observability; temporal endpoint terms arise in that application.
\end{remark}

For spatial approximation orders $h^m$ and $h^{2m}$, the next corollary gives an explicit sufficient bound on $q_{\max}$ in terms of $h,{\delta t}$ that ensures the smallness conditions in \cref{thm:main}.

\begin{corollary}\label{cor:general-scales}
Under the analytic and numerical assumptions of \cref{thm:main}, suppose $a_E(h)\leq Ch^m$ and $a_0(h)+\eta_A(h)\leq Ch^{2m}$ uniformly in $h$. Set
\begin{align}\label{eq:budgetexponents}
\nu_E=\max\{d_*,4m+d_*-c_0,0\},\qquad
\nu_0=\max\{b_*,4m+b_*-c_0,0\}.
\end{align}
Then there exist $\kappa,h_0,\Delta_0,C>0$, with $\Delta_0$ an upper bound on the time step, such that \cref{eq:finalCE,eq:rawCE} hold for $h\leq h_0$, ${\delta t}\leq \Delta_0$, and
\begin{align}\label{eq:scalerule}
q_{\max}\leq\kappa\min\{h^{-\alpha_x},{(\delta t)}^{-\alpha_t}\},
\end{align}
where
$\alpha_x=\min\{1,2m/\nu_E,4m/\nu_0\}$,
$\alpha_t=\min\{1/2,1/\nu_E,2/\nu_0\}$,
and ratios with zero denominator are omitted from the minima.
\end{corollary}

Applying this general scale rule to the Carleman estimate for the heat equation gives the following parameter range.

\begin{corollary}\label{cor:scales}
Under the same analytic, numerical, and approximation assumptions as \cref{cor:general-scales}, take $m=1$ and $(c_0,c_1;o)=(3,1;3)$.
Then there exist $\kappa,h_0,\Delta_0,C>0$ such that \cref{eq:finalCE,eq:rawCE} hold whenever $h\leq h_0$, ${\delta t}\leq \Delta_0$, and $q_{\max}\leq S_{h,{\delta t}}$, where
\begin{align}\label{eq:mainS}
S_{h,{\delta t}}=\kappa\min\{h^{-4/5},{(\delta t)}^{-2/5}\}.
\end{align}
\end{corollary}

\Cref{thm:main,cor:general-scales,cor:scales} are proved in \Cref{sec:transfer}. The two realizations in \Cref{sec:realizations} use this Carleman estimate for the heat equation and allow $\tau_*=4$. For other orders and continuous Carleman estimates, the range is determined by \cref{eq:scalerule} under the analytic and numerical assumptions, with the weight powers covered by \cref{eq:weightcoverage}.

\section{Weighted temporal estimates for the backward Euler scheme}\label{sec:time}

Throughout this section, we assume \cref{ass:form,ass:localization,ass:weighted}. We adapt the standard energy and parabolic duality argument for temporal discretization to the Carleman weights $\sigma=q^\tau\theta^{\pm1}$. The auxiliary graph estimates are required to hold uniformly with respect to these weights and their reciprocals. The adjoint problem then yields a weighted $L^2(\mathcal Q)$ error of order $\delta t$, whereas spatial interpolation yields a weighted energy error of order $\sqrt{\delta t}$. The time grid and oriented differences are those introduced in \Cref{sec:interfaces}.

\subsection{Regularity of the continuous auxiliary problem}

We begin by introducing the shifted auxiliary problem that will be used throughout the temporal and spatial comparisons. For $L\geq1$, define the positive scalar function $\mu(t):=(Lq(t))^{2m}$. At each fixed time, the associated bilinear form on $V$ is
\begin{align}
a_\mu(v,w):=a(v,w)+\mu(t)\,(v,w)_H,\qquad v,w\in V.\label{eq:shiftedform}
\end{align}
The same notation is used for the nodal and cellwise constant values of the shift.

The following lemma provides the weighted regularity and coercivity for this shifted problem. These estimates are uniform in the weight and will later allow us to absorb the temporal and spatial comparison errors.
\begin{lemma}\label{lem:graph}
There exists $C>0$ such that, for every prescribed margin $B>0$, one can choose $L_0(B)\geq1$ so that, for every $L\geq L_0(B)$ and $g\in L^2(0,T;H)$, the problem $(\epsilon\partial_t+A+\mu)U=g$, $U(0)=U(T)$, has a unique solution in $H^1(0,T;H)\cap L^2(0,T;D(A))$ satisfying
\begin{align}\label{eq:primitivegraph}
\|\sigma U_t\|_{L^2(\mathcal Q)}+\mathcal R_{\sigma,\mathcal Q}(U)\leq C\|\sigma g\|_{L^2(\mathcal Q)}.
\end{align}
There is also $C(L)>0$ such that, for all spatial slices $v,w\in V$, the shifted form satisfies
\begin{align}
|a_\mu(v,w)|&\leq C(L)\mathcal N_{\sigma,\Omega}(v)\mathcal N_{\sigma^{-1},\Omega}(w).\label{eq:formbound}
\end{align}
The prescribed coercivity margin is retained in
\begin{align}
a_\mu(v,\sigma^2v)&\geq c\mathcal N_{\sigma,\Omega}(v)^2+B\|q\sigma v\|_{L^2(\Omega)}^2.\label{eq:formcoercive}
\end{align}
\end{lemma}
\begin{proof}
{\bf Step 1.} In this step,
we prove a graph estimate for the unweighted shifted operator.

Sobolev extension on the bounded Lipschitz domain $\Omega$ and interpolation, followed by \cref{eq:ellreg}, give, uniformly for $r\geq1$,
\begin{align}\label{eq:parametergraph}
\sum_{j=0}^{2m}r^{2m-j}\|D^jv\|_{L^2(\Omega)}
\leq C\bigl(\|Av\|_{L^2(\Omega)}+r^{2m}\|v\|_{L^2(\Omega)}\bigr),
\qquad v\in D(A).
\end{align}

For a smooth periodic graph-domain function, time independence and self-adjointness of $A$ give
\begin{align}\label{eq:shiftidentity}
\|(\epsilon\partial_t+A+\mu)v\|_{L^2(\mathcal Q)}^2
=
\|v_t\|_{L^2(\mathcal Q)}^2+\|Av\|_{L^2(\mathcal Q)}^2+\|\mu v\|_{L^2(\mathcal Q)}^2
+2\int_0^T\mu a(v,v)\,dt
-\epsilon\int_0^T\mu'\|v\|_{L^2(\Omega)}^2\,dt.
\end{align}
By \cref{eq:qclass}, $$|\mu'|/\mu^2\leq2m\beta L^{-2m}q_c^{1-2m},$$ so the last term in \cref{eq:shiftidentity} is absorbed for large $L$. Applying \cref{eq:parametergraph} with $r=Lq(t)$ then gives
\begin{align}\label{eq:parameterparabolic}
\|v_t\|_{L^2(\mathcal Q)}+\sum_{j=0}^{2m}\|(Lq)^{2m-j}D^jv\|_{L^2(\mathcal Q)}
\leq C\|(\epsilon\partial_t+A+\mu)v\|_{L^2(\mathcal Q)}.
\end{align}
Existence follows by a standard spectral Galerkin argument, and uniqueness follows from \cref{eq:parameterparabolic}.

\medskip

{\bf Step 2.}
In this step, we transfer the graph estimate to the weighted setting. The weight definition and \cref{eq:qclass} give
\begin{align}\label{eq:allowed-weight-derivatives}
|D^j\sigma^{\pm1}|\leq Cq^j\sigma^{\pm1}\quad(0\leq j\leq2m),
\qquad |\partial_t\log\sigma|\leq Cq^2.
\end{align}
\Cref{ass:localization} ensures that multiplication by $\sigma^{\pm1}$ preserves $D(A)$. Thus $v=\sigma U$ belongs to the graph space and satisfies
\begin{align}\label{9.26-eq2}
(\epsilon\partial_t+A+\mu)v
=\sigma g+[A,\sigma]\sigma^{-1}v
+\epsilon(\partial_t\log\sigma)v.
\end{align}
Before fixing $L$, by $q\geq1$ and $m\geq1$, \cref{eq:commutator,eq:allowed-weight-derivatives} give
\begin{align}\label{9.26-eq3}
\|[A,\sigma]\sigma^{-1}v\|_{L^2(\mathcal Q)}
\leq\frac{C}{L}\sum_{j=0}^{2m}\|(Lq)^{2m-j}D^jv\|_{L^2(\mathcal Q)}.
\end{align}
The temporal weight term satisfies
\begin{align}\label{9.26-eq4}
\|(\partial_t\log\sigma)v\|_{L^2(\mathcal Q)}
\leq\frac{C}{L^{2m}}\|\mu v\|_{L^2(\mathcal Q)}.
\end{align}
Combining \cref{9.26-eq2,9.26-eq3,9.26-eq4}, increasing $L$ absorbs both terms in \cref{eq:parameterparabolic} and gives
\begin{align}\label{eq:conjugated-graph-estimate}
\|v_t\|_{L^2(\mathcal Q)}
+\sum_{j=0}^{2m}\|(Lq)^{2m-j}D^jv\|_{L^2(\mathcal Q)}
\leq C\|\sigma g\|_{L^2(\mathcal Q)}.
\end{align}
Since $U=\sigma^{-1}v$, the product rule and \cref{eq:spacetimegraphnorms,eq:parameterparabolic,eq:allowed-weight-derivatives,eq:conjugated-graph-estimate} yield
\begin{align*}
\|\sigma U_t\|_{L^2(\mathcal Q)}+\mathcal R_{\sigma,\mathcal Q}(U)
&\leq C\Bigl(\|v_t\|_{L^2(\mathcal Q)}+\|q^2v\|_{L^2(\mathcal Q)}
+\sum_{j=0}^{2m}\sum_{\ell=0}^j\|q^{2m-\ell}D^\ell v\|_{L^2(\mathcal Q)}\Bigr)\\
&\leq C\Bigl(\|v_t\|_{L^2(\mathcal Q)}
+\sum_{j=0}^{2m}\|(Lq)^{2m-j}D^jv\|_{L^2(\mathcal Q)}\Bigr)
\leq C\|\sigma g\|_{L^2(\mathcal Q)},
\end{align*}
so \cref{eq:primitivegraph} follows.

\medskip

{\bf Step 3.}
In this step, we verify the coercivity estimate.

By \cref{eq:unshiftedformbound,eq:unshiftedcoercive,eq:shiftedform} and $\mu=L^{2m}q^{2m}$,
\begin{align*}
|a_\mu(v,w)|\leq(C+L^{2m})\mathcal N_{\sigma,\Omega}(v)\mathcal N_{\sigma^{-1},\Omega}(w).
\end{align*}
The corresponding lower bound is
\begin{align*}
a_\mu(v,\sigma^2v)
\geq c\mathcal N_{\sigma,\Omega}(v)^2+
(L^{2m}-C)\|q^m\sigma v\|_{L^2(\Omega)}^2.
\end{align*}
Choosing $L^{2m}-C\geq B$ and using $q^m\geq q$, we  obtain \cref{eq:formcoercive}.
\end{proof}

\subsection{Backward Euler regularity and temporal error estimates}

We now establish the weighted regularity of the backward Euler scheme and derive the temporal error estimates that will later be combined with the spatial comparison. Put $\mu_n=(Lq_n)^{2m}$.
Define the cellwise extension by\vspace{-2mm}
\begin{align}\label{eq:cellwise-extension}
\bar v(t)=\begin{cases}
v^n, & t\in(t_{n-1},t_n],\quad\text{for }D_+,\\
v^n, & t\in[t_n,t_{n+1}),\quad\text{for }D_-.
\end{cases}
\end{align}
Write $\widehat v$ for the continuous piecewise affine interpolant of the nodal sequence.
Then $\partial_t\widehat v=\overline{D_+v}$ for $D_+$ and $-\partial_t\widehat v=\overline{D_-v}$ for $D_-$, almost everywhere.
The extensions $\bar\mu,\bar g$ use the same time cells as $\bar v$.

We first control the variation of the weights and the shift on each time cell. These bounds will be used repeatedly in the sequel.
\begin{lemma}\label{lem:cell}
There exist $\delta_{\mathrm{time}},C>0$ such that, if ${\delta t}\le1$ and ${\delta t}q_{\max}^2\le\delta_{\mathrm{time}}$, the following hold for every $L\geq1$, $n$, and $x\in\Omega$.
\begin{enumerate}[label=(\roman*)]
\item For $t\in[t_{n-1},t_n]$ and $i\in\{n-1,n\}$,
$
q(t)\asymp q_i$ and $\sigma(t,x)\asymp\sigma_i(x)$.
\item  For $t\in(t_{n-1},t_n]$ with $D_+$, or $t\in[t_n,t_{n+1})$ with $D_-$,
\begin{align}\label{eq:cellbounds}
\biggl|\frac{q(t)}{q_n}-1\biggr|\le C{\delta t}q_n,
\qquad \biggl|\frac{\sigma(t,x)}{\sigma_n(x)}-1\biggr|\le C{\delta t}q_n^2.
\end{align}
On the same cells, the shift satisfies
\begin{align}\label{eq:cell-shift-bound}
|\mu(t)-\bar\mu(t)|\le CL^{2m}{\delta t}q(t)^{2m+1}.
\end{align}
\item For $D_+$, set\vspace{-3mm}
\begin{align*}
H_n(x)=\frac{\sigma_n(x)/\sigma_{n-1}(x)-1}{{\delta t}}.
\end{align*}
Then\vspace{-2mm}
\begin{align}
\|H_n\|_{L^\infty(\Omega)}
&\leq Cq_n^2.\label{eq:weight-time-quotient}
\end{align}
Moreover,\vspace{-2mm}
\begin{align}
|\mu_{n+1}-\mu_n|&\leq CL^{2m}{\delta t}q_n^{2m+1}.\label{eq:shift-time-increment}
\end{align}
\end{enumerate}
\end{lemma}
\begin{proof}
Let $t\in[t_{n-1},t_n]$ and $i\in\{n-1,n\}$. Choose $\delta_{\mathrm{time}}$ so that $\beta\delta_{\mathrm{time}}\leq1/2$. By \cref{eq:qclass},\vspace{-2mm}
\begin{align}\label{9.26-eq1}
\bigg|\frac1{q(t)}-\frac1{q_i}\bigg|
&=\bigg|\int_{t_i}^t\frac{q'(s)}{q(s)^2}\,ds\bigg|\leq\beta {\delta t}.
\end{align}
Multiplying \eqref{9.26-eq1} by $q_i$ gives\vspace{-2mm}
\begin{align*}
\bigg|\frac{q_i}{q(t)}-1\bigg|\leq\beta {\delta t}q_i\leq\frac{1}{2}.
\end{align*}
Hence $q(t)\asymp q_i$ and $|q(t)/q_i-1|\leq C{\delta t}q_i$.
Using \cref{eq:allowed-weight-derivatives}, we have
\begin{align}\label{eq:cell-log-weight}
\bigg|\log\frac{\sigma(t,x)}{\sigma_i(x)}\bigg|
&=\bigg|\int_{t_i}^t\partial_s\log\sigma(s,x)\,ds\bigg|\leq C{\delta t}q_i^2,
\end{align}
and\vspace{-2mm}
\begin{align}\label{eq:cell-weight-ratio}
\Big|\frac{\sigma(t,x)}{\sigma_i(x)}-1\Big|
&\leq C{\delta t}q_i^2 e^{C{\delta t}q_i^2}\leq C{\delta t}q_i^2.
\end{align}
By \cref{eq:cell-log-weight}, $\sigma(t,x)\asymp\sigma_i(x)$. Taking $t=t_n$ and $i=n-1$ in \cref{eq:cell-weight-ratio} and using $q_{n-1}\asymp q_n$ gives \cref{eq:weight-time-quotient}. By \cref{eq:qclass},
\begin{align}\label{eq:cell-shift-difference}
|\mu(t)-\mu_i|
&=2mL^{2m}\Big|\int_{t_i}^t q(s)^{2m-1}q'(s)\,ds\Big|
\leq CL^{2m}{\delta t}q_i^{2m+1}.
\end{align}
On either time cell associated with $t_n$, using \cref{eq:cell-shift-difference} with $i=n$, $\bar\mu(t)=\mu_n$, and $q_n\asymp q(t)$, we obtain the bound for $\mu$ in \cref{eq:cell-shift-bound}. Taking $t=t_{n+1}$ and $i=n$ in \cref{eq:cell-shift-difference} gives \cref{eq:shift-time-increment}.
\end{proof}

Unweighted step-uniform maximal parabolic regularity for backward Euler is standard (see \cite[Theorem~3.1]{KLL}). We prove the weighted periodic estimate below to keep the dependence on the shift and the weights explicit.
\begin{lemma}\label{lem:eulerMR}
Let $L_0$ be as in \cref{lem:graph}. There exist $L_1\geq L_0(1)$ and $\delta_{\mathrm{time}},C>0$ such that, for every $L\geq L_1$, ${\delta t}\leq1$, ${\delta t}q_{\max}^2\leq\delta_{\mathrm{time}}$, and $g\in\ell_{\delta t}^2(L^2(\Omega))$, the periodic problem
\begin{align}\label{eq:spacecontinuousEuler}
D_\pm u^n+(A+\mu_n)u^n=g^n
\end{align}
has a unique solution $u\in\ell_{\delta t}^2(D(A))$ satisfying
\begin{align}\label{eq:eulerMR}
\|\sigma D_\pm u\|_{\ell_{\delta t}^2(L^2(\Omega))}+\mathcal{R}_{\sigma,{\delta t}}(u)
\le C\|\sigma g\|_{\ell_{\delta t}^2(L^2(\Omega))}
\end{align}
for both orientations and all allowed weights.
\end{lemma}
\begin{proof}
We prove the lemma for $D_+$. The case $D_-$ follows by the time reflection $n\mapsto -n$, with the same constants.

{\bf Step 1.} In this step, we obtain the unweighted graph estimate.

Let $v\in\ell_{\delta t}^2(D(A))$ be $N$-periodic. Expanding the squared norm and applying periodic summation by parts give
\begin{align}\label{eq:euleridentity}
\|D_+v+(A+\mu_n)v^n\|_{\ell_{\delta t}^2(L^2(\Omega))}^2
&=\|D_+v\|_{\ell_{\delta t}^2(L^2(\Omega))}^2+\|(A+\mu_n)v^n\|_{\ell_{\delta t}^2(L^2(\Omega))}^2
+{\delta t}\|A^{1/2}D_+v\|_{\ell_{\delta t}^2(L^2(\Omega))}^2 \notag\\
&\quad + {\delta t}\|\sqrt{\mu_n}D_+v\|_{\ell_{\delta t}^2(L^2(\Omega))}^2
-\sum_{n=0}^{N-1}(\mu_{n+1}-\mu_n)\|v^n\|_{L^2(\Omega)}^2.
\end{align}
By \cref{eq:shift-time-increment}, $q_n\geq1$, and positivity of $A$, we have
\begin{align}\label{9.26-eq5}
\Bigl|\sum_{n=0}^{N-1}(\mu_{n+1}-\mu_n)\|v^n\|_{L^2(\Omega)}^2\Bigr|
&\leq CL^{2m}{\delta t}\sum_{n=0}^{N-1} q_n^{2m+1}\|v^n\|_{L^2(\Omega)}^2
\leq CL^{-2m}\|\mu_nv^n\|_{\ell_{\delta t}^2(L^2(\Omega))}^2,
\end{align}
and
\begin{align}\label{9.26-eq6}
\|(A+\mu_n)v^n\|_{L^2(\Omega)}^2
&\geq\|Av^n\|_{L^2(\Omega)}^2+\|\mu_nv^n\|_{L^2(\Omega)}^2.
\end{align}
For large $L$,  combining \cref{eq:euleridentity,9.26-eq5,9.26-eq6}, and applying \cref{eq:parametergraph} at each node yield
\begin{align}\label{eq:euler-unweighted-graph}
\|D_+v\|_{\ell_{\delta t}^2(L^2(\Omega))}+\sum_{j=0}^{2m}\|(Lq)^{2m-j}D^jv\|_{\ell_{\delta t}^2(L^2(\Omega))}
\leq C\|D_+v+(A+\mu_n)v^n\|_{\ell_{\delta t}^2(L^2(\Omega))}.
\end{align}

\medskip

{\bf Step 2.} In this step, we study the conjugation of the recurrence.

By \cref{ass:localization}, $v^n=\sigma_nu^n\in D(A)$. Conjugating \cref{eq:spacecontinuousEuler} gives
\begin{align}\label{9.26-eq7}
D_+v^n+(A+\mu_n)v^n
=\sigma_ng^n+[A,\sigma_n]\sigma_n^{-1}v^n
+H_nv^{n-1}.
\end{align}
By \cref{eq:commutator,eq:weight-time-quotient,eq:cellbounds} and periodic reindexing,
it holds that
\begin{align}\label{9.26-eq8}
\|[A,\sigma_n]\sigma_n^{-1}v^n\|_{\ell_{\delta t}^2(L^2(\Omega))}
&\leq\frac{C}{L}\sum_{j=0}^{2m}\|(Lq)^{2m-j}D^jv\|_{\ell_{\delta t}^2(L^2(\Omega))},
\end{align}
and
\begin{align}\label{9.26-eq9}
\|H_nv^{n-1}\|_{\ell_{\delta t}^2(L^2(\Omega))}
&\leq C\|q_n^2v^{n-1}\|_{\ell_{\delta t}^2(L^2(\Omega))}
=C\|q_{n+1}^2v^n\|_{\ell_{\delta t}^2(L^2(\Omega))}
\notag
\\
&\leq C\|q_n^2v^n\|_{\ell_{\delta t}^2(L^2(\Omega))}
\leq CL^{-2m}\|\mu_nv^n\|_{\ell_{\delta t}^2(L^2(\Omega))}.
\end{align}
By \cref{9.26-eq7,9.26-eq8,9.26-eq9}, choosing $L$ sufficiently large allows us to absorb $[A,\sigma_n]\sigma_n^{-1}v^n$ and $H_nv^{n-1}$ into the left-hand side of \cref{eq:euler-unweighted-graph}. The identity
\begin{align*}
\sigma_nD_+u^n=D_+v^n-H_nv^{n-1},
\end{align*}
together with the product rule and \cref{eq:allowed-weight-derivatives}, then yields \cref{eq:eulerMR}.

Existence and uniqueness follow by expansion in the eigenbasis of $A$: each scalar recurrence has factors $1+{\delta t}(\lambda+\mu_n)>1$ and a unique periodic solution. Estimate \cref{eq:euler-unweighted-graph} permits passage from finite spectral sums to arbitrary $g\in\ell_{\delta t}^2(L^2(\Omega))$.
\end{proof}

The continuous weighted graph estimate also holds when the shift is frozen on each time cell. This will be used to compare the discrete and continuous auxiliary problems.
\begin{lemma}\label{lem:frozen}
Let $L_0$ be as in \cref{lem:graph}.
There exists $C>0$ such that, for every $L\geq L_0(1)$, one can choose $\delta_{\mathrm{time}}(L)>0$ for which the following holds whenever ${\delta t}\leq1$ and ${\delta t}q_{\max}^2\leq\delta_{\mathrm{time}}(L)$. For either sign of $\epsilon$ and either cellwise extension $\bar\mu$ of $\mu_n$, the periodic problem
\begin{align}\label{eq:frozencontinuous}
(\epsilon\partial_t+A+\bar\mu)U=G
\end{align}
has a unique solution in $H^1(0,T;H)\cap L^2(0,T;D(A))$ with $U(0)=U(T)$ for every $G\in L^2(0,T;H)$, and
\begin{align}\label{eq:frozengraph}
\|\sigma U_t\|_{L^2(\mathcal Q)}+\mathcal{R}_{\sigma,\mathcal Q}(U)\le C\|\sigma G\|_{L^2(\mathcal Q)}.
\end{align}
\end{lemma}
\begin{proof}
Let $\mathscr L_\mu U:=(\epsilon\partial_t+A+\mu)U$ on the subspace of $H^1(0,T;H)\cap L^2(0,T;D(A))$ with $U(0)=U(T)$. By \cref{lem:graph}, $\mathscr L_\mu$ is invertible. Denote the constant in \cref{eq:primitivegraph} by $C_0$. From \cref{eq:cell-shift-bound,eq:spacetimegraphnorms},
\begin{align}\label{eq:frozen-shift-perturbation}
\|\sigma(\bar\mu-\mu)U\|_{L^2(\mathcal Q)}
&\leq C(L){\delta t}q_{\max}\|q^{2m}\sigma U\|_{L^2(\mathcal Q)}
\leq C(L){\delta t}q_{\max}\mathcal R_{\sigma,\mathcal Q}(U).
\end{align}
Choose $\delta_{\mathrm{time}}(L)$ within the range of \cref{lem:cell} so that $C_0C(L)\delta_{\mathrm{time}}(L)\leq1/2$, with $C(L)$ from \cref{eq:frozen-shift-perturbation}. Since $q_{\max}\geq1$, we have ${\delta t}q_{\max}\leq {\delta t}q_{\max}^2\leq\delta_{\mathrm{time}}(L)$. For every $F\in L^2(\mathcal Q)$, applying \cref{eq:primitivegraph,eq:frozen-shift-perturbation} to $\mathscr L_\mu^{-1}F$ gives
\begin{align}\label{eq:frozen-contraction}
\|\sigma(\bar\mu-\mu)\mathscr L_\mu^{-1}F\|_{L^2(\mathcal Q)}
\leq C_0C(L){\delta t}q_{\max}\|\sigma F\|_{L^2(\mathcal Q)}
\leq\dfrac12\|\sigma F\|_{L^2(\mathcal Q)}.
\end{align}
By \cref{eq:frozen-contraction}, the Neumann series in $L^2(\mathcal Q)$, equipped with the norm $\|\sigma\,\cdot\|_{L^2(\mathcal Q)}$, yields a unique solution of
\begin{align*}
F+(\bar\mu-\mu)\mathscr L_\mu^{-1}F=G.
\end{align*}
The solution satisfies
\begin{align*}
\|\sigma F\|_{L^2(\mathcal Q)}\leq2\|\sigma G\|_{L^2(\mathcal Q)}.
\end{align*}
Then $U=\mathscr L_\mu^{-1}F$ is the unique solution of \cref{eq:frozencontinuous}, and \cref{eq:primitivegraph} gives
\begin{align*}
\|\sigma U_t\|_{L^2(\mathcal Q)}+\mathcal R_{\sigma,\mathcal Q}(U)
\leq C_0\|\sigma F\|_{L^2(\mathcal Q)}
\leq2C_0\|\sigma G\|_{L^2(\mathcal Q)}.
\end{align*}
Thus \cref{eq:frozengraph} holds with $C=2C_0$, independent of $L$.
\end{proof}

The estimates \cref{eq:eulerMR,eq:frozengraph} yield the temporal $L^2$ error by duality. Spatial interpolation then gives the energy error.
\begin{theorem}\label{thm:timeerror}
Under the hypotheses and temporal restrictions of \cref{lem:eulerMR,lem:frozen}, there exists $C>0$ such that, for every $g\in\ell_{\delta t}^2(L^2(\Omega))$ and every allowed weight, the solutions $u$ of \cref{eq:spacecontinuousEuler} and $U$ of \cref{eq:frozencontinuous} with $G=\bar g$ satisfy
\begin{align}\label{eq:temporal-errors}
\|\sigma(U-\bar u)\|_{L^2(\mathcal Q)}\leq C{\delta t}\|\sigma g\|_{\ell_{\delta t}^2(L^2(\Omega))}.
\end{align}
Their weighted energy error satisfies
\begin{align}\label{eq:temporal-energy-error}
\mathcal{N}_{\sigma,\mathcal Q}(U-\bar u)\leq C\sqrt{{\delta t}}\|\sigma g\|_{\ell_{\delta t}^2(L^2(\Omega))}.
\end{align}
\end{theorem}
\begin{proof}
{\bf Step 1.} In this step, we estimate the $L^2$ error by duality.

Put $\delta=\widehat u-\bar u$, $E=U-\widehat u$, and $e_t=U-\bar u=E+\delta$. The interpolation formula and \cref{eq:cellbounds,eq:eulerMR} give
\begin{align}\label{eq:timedelta}
\|\sigma\delta\|_{L^2(\mathcal Q)}
\leq C{\delta t}\|\sigma D_\pm u\|_{\ell_{\delta t}^2(L^2(\Omega))}
\leq C{\delta t}\|\sigma g\|_{\ell_{\delta t}^2(L^2(\Omega))}.
\end{align}
Subtracting \cref{eq:spacecontinuousEuler,eq:frozencontinuous} yields
\begin{align}
\epsilon\partial_tE+(A+\bar\mu)e_t=0
\qquad\text{in }L^2(0,T;H).\label{eq:temporal-error-equation}
\end{align}
Let $Y$ solve the periodic adjoint problem
\begin{align}
(-\epsilon\partial_t+A+\bar\mu)Y=\sigma^2e_t.\label{eq:temporal-dual-equation}
\end{align}
Applying \cref{eq:frozengraph} with $\sigma^{-1}$ gives
$\|\sigma^{-1}Y_t\|_{L^2(\mathcal Q)}\leq C\|\sigma e_t\|_{L^2(\mathcal Q)}$.
Both $E$ and $Y$ belong to $H^1(0,T;H)\cap L^2(0,T;D(A))$ and satisfy $E(0)=E(T)$ and $Y(0)=Y(T)$. Using \cref{eq:temporal-error-equation,eq:temporal-dual-equation}, periodic integration by parts gives
\begin{align}\label{eq:timecancellation}
\|\sigma e_t\|_{L^2(\mathcal Q)}^2
=\int_0^T\bigl((\epsilon\partial_tE,Y)_H+a_{\bar\mu}(e_t,Y)\bigr)\,dt
-\epsilon\int_0^T(\delta,Y_t)_H\,dt
=-\epsilon\int_0^T(\delta,Y_t)_H\,dt.
\end{align}
Cauchy's inequality in \cref{eq:timecancellation}, together with \cref{eq:timedelta,eq:frozengraph}, gives
\begin{align}\label{eq:temporal-l2-error}
\|\sigma e_t\|_{L^2(\mathcal Q)}
\leq C\|\sigma\delta\|_{L^2(\mathcal Q)}
\leq C{\delta t}\|\sigma g\|_{\ell_{\delta t}^2(L^2(\Omega))}.
\end{align}

{\bf Step 2.} In this step, we  obtain the energy error by spatial interpolation.

The graph estimates \cref{eq:frozengraph,eq:eulerMR} and the cell bounds \cref{eq:cellbounds} give
\begin{align}
\mathcal R_{\sigma,\mathcal Q}(e_t)
\leq\mathcal R_{\sigma,\mathcal Q}(U)+\mathcal R_{\sigma,\mathcal Q}(\bar u)
\leq C\|\sigma g\|_{\ell_{\delta t}^2(L^2(\Omega))}.\label{eq:temporal-graph-error}
\end{align}
For $v\in L^2(0,T;H^{2m}(\Omega))$, fix time and put $w=\sigma v$. Sobolev interpolation and Young's inequality give, for $0\leq j\leq m$,
\begin{align}\label{eq:spatial-parameter-interpolation}
q^{2(m-j)}\|D^jw\|_{L^2(\Omega)}^2
\leq C\|w\|_{L^2(\Omega)}
\bigl(\|w\|_{H^{2m}(\Omega)}+q^{2m}\|w\|_{L^2(\Omega)}\bigr).
\end{align}
Since $q\geq1$, the product rule and \cref{eq:allowed-weight-derivatives,eq:graphnorms} convert the sum of \cref{eq:spatial-parameter-interpolation} over $j$ to
\begin{align*}
\mathcal N_{\sigma,\Omega}(v)^2\leq C\|\sigma v\|_{L^2(\Omega)}\mathcal R_{\sigma,\Omega}(v).
\end{align*}
Integration in time, Cauchy's inequality, and \cref{eq:spacetimegraphnorms} yield
\begin{align}\label{eq:weighted-interpolation}
\mathcal N_{\sigma,\mathcal Q}(v)^2
\leq C\|\sigma v\|_{L^2(\mathcal Q)}\mathcal R_{\sigma,\mathcal Q}(v).
\end{align}
Applying \cref{eq:weighted-interpolation} to $e_t$ and using \cref{eq:temporal-l2-error,eq:temporal-graph-error} gives
$\mathcal N_{\sigma,\mathcal Q}(e_t)^2\leq C{\delta t}\|\sigma g\|_{\ell_{\delta t}^2(L^2(\Omega))}^2$.
\end{proof}

\begin{remark}
The constants in \cref{eq:eulerMR,eq:temporal-errors,eq:temporal-energy-error} are uniform in ${\delta t}\lambda$, where $\lambda$ ranges over the eigenvalues of $A$. Thus no spectral CFL condition ${\delta t}\lambda\leq C$ is required.
\end{remark}

\section{Weighted error estimates for spatial discretization}\label{sec:error}

In this section we compare the spatially continuous and fully discrete auxiliary problems on the same Euler grid. The temporal estimate in \cref{thm:timeerror} will then yield the full space--time error.

Throughout this section we assume \cref{ass:form,ass:localization,ass:weighted}, \cref{ass:interfaces,Papprox,Ptest}, and a lift $\mathcal L_h$ satisfying $J_h\mathcal L_h=\Gamma_h$ and \cref{eq:loadinterface}. We choose the margin $B$ in \cref{eq:formcoercive} sufficiently large in terms of the constant in \cref{eq:testtime} and the fixed weight bounds, and then fix $L$ as in \cref{lem:graph,lem:eulerMR}.

Under these assumptions, there exist $\eta_*,\delta_x,\delta_{\mathrm{time}},C>0$, with $\delta_x\leq\delta_{\rm num}$, such that the estimates below hold whenever\vspace{-2mm}
\begin{align*}
a_E\leq1,\qquad \eta_A\leq\eta_*,\qquad {\delta t}\leq1,\qquad
hq_{\max}\leq\delta_x,\qquad {\delta t}q_{\max}^2\leq\delta_{\mathrm{time}}.
\end{align*}
The temporal threshold is chosen within the ranges of \cref{lem:eulerMR,lem:frozen}.

For $g\in\ell_{\delta t}^2(H)$, let $u$ solve \cref{eq:spacecontinuousEuler}, and let $u_h\in\ell_{\delta t}^2(X_h)$ solve the periodic auxiliary problem
\begin{align}\label{eq:auxEuler}
D_\pm u_h^n+A_hu_h^n+\mu_nG_hu_h^n=J_hg^n.
\end{align}
Here the time difference is taken with respect to the inner product of $X_h$, while $G_h=J_hR_h$ represents the reconstruction mass in the shift.
It is readily seen that \cref{eq:auxEuler} has a unique periodic solution.

We now compare these two solutions. The following lemma combines weighted C\'ea-type energy and Aubin--Nitsche $L^2$ estimates, allowing nonorthogonal transfers and stiffness defects.
\begin{lemma}\label{lem:spatialEuler}
Under the assumptions of this section, let $u$ and $u_h$ solve \cref{eq:spacecontinuousEuler,eq:auxEuler} for $g\in\ell_{\delta t}^2(H)$. For every allowed weight $\sigma$, $e_s=u-R_hu_h$ satisfies
\begin{align}\label{eq:spatialEulerError}
\mathcal N_{\sigma,{\delta t}}(e_s)
\leq C(a_E+\eta_A)\|\sigma g\|_{\ell_{\delta t}^2(L^2(\Omega))}.
\end{align}
The weighted $L^2$ error satisfies\vspace{-2mm}
\begin{align}\label{eq:spatialEulerL2Error}
\|\sigma e_s\|_{\ell_{\delta t}^2(L^2(\Omega))}
\leq C(a_0+a_E^2+\eta_A)\|\sigma g\|_{\ell_{\delta t}^2(L^2(\Omega))}.
\end{align}
Recall that $a_E,a_0$ are defined in \cref{eq:approximation,eq:approximation-l2} and $\eta_A$ in \cref{eq:stiffnessdefect}.
\end{lemma}
\begin{proof}
{\bf Step 1.} In this step, we estimate the weighted energy error.

We only prove the estimate for $D_+$. Time reflection gives the case $D_-$. Put $\eta^n=u^n-Q_hu^n$ and $e_h^n=u_h^n-J_hu^n$, so that $e_s=\eta-R_he_h$. Subtracting \cref{eq:spacecontinuousEuler,eq:auxEuler} and using \cref{eq:massinterface,eq:stiffnessinterface}, we obtain, for every $v\in X_h$,\vspace{-2mm}
\begin{align}\label{eq:thetaEuler}
(D_+e_h^n,v)_{X_h}+a_{\mu_n}(R_he_h^n,R_hv)+\mathfrak b_h(e_h^n,v)
=a_{\mu_n}(\eta^n,R_hv)-\mathfrak b_h(J_hu^n,v).
\end{align}

Set $\mathscr{T}_n=\mathscr{T}_{\sigma,h}(t_n)$ and extend this sequence $N$-periodically, using the endpoint compatibility in \cref{Ptest}.
We test \cref{eq:thetaEuler} with $\mathscr{T}_ne_h^n$ and sum over one period. Self-adjointness and periodic reindexing give\vspace{-2mm}
\begin{align}\label{eq:weighted-time-summation}
{\delta t}\sum_{n=0}^{N-1}(D_+e_h^n,\mathscr{T}_ne_h^n)_{X_h}
&=\frac{1}{2}\sum_{n=0}^{N-1}(e_h^n,(\mathscr{T}_n-\mathscr{T}_{n+1})e_h^n)_{X_h}
+\frac{1}{2}\sum_{n=0}^{N-1}(e_h^n-e_h^{n-1},\mathscr{T}_n(e_h^n-e_h^{n-1}))_{X_h}.
\end{align}
The second sum is nonnegative. By \cref{eq:testtime,eq:cellbounds}, the absolute value of the first sum is bounded by $C\|q\sigma R_he_h\|_{\ell_{\delta t}^2(L^2(\Omega))}^2$ and is absorbed by the margin $B$ in \cref{eq:formcoercive}.

The product rule and \cref{eq:allowed-weight-derivatives,eq:testsuper} imply
\begin{align*}
\mathcal N_{\sigma^{-1},\Omega}(R_h\mathscr{T}_{\sigma,h}z)
\leq C\mathcal N_{\sigma,\Omega}(R_hz),\qquad z\in X_h.
\end{align*}
By \cref{eq:formbound,eq:testsuper,eq:stiffnessdefect},
\begin{align}\label{eq:spatial-test-errors}
&{\delta t}\sum_{n=0}^{N-1}\Bigl(
\bigl|a_{\mu_n}(R_he_h^n,R_h\mathscr{T}_ne_h^n-\sigma_n^2R_he_h^n)\bigr|
+\bigl|\mathfrak b_h(e_h^n,\mathscr{T}_ne_h^n)\bigr|\Bigr)
\leq (C(L)hq_{\max}+C\eta_A)\mathcal N_{\sigma,{\delta t}}(R_he_h)^2.
\end{align}
After fixing $L$, we choose $\delta_x$ and $\eta_*$ sufficiently small. Absorbing the left-hand side of \cref{eq:spatial-test-errors} into the coercive term and estimating the right side of \cref{eq:thetaEuler} by \cref{eq:formbound,eq:stiffnessdefect}, we obtain
\begin{align*}
\mathcal N_{\sigma,{\delta t}}(R_he_h)
\leq C\bigl(\mathcal N_{\sigma,{\delta t}}(\eta)+\eta_A\mathcal N_{\sigma,{\delta t}}(Q_hu)\bigr).
\end{align*}
Now \cref{eq:approximation} gives $\mathcal N_{\sigma,{\delta t}}(\eta)\leq a_E\mathcal R_{\sigma,{\delta t}}(u)$ and $\mathcal N_{\sigma,{\delta t}}(Q_hu)\leq(1+a_E)\mathcal R_{\sigma,{\delta t}}(u)$. Since $a_E\leq1$, \cref{eq:eulerMR} yields
\begin{align}\label{eq:spatial-energy-stability}
\mathcal N_{\sigma,{\delta t}}(R_he_h)
\leq C(a_E+\eta_A)\|\sigma g\|_{\ell_{\delta t}^2(L^2(\Omega))}.
\end{align}
Using $e_s=\eta-R_he_h$ proves the energy estimate in \cref{eq:spatialEulerError}.
The same bounds, together with $R_hu_h=Q_hu+R_he_h$, give the stability estimate
\begin{align}\label{eq:spatial-state-stability}
\mathcal N_{\sigma,{\delta t}}(R_hu_h)
\leq C\|\sigma g\|_{\ell_{\delta t}^2(L^2(\Omega))}.
\end{align}

{\bf Step 2.} In this step, we estimate the $L^2$ error by duality.

Let $Y\in\ell_{\delta t}^2(D(A))$ solve the periodic problem
\begin{align}\label{eq:spatial-dual-equation}
D_-Y^n+(A+\mu_n)Y^n=\sigma_n^2e_s^n.
\end{align}
Applying \cref{eq:eulerMR} with $\sigma^{-1}$ gives\vspace{-2mm}
\begin{align}\label{eq:spatial-dual-graph}
\mathcal R_{\sigma^{-1},{\delta t}}(Y)
\leq C\|\sigma e_s\|_{\ell_{\delta t}^2(L^2(\Omega))}.
\end{align}
Testing \cref{eq:auxEuler} with $J_hY^n$ and using \cref{eq:massinterface,eq:stiffnessinterface}, we obtain
\begin{align*}
(R_hD_+u_h^n,Y^n)_H+a_{\mu_n}(R_hu_h^n,Q_hY^n)
+\mathfrak b_h(u_h^n,J_hY^n)=(g^n,Q_hY^n)_H.
\end{align*}
Since $D_+^*=D_-$ and $R_h$ is time independent, \cref{eq:spacecontinuousEuler,eq:spatial-dual-equation} now give
\begin{align}\label{eq:completeidentity}
\|\sigma e_s\|_{\ell_{\delta t}^2(L^2(\Omega))}^2
&={\delta t}\sum_{n=0}^{N-1}\bigl((D_+e_s^n,Y^n)_H+a_{\mu_n}(e_s^n,Y^n)\bigr)\notag\\
&={\delta t}\sum_{n=0}^{N-1}\bigl((g^n,Y^n-Q_hY^n)_H
-a_{\mu_n}(R_hu_h^n,Y^n-Q_hY^n)\bigr)
+{\delta t}\sum_{n=0}^{N-1}\mathfrak b_h(u_h^n,J_hY^n)\notag\\
&={\delta t}\sum_{n=0}^{N-1}\bigl((D_+u^n,Y^n-Q_hY^n)_H
+a_{\mu_n}(e_s^n,Y^n-Q_hY^n)\bigr)
+{\delta t}\sum_{n=0}^{N-1}\mathfrak b_h(u_h^n,J_hY^n).
\end{align}

We estimate the three terms in the final expression separately. By \cref{eq:approximation-l2,eq:eulerMR},
\begin{align*}
\Bigl|{\delta t}\sum_{n=0}^{N-1}(D_+u^n,Y^n-Q_hY^n)_H\Bigr|
\leq Ca_0\|\sigma g\|_{\ell_{\delta t}^2(L^2(\Omega))}\mathcal R_{\sigma^{-1},{\delta t}}(Y).
\end{align*}
Estimate \cref{eq:spatialEulerError}, together with \cref{eq:approximation,eq:formbound}, gives
\begin{align*}
\Bigl|{\delta t}\sum_{n=0}^{N-1} a_{\mu_n}(e_s^n,Y^n-Q_hY^n)\Bigr|
\leq Ca_E(a_E+\eta_A)\|\sigma g\|_{\ell_{\delta t}^2(L^2(\Omega))}\mathcal R_{\sigma^{-1},{\delta t}}(Y).
\end{align*}
Finally, \cref{eq:stiffnessdefect,eq:spatial-state-stability,eq:approximation} and
$\mathcal N_{\sigma^{-1},{\delta t}}(Q_hY)\leq(1+a_E)\mathcal R_{\sigma^{-1},{\delta t}}(Y)$ imply
\begin{align*}
\Bigl|{\delta t}\sum_{n=0}^{N-1}\mathfrak b_h(u_h^n,J_hY^n)\Bigr|
\leq C\eta_A\|\sigma g\|_{\ell_{\delta t}^2(L^2(\Omega))}\mathcal R_{\sigma^{-1},{\delta t}}(Y).
\end{align*}
Combining these estimates with \cref{eq:completeidentity,eq:spatial-dual-graph} and $a_E\leq1$ proves the $L^2$ estimate in \cref{eq:spatialEulerL2Error}.
\end{proof}

Combining the spatial and temporal estimates, we obtain the full space-time error.
\begin{theorem}\label{thm:combined}
Under the assumptions of this section, let $U$ solve \cref{eq:frozencontinuous} with $G=\bar g$, and let $u_h$ solve \cref{eq:auxEuler}, for $g\in\ell_{\delta t}^2(H)$. For either orientation and every allowed weight, with $\varepsilon_E,\varepsilon_0$ defined in \cref{eq:error-scales},
\begin{align}\label{eq:combinederror}
\mathcal N_{\sigma,\mathcal Q}(U-R_h\bar u_h)
\leq C\varepsilon_E\|\sigma g\|_{\ell_{\delta t}^2(L^2(\Omega))}.
\end{align}
The weighted $L^2$ error satisfies
\begin{align}\label{eq:combined-l2-error}
\|\sigma(U-R_h\bar u_h)\|_{L^2(\mathcal Q)}
\leq C\varepsilon_0\|\sigma g\|_{\ell_{\delta t}^2(L^2(\Omega))}.
\end{align}
\end{theorem}
\begin{proof}
Decompose $U-R_h\bar u_h=(U-\bar u)+\overline{u-R_hu_h}$ and apply \cref{eq:temporal-errors,eq:temporal-energy-error,eq:spatialEulerError,eq:spatialEulerL2Error,eq:cellbounds}.
\end{proof}

\section{Proof of the transfer theorem}\label{sec:transfer}

We choose the source term of the shifted continuous problem so that its fully discrete solution coincides with the prescribed state. We then apply the continuous Carleman estimate and use \cref{thm:combined} to recover the original discrete energy, residual, and observation.
\begin{proof}[Proof of \cref{thm:main}]
{\bf Step 1.} In this step, we choose the auxiliary load.

Given the pair $(z,f)$ in \cref{eq:state-load-equation}, we set\vspace{-2mm}
\begin{align}\label{eq:transfer-load}
g^n=\mathcal L_hf^n+\mu_nR_hz^n,
\end{align}
and let $U$ solve \cref{eq:frozencontinuous} with $G=\bar g$.
Recall the extension defined in \cref{eq:cellwise-extension}.
By \cref{ass:loads,eq:transfer-compositions,eq:state-load-equation},\vspace{-2mm}
\begin{align*}
J_hg^n=\Gamma_hf^n+\mu_nG_hz^n
=D_\pm z^n+A_hz^n+\mu_nG_hz^n.
\end{align*}
Uniqueness in \cref{eq:auxEuler} gives $u_h=z$. Put $e=U-R_h\bar z$.
Using \cref{eq:frozencontinuous} with $G=\bar g$, then \cref{eq:transfer-load,eq:cellwise-extension} and $U=R_h\bar z+e$, we obtain\vspace{-2mm}
\begin{align}\label{eq:transfer-residual}
(\epsilon\partial_t+A)U
=\bar g-\bar\mu U
=\overline{\mathcal L_hf}+\bar\mu R_h\bar z-\bar\mu(R_h\bar z+e)
=\overline{\mathcal L_hf}-\bar\mu e.
\end{align}
This identity is the parabolic counterpart of \cref{intro:poisson-residual}.

{\bf Step 2.} In this step, we compare the continuous and discrete quantities.

We now apply the continuous Carleman estimate \cref{eq:continuousabstract} to $U=R_h\bar z+e$. By \cref{eq:reconstruction-norms,eq:cellbounds,eq:loadinterface,eq:observationinterface,eq:transfer-residual},\vspace{-2mm}
\begin{align}\label{eq:transfer-comparison}
\mathcal E_{q,h,{\delta t}}(z)&\leq C\bigl(\|f\|_{\ell_{\delta t}^2(\mathfrak F_h),\theta}^2
+\|q^{o/2}B_h^*z\|_{\ell_{\delta t}^2(U_h),\theta}^2\bigr)
+C\sum_{j=1}^m\|q^{c_j/2}\theta\mathcal D_je\|_{L^2(\mathcal Q)}^2
+C\|q^{c_0/2}\theta e\|_{L^2(\mathcal Q)}^2
\notag\\
&\quad
+C\|q^{o/2}\theta e\|_{L^2(\mathcal Q_{\omega_0})}^2
+C\|\bar\mu\theta e\|_{L^2(\mathcal Q)}^2.
\end{align}

{\bf Step 3.} In this step, we  estimate the error terms on the right-hand side of \cref{eq:transfer-comparison}.

By the order of $\mathcal D_j$ in \cref{ass:carleman} and \cref{eq:spacetimegraphnorms},
$\|q^{m-j}\sigma\mathcal D_je\|_{L^2(\mathcal Q)}\leq C\mathcal N_{\sigma,\mathcal Q}(e)$.
For each $j$, we apply the energy estimate in \cref{eq:combinederror} with $\sigma=q^{(c_j-2(m-j))/2}\theta$ when the exponent is positive, and with $\sigma=\theta$ otherwise. These weights are covered by \cref{eq:weightcoverage}. Using \cref{eq:transfer-exponents,eq:transfer-load}, we obtain\vspace{-2mm}
\begin{align}\label{eq:derivative-transfer-error}
\sum_{j=1}^m\|q^{c_j/2}\theta\mathcal D_je\|_{L^2(\mathcal Q)}^2
\leq C\varepsilon_E^2\bigl(
\|q^{d_*/2}\theta\mathcal L_hf\|_{\ell_{\delta t}^2(L^2(\Omega))}^2
+\|q^{2m+d_*/2}\theta R_hz\|_{\ell_{\delta t}^2(L^2(\Omega))}^2\bigr).
\end{align}
For the remaining errors, we use the $L^2$ estimate in \cref{eq:combined-l2-error} with the weights $q^{c_0/2}\theta$, $q^{o/2}\theta$, and $q^{2m}\theta$.

By \cref{eq:cellwise-extension,eq:cellbounds}, $\bar\mu(t)=\mu_n=(Lq_n)^{2m}\asymp(Lq(t))^{2m}$ on each cell. Together with $b_*=\max\{c_0,o,4m\}$ in \cref{eq:transfer-exponents}, this gives\vspace{-2mm}
\begin{align}\label{eq:zeroth-order-transfer-error}
\|q^{c_0/2}\theta e\|_{L^2(\mathcal Q)}^2\!
+\!\|q^{o/2}\theta e\|_{L^2(\mathcal Q_{\omega_0})}^2
\!+\!\|\bar\mu\theta e\|_{L^2(\mathcal Q)}^2
&\leq\! C\varepsilon_0^2\bigl(
\|q^{b_*/2}\theta\mathcal L_hf\|_{\ell_{\delta t}^2(L^2(\Omega))}^2\!
+\! \|q^{2m+b_*/2}\theta R_hz\|_{\ell_{\delta t}^2(L^2(\Omega))}^2\bigr).
\end{align}

{\bf Step 4.} In this step, we absorb the state error.

For $\nu\in\{d_*,b_*\}$, \cref{eq:loadinterface,eq:reconstruction-norms,eq:discreteenergy} give\vspace{-2mm}
\begin{align*}
\|q^{\nu/2}\theta\mathcal L_hf\|_{\ell_{\delta t}^2(L^2(\Omega))}^2
\leq Cq_{\max}^{\nu}\|f\|_{\ell_{\delta t}^2(\mathfrak F_h),\theta}^2.
\end{align*}
The state term is bounded by\vspace{-2mm}
\begin{align*}
\|q^{2m+\nu/2}\theta R_hz\|_{\ell_{\delta t}^2(L^2(\Omega))}^2
\leq C\max\{1,q_{\max}^{4m+\nu-c_0}\}\mathcal E_{q,h,{\delta t}}(z).
\end{align*}
Combining \cref{eq:transfer-comparison,eq:derivative-transfer-error,eq:zeroth-order-transfer-error} with \cref{eq:transfer-defects,eq:transfer-state-defect}, we obtain
\begin{align}\label{eq:transfer-absorption}
\mathcal E_{q,h,{\delta t}}(z)\leq C(1+\Lambda_f)\|f\|_{\ell_{\delta t}^2(\mathfrak F_h),\theta}^2
+C\|q^{o/2}B_h^*z\|_{\ell_{\delta t}^2(U_h),\theta}^2
+C\Lambda_z\mathcal E_{q,h,{\delta t}}(z).
\end{align}
With the constant $C$ in \cref{eq:transfer-absorption} fixed, choose $c_*$ such that $Cc_*\leq1/2$. Since $\Lambda_f\leq1$ and $\Lambda_z\leq c_*$, we get \cref{eq:finalCE}. Finally, taking $f=\mathcal I_h(D_\pm z+A_hz)$ and using \cref{eq:stateload} gives \cref{eq:rawCE}.
\end{proof}

\begin{proof}[Proof of \cref{cor:general-scales}]
By \cref{eq:error-scales} and the approximation bounds in \cref{cor:general-scales}, for $h\leq1$, we have
\begin{align*}
\varepsilon_E^2\leq C(h^{2m}+{\delta t}),\qquad
\varepsilon_0^2\leq C(h^{4m}+{|\delta t|}^2).
\end{align*}
Since $q_{\max}\geq1$, \cref{eq:transfer-defects,eq:transfer-state-defect,eq:budgetexponents} then give
\begin{align*}
\Lambda_f+\Lambda_z\leq C\bigl[(h^{2m}+{\delta t})q_{\max}^{\nu_E}
+(h^{4m}+{|\delta t|}^2)q_{\max}^{\nu_0}\bigr].
\end{align*}
For each positive exponent, \cref{eq:scalerule} bounds the spatial terms by $C\kappa^{\nu_E}$ or $C\kappa^{\nu_0}$, and likewise for the temporal terms. All remaining powers of $h$ and ${\delta t}$ are nonnegative. For zero exponents, $\lim\limits_{h,\delta t\to0}(h^{2m}+\delta t)=0$ and $\lim\limits_{h,\delta t\to0}(h^{4m}+(\delta t)^2)=0$.

We choose $\kappa$ so that the positive-exponent terms are below the fixed thresholds for $\Lambda_f$ and $\Lambda_z$, and also $\kappa\leq\delta_x$, $\kappa^2\leq\delta_{\mathrm{time}}$. Then we decrease $h_0,\Delta_0\leq1$ to control any zero-exponent terms and to ensure $a_E\leq1$, $\eta_A\leq\eta_*$.

Since $\alpha_x\leq1$ and $\alpha_t\leq1/2$, the same choice gives $hq_{\max}\leq\delta_x$ and ${\delta t}q_{\max}^2\leq\delta_{\mathrm{time}}$. Thus \cref{thm:main} applies.
\end{proof}
\begin{proof}[Proof of \cref{cor:scales}]
For $m=1$ and $(c_0,c_1;o)=(3,1;3)$, \cref{eq:transfer-exponents,eq:budgetexponents} give
$(d_*,b_*,\nu_E,\nu_0)=(1,4,2,5)$.
Substituting these values into \cref{eq:error-scales,eq:transfer-defects,eq:transfer-state-defect}, we obtain\vspace{-1mm}
\begin{align*}
\Lambda_f\leq C\bigl[(h^2+{\delta t})q_{\max}+(h^4+{|\delta t|}^2)q_{\max}^4\bigr].
\end{align*}
The state defect satisfies\vspace{-3mm}
\begin{align*}
\Lambda_z \leq C\bigl[(h^2+{\delta t})q_{\max}^2+(h^4+{|\delta t|}^2)q_{\max}^5\bigr].
\end{align*}
Finally, \cref{eq:scalerule} gives $(\alpha_x,\alpha_t)=(4/5,2/5)$, which proves \cref{eq:mainS}.
\end{proof}

\section{Finite-element and finite-difference discretizations}\label{sec:realizations}

We now verify \cref{ass:interfaces,ass:loads,Papprox,Ptest,ass:observation} for two spatial discretizations of the heat equation. The first is the consistent-mass $P_1$ finite-element method on locally graded simplicial meshes in dimensions two and three, and the second is the standard Cartesian finite-difference scheme in any fixed dimension. In both cases, the continuous analytic assumptions are supplied by \cref{prop:analytic-realizations,prop:continuousCE}, with $\tau_*=4$, so that the transfer theorem and its corollaries apply once the numerical hypotheses are checked. Throughout this section, weighted estimates are evaluated at a fixed time. For each realization, we fix an open control region $\omega$ with $\omega_0\Subset\omega\Subset\Omega$. Constants may also depend on $\omega$. The finite-element case is treated in \Cref{sec:p1}, and the finite-difference case in \Cref{sec:fd}.

\subsection{Consistent-mass \texorpdfstring{$P_1$}{P1} elements in two and three dimensions}\label{sec:p1}

We now verify the numerical assumptions for the consistent-mass $P_1$ finite-element discretization of the heat equation in two and three dimensions.

Let $d\in\{2,3\}$ and let $\Omega\subset\mathbb R^d$ be a fixed bounded convex polytope.
Take $V=H_0^1(\Omega)$, $a(v,w)=(\nabla v,\nabla w)_{L^2(\Omega)}$, and $A=-\Delta_D$.
Let $\mathcal T_h=\{K\}$ be a conforming, uniformly shape-regular simplicial mesh of $\Omega$, with $h_K=\operatorname{diam}K$ and $h=\max_{K\in\mathcal T_h}h_K$.

Set $\Gamma_d=(\sqrt{d+2}+1)/(\sqrt{d+2}-1)$, so that $\Gamma_2=3$ and $\Gamma_3=(3+\sqrt5)/2$.
Assume there are positive numbers $\ell_K$ satisfying
\begin{align}\label{eq:grading}
ch_K\leq\ell_K\leq Ch_K,\qquad
\gamma^{-1}\leq\ell_K/\ell_{K'}\leq\gamma<\Gamma_d
\quad\text{whenever $K,K'$ share a vertex},
\end{align}
where $c,C>0$ and $1\leq\gamma<\Gamma_d$ are fixed.
Write $\ell|_K=\ell_K$ for the associated elementwise mesh-size function.

Set $X_h=V_h=\{v\in H_0^1(\Omega):v|_K\in P_1\}$, equipped with the $L^2(\Omega)$ inner product, and define $(A_hz,w)_{X_h}=a(z,w)$. The reconstruction $R_h$ is inclusion and $J_h=\Pi_h$, where $\Pi_h$ is the $L^2(\Omega)$ orthogonal projection onto $V_h$. Thus $Q_h=R_h\Pi_h$ is identified with $\Pi_h$, $G_h=\mathrm{Id}_{X_h}$, and $\eta_A=0$.

Take $\mathfrak F_h=H$, $\Gamma_h=J_h$, $\mathcal L_h=\mathrm{Id}_H$, and $\mathcal I_h=R_h$. For $U_h=L^2(\omega)$, define $(B_hu,z)_{X_h}=\int_\omega uz\,dx$, so that $B_h^*z=z|_\omega$.
The weighted norms are
\begin{align*}
\|z\|_{X_h,\sigma}&:=\|\sigma z\|_{L^2(\Omega)},&
\|f\|_{\mathfrak F_h,\sigma}&:=\|\sigma f\|_{L^2(\Omega)},&
\|w\|_{U_h,\sigma}&:=\|\sigma w\|_{L^2(\omega)}.
\end{align*}

The following proposition verifies \cref{ass:interfaces,ass:loads,Papprox,Ptest,ass:observation} for $P_1$ elements under the mesh grading condition.
\begin{proposition}\label{prop:p1}
Under \cref{eq:grading}, there exist $\delta_{\rm num},C>0$, depending additionally on the grading bounds in \cref{eq:grading} and the shape-regularity bound, such that \cref{ass:interfaces,ass:loads,Papprox,Ptest,ass:observation} hold whenever $hq_{\max}\leq\delta_{\rm num}$, with $m=1$, $a_E=Ch$, $a_0=Ch^2$, $\eta_A=0$, and $\mathscr{T}_{\sigma,h}z=\Pi_h(\sigma^2z)$.
\end{proposition}
\begin{proof}
We verify the assumptions in three steps: weighted stability of the $L^2$ projection, weighted approximation estimates, and the weighted test estimates.

{\bf Step 1.} In this step, we establish the weighted stability of the $L^2$ projection.

We use the weighted projection stability in \cite[Theorems~4.4 and~4.14(i), Remark~4.5]{DST} for degree one, $p=2$, and zero boundary values on $\partial\Omega$. The admissible weight-grading threshold is $\Gamma_d$. Consider the four weights\vspace{-3mm}
\begin{align}\label{eq:fourweights}
\sigma,\qquad\sigma^{-1},\qquad\ell^{-1}\sigma,\qquad\ell^{-1}\sigma^{-1}.
\end{align}
For each such weight $\zeta$, \cref{eq:allowed-weight-derivatives,eq:grading} give $\zeta(x)/\zeta(y)\leq\gamma e^{Chq}$ for a.e. $x\in K$, $y\in K'$ whenever $K,K'$ share a vertex. Choose $\delta_{\rm num}\leq1$ so that
$$\gamma e^{Chq_{\max}}\leq(\gamma+\Gamma_d)/2<\Gamma_d,\qquad \mbox{for }hq_{\max}\leq\delta_{\rm num}.$$
The cited stability result then gives, uniformly for every $\zeta$ in \cref{eq:fourweights},
\begin{align}\label{eq:p1-weighted-projection}
\|\zeta\Pi_hv\|_{L^2(\Omega)}\leq C\|\zeta v\|_{L^2(\Omega)}.
\end{align}

{\bf Step 2.} In this step, we derive the weighted approximation estimates.

For each $K\in\mathcal T_h$, choose $x_K\in K$ and set $\sigma_K=\sigma(x_K)$. By \cref{eq:allowed-weight-derivatives},
\begin{align}\label{eq:element-weight-comparison}
e^{-Chq}\leq\frac{\sigma(x)}{\sigma_K}\leq e^{Chq},\qquad x\in K.
\end{align}
Thus $hq\leq\delta_{\rm num}$ makes $\sigma$ uniformly comparable to $\sigma_K$ on $K$. The same holds for $\sigma^{-1}$ and, by convexity of $\Omega$, on patches of diameter $O(h)$.
Let $v\in D(A)$ and let $I_h$ be the nodal interpolant, which is well defined since $H^2(\Omega)\hookrightarrow C^0(\overline\Omega)$ for $d=2,3$. The standard local interpolation estimate in \cite[Theorem~4.4.4]{Brenner2008}, together with \cref{eq:element-weight-comparison} and $\ell_K\asymp h_K$, gives
\begin{align}\label{eq:p1-interpolation}
\|\sigma(v-I_hv)\|_{L^2(\Omega)}
&\leq C\|\ell^2\sigma D^2v\|_{L^2(\Omega)},
\end{align}
and
\begin{align}\label{eq:p1-interpolation-energy}
\|\sigma\nabla(v-I_hv)\|_{L^2(\Omega)}
+\|\ell^{-1}\sigma(v-I_hv)\|_{L^2(\Omega)}
&\leq C\|\ell\sigma D^2v\|_{L^2(\Omega)}.
\end{align}
The local inverse estimate~\cite[Lemma~4.5.3]{Brenner2008}, with the same weight comparison, yields
\begin{align}\label{eq:p1-weighted-inverse}
\|\sigma\nabla z_h\|_{L^2(\Omega)}
\leq C\|\ell^{-1}\sigma z_h\|_{L^2(\Omega)},\qquad z_h\in V_h.
\end{align}
Using $v-\Pi_hv=(v-I_hv)-\Pi_h(v-I_hv)$, \cref{eq:p1-interpolation,eq:p1-interpolation-energy,eq:p1-weighted-projection,eq:p1-weighted-inverse}, and $hq\leq1$, we obtain
\begin{align*}
\|\sigma(v-\Pi_hv)\|_{L^2(\Omega)}
\leq Ch^2\mathcal R_{\sigma,\Omega}(v).
\end{align*}
The corresponding energy estimate is
\begin{align*}
\mathcal N_{\sigma,\Omega}(v-\Pi_hv)
\leq Ch\mathcal R_{\sigma,\Omega}(v).
\end{align*}

{\bf Step 3.}
In this step, we construct the test operator and verify the superapproximation estimate.

For $z_h\in V_h$, set $w=\sigma^2z_h-I_h(\sigma^2z_h)$. We use the elementwise superapproximation argument in the proof of \cite[Theorem~2.1]{Demlow2010}, with the nodal interpolant. By \cref{eq:element-weight-comparison,eq:allowed-weight-derivatives}, the normalized weight $\sigma/\sigma_K$ is uniformly comparable to one and satisfies $|D^j(\sigma/\sigma_K)|\leq Cq^j$ on $K$ for $j=1,2$. The cited argument therefore applies with length scale $q^{-1}\geq h_K$. For the $L^2$ bounds, since $D^2z_h=0$  on each simplex, we also use \cref{eq:p1-interpolation} elementwise and $$|D^2(\sigma^2z_h)|\leq C\sigma^2(q|\nabla z_h|+q^2|z_h|).$$  Together these give
\begin{align}\label{9.26-eq10}
\|\sigma^{-1}\nabla w\|_{L^2(\Omega)}
+\|\ell^{-1}\sigma^{-1}w\|_{L^2(\Omega)}
\leq Chq_{\max}\mathcal N_{\sigma,\Omega}(z_h).
\end{align}
The zeroth-order term also satisfies\vspace{-2mm}
\begin{align}\label{9.26-eq11}
q\|\sigma^{-1}w\|_{L^2(\Omega)}
\leq Ch^2q_{\max}^2\mathcal N_{\sigma,\Omega}(z_h).
\end{align}
Since $I_h(\sigma^2z_h)\in V_h$, we have\vspace{-2mm}
\begin{align}\label{9.26-eq12}
\Pi_h(\sigma^2z_h)-\sigma^2z_h=\Pi_hw-w.
\end{align}
Applying \cref{eq:p1-weighted-projection} with $\zeta=\sigma^{-1},\ell^{-1}\sigma^{-1}$ and \cref{eq:p1-weighted-inverse} with $\sigma^{-1}$, combining \cref{9.26-eq10,9.26-eq11,9.26-eq12}, we obtain \cref{eq:testsuper}.

The test family is $T$-periodic because $\sigma$ is. Orthogonality gives the norm identity\vspace{-2mm}
\begin{align}\label{9.26-eq13}
(z_h,\mathscr{T}_{\sigma,h}z_h)_{X_h}=\|\sigma z_h\|_{L^2(\Omega)}^2
\end{align}
and the temporal derivative identity\vspace{-2mm}
\begin{align}\label{9.26-eq14}
(z_h,\mathscr{T}_{\sigma,h}'z_h)_{X_h}=\int_\Omega(\sigma^2)'z_h^2\,dx.
\end{align}
From \cref{eq:allowed-weight-derivatives,9.26-eq13,9.26-eq14}, we get \cref{eq:test-norm,eq:testtime}.

Finally, the definitions give the exact stiffness identity
\begin{align*}
(A_hz,w)_{X_h}=a(R_hz,R_hw),\qquad z,w\in X_h.
\end{align*}
The load maps satisfy $J_h\mathcal L_h=\Gamma_h$ and $\Gamma_h\mathcal I_h=\mathrm{Id}_{X_h}$.
Since $R_h$ is inclusion and $A_h$ is self-adjoint and nonnegative, these verify \cref{ass:interfaces,ass:loads} with $\eta_A=0$. All norm comparisons in \cref{eq:reconstruction-norms,eq:loadinterface,eq:stateload} are equalities. Moreover, $B_h^*z=z|_\omega$ and $\omega_0\subset\omega$ give
\begin{align*}
\|\sigma R_hz\|_{L^2(\omega_0)}
\leq\|\sigma z\|_{L^2(\omega)}
=\|B_h^*z\|_{U_h,\sigma},
\end{align*}
which proves \cref{ass:observation}.
\end{proof}

From \cref{prop:analytic-realizations,prop:continuousCE}, the preceding verification allows us to apply \cref{thm:main,cor:scales}, giving the heat range in \cref{eq:mainS}.

\subsection{Standard Cartesian finite-difference discretization}\label{sec:fd}

We next verify the numerical assumptions for the standard Cartesian finite-difference discretization of the heat equation in any fixed dimension.

Fix $d\geq1$ and let $\Omega=\prod\limits_{r=1}^d(0,L_r)$, $V=H_0^1(\Omega)$, and $a(v,w)=(\nabla v,\nabla w)_{L^2(\Omega)}$. For integers $N_r\geq2$, define the uniform Cartesian grid by
\begin{align*}
h_r=L_r/N_r,\qquad x_{r,j}=jh_r,\quad 0\leq j\leq N_r,\qquad h=\max_r h_r.
\end{align*}
Assume the steps $h_r$ are mutually comparable. The largest reconstruction-element diameter is then comparable to $h$.

For interior multi-indices $i\in\prod\limits_{r=1}^d\{1,\ldots,N_r-1\}$, impose zero boundary values and set $\omega_i=\prod\limits_{r=1}^d h_r$. The diagonal physical inner product is\vspace{-4mm}
\begin{align*}
(z,w)_{X_h}=\sum_i\omega_i z_iw_i.
\end{align*}
The standard $(2d+1)$-point Laplacian is\vspace{-2mm}
\begin{align}\label{eq:physicalFD}
(A_hz)_i=\sum_{r=1}^d\frac{2z_i-z_{i+e_r}-z_{i-e_r}}{h_r^2},
\end{align}
where $e_r$ is the $r$th coordinate unit vector.

To connect this scheme with the continuous problem, we subdivide each box into $d!$ simplices by the Kuhn triangulation, using the same orientation in every box.
For the interior nodal basis functions $\phi_i$, set $R_hz=\sum\limits_i z_i\phi_i$ and $(J_hv)_i=\omega_i^{-1}(v,\phi_i)_H$.
Thus $Q_h=R_hJ_h$ and $(G_hz)_i=\omega_i^{-1}\sum\limits_j(\phi_j,\phi_i)_Hz_j$.

Take $\mathfrak F_h=X_h$ and $\Gamma_h=\mathcal I_h=\mathrm{Id}_{X_h}$. The control space consists of nodal values at nodes in $\omega$, with the restricted $X_h$ inner product. $B_h$ is zero extension and $B_h^*$ is restriction.
The weighted norms are
\begin{align*}
\|z\|_{X_h,\sigma}^2:=\sum_i\omega_i\sigma(x_i)^2z_i^2,\qquad
\|f\|_{\mathfrak F_h,\sigma}:=\|f\|_{X_h,\sigma},\qquad
\|w\|_{U_h,\sigma}^2:=\sum_{i:x_i\in\omega}\omega_i\sigma(x_i)^2w_i^2.
\end{align*}

The following proposition shows that the Cartesian difference scheme satisfies the same numerical assumptions as the finite-element scheme, through its piecewise affine reconstruction.
\begin{proposition}\label{prop:fd}
There exist $\delta_{\rm num},C>0$, depending additionally on the comparability bound for $h_r$, such that \cref{eq:physicalFD} satisfies \cref{ass:interfaces,ass:loads,Papprox,Ptest,ass:observation} whenever $hq_{\max}\leq\delta_{\rm num}$, with $m=1$, $a_E=Ch$, $a_0=Ch^2$, $\eta_A=0$, and the test operator\vspace{-2mm}
\begin{align*}
(\mathscr{T}_{\sigma,h}z)_i=\sigma(x_i)^2z_i.
\end{align*}
\end{proposition}
\begin{proof}
We first prove the exact stiffness identity\vspace{-2mm}
\begin{align}\label{eq:fdstiffness}
(A_hz,w)_{X_h}=a(R_hz,R_hw).
\end{align}
A Kuhn simplex associated with a permutation $\pi$ has vertices $v_0,\ldots,v_d$ with $v_j-v_{j-1}=h_{\pi_j}e_{\pi_j}$. Since $R_hz$ is affine on the simplex,\vspace{-2mm}
\begin{align*}
\partial_{\pi_j}R_hz=\frac{z(v_j)-z(v_{j-1})}{h_{\pi_j}}.
\end{align*}
Each coordinate edge $[x_i,x_{i+e_r}]$ with interior transverse indices belongs to $d!$ such simplices, one for each permutation. Each simplex has volume $\prod_s h_s/d!$, so the total contribution of this edge to $a(R_hz,R_hw)$ is
\begin{align*}
\biggl(\prod_{s=1}^d h_s\biggr)
\frac{z_{i+e_r}-z_i}{h_r}\,
\frac{w_{i+e_r}-w_i}{h_r}.
\end{align*}
Edges contained in the boundary contribute zero. Summation by parts in \cref{eq:physicalFD} gives the same contributions to $(A_hz,w)_{X_h}$, proving \cref{eq:fdstiffness} and hence $\mathfrak b_h=0$. The mass, load, approximation, test, and observation estimates are proved in \cref{sec:fd-verification}.
\end{proof}

By \cref{prop:analytic-realizations,prop:continuousCE,cor:scales}, the estimate \cref{eq:rawCE} holds for $h\leq h_0$, ${\delta t}\leq \Delta_0$, and $q_{\max}\leq S_{h,{\delta t}}$, with the constants in \cref{cor:scales}.

\begin{remark}\label{rem:general-fd}
The result extends to the standard conservative $(2d+1)$-point scheme on smoothly stretched tensor grids $x_{r,j}=F_r(j/N_r)$, where the $F_r:[0,1]\to[0,L_r]$ are fixed $C^2$ diffeomorphisms with $F_r'\geq c>0$ and the $N_r^{-1}$ are mutually comparable. We use the same Kuhn reconstruction and diagonal mass weights given by products of averaged adjacent steps.

With $h$ denoting the largest step, adjacent steps differ by $O(h^2)$. The local stiffness and nodal-moment estimates give the weighted bounds $a_E=O(h)$ and $a_0+\eta_A=O(h^2)$. Thus \cref{cor:scales} gives the same heat parameter range, with constants depending additionally on the fixed $C^2$ bounds, the lower bound for $F_r'$, and the step comparability bounds.
\end{remark}

\section{Observability and control for the fully discrete heat system}\label{sec:control}

We now turn from the Carleman estimate to its  consequences in controllability theory. Fix one of the heat discretizations in \cref{sec:realizations} and a control time $T>0$. Throughout this section, $A=-\Delta_D$ and $V=H_0^1(\Omega)$. Let $N\geq2$, ${\delta t}=T/N$, and $t_n=n{\delta t}$. The control space and operator are those of the chosen realization. In each case, $\|B_h\|_{\mathcal L(U_h,X_h)}\leq1$.
Here $C(T,M)$ denotes a generic positive constant that may also depend on the fixed domain, control region, and mesh-regularity and grading bounds, but not on $h$, $\delta t$, or the initial datum.

When the finite element method is used to construct numerical solutions, the coefficient $b^n$ is represented at the discrete level by a bounded potential. To fix the setting, let $b^n\in L^\infty(\Omega;\mathbb R)$ with $\|b^n\|_{L^\infty(\Omega)}\le M$, and define the potential operator $\mathsf V_h^n$ by
\begin{align*}
(\mathsf V_h^ny,z)_{X_h}=\int_\Omega b^n(R_hy)(R_hz)\,dx.
\end{align*}
For the finite difference discretization, the continuous coefficient is replaced by nodal values of the potential. Specifically, $\mathsf V_h^n$ is defined as multiplication by real nodal values $b_i^n$ with $|b_i^n|\le M$.
The averaging rules used for a fixed continuous coefficient are specified in \Cref{sec:control-convergence}.

For $y_{0,h}\in X_h$, consider the discrete control system
\begin{align}\label{eq:forward}
\begin{cases}
\dfrac{y^n-y^{n-1}}{\delta t}+(A_h+\mathsf V_h^n)y^n=B_hu^n,
&1\leq n\leq N,\\
y^0=y_{0,h}.
\end{cases}
\end{align}
For $p_T\in X_h$, the corresponding adjoint system is
\begin{align}\label{eq:adjoint}
\begin{cases}
\dfrac{p^{n-1}-p^n}{\delta t}+(A_h+\mathsf V_h^n)^*p^{n-1}=0,
&1\leq n\leq N,\\
p^N=p_T.
\end{cases}
\end{align}
All adjoints and terminal norms use the inner product of $X_h$. For controls, $\ell_{\delta t}^2(U_h)$ uses the index set $\mathcal J=\{1,\ldots,N\}$. The time mass is the consistent $L^2$ mass for finite elements and the original diagonal mass for differences.

We first record the energy bounds and the exact duality identity for the forward and adjoint schemes.
\begin{lemma}\label{lem:physicalduality}
Assume that $\delta t M\le1/2$. Then the systems \cref{eq:forward,eq:adjoint} are well posed, and there exists a constant $C(T,M)>0$ with the following properties:
\begin{enumerate}
\item Every solution of \cref{eq:adjoint} satisfies the backward stability bound
\begin{align}\label{eq:adjointenergy}
\|p^j\|_{X_h}\le e^{2M(t_\ell-t_j)}\|p^\ell\|_{X_h},
\qquad0\le j\le\ell\le N.
\end{align}
\item The solution  to \eqref{eq:forward}  satisfies the a priori bound
\begin{align}\label{eq:forward-energy-bound}
\max_{0\le n\le N}\|y^n\|_{X_h}
\le C(T,M)\bigl(\|y_{0,h}\|_{X_h}+\|u\|_{\ell_{\delta t}^2(U_h)}\bigr).
\end{align}
\item For solutions to \eqref{eq:forward} and \eqref{eq:adjoint}, we have  the duality identity
\begin{align}\label{eq:trueD}
(y^N,p^N)_{X_h}-(y^0,p^0)_{X_h}
={\delta t}\sum_{n=1}^N(u^n,B_h^*p^{n-1})_{U_h}.
\end{align}
\end{enumerate}
\end{lemma}

\begin{proof}
For completeness, we recall the standard backward-Euler stability and discrete duality argument.

Since $\mathsf V_h^n$ is self-adjoint with norm at most $M$, the operator
$\mathrm{Id}_{X_h}+\delta t(A_h+\mathsf V_h^n)$ is bounded below by
$1-\delta t M>0$. Hence
\begin{align*}
\|(\mathrm{Id}_{X_h}+\delta t(A_h+\mathsf V_h^n))^{-1}\|
\le (1-\delta t M)^{-1}
\le e^{2\delta t M}.
\end{align*}
Iterating this one-step bound from $j$ to $\ell$ yields \cref{eq:adjointenergy}.

Next, applying the same bound to \cref{eq:forward} and using
\begin{align*}
\delta t\sum_{n=1}^N\|u^n\|_{U_h}
\le\sqrt T\|u\|_{\ell_{\delta t}^2(U_h)},
\end{align*}
gives \cref{eq:forward-energy-bound}.

It remains to prove the duality identity. For each time step, \cref{eq:forward,eq:adjoint} yield
\begin{align*}
(y^n,p^n)_{X_h}
&=(y^n,p^{n-1})_{X_h}
+\delta t\bigl((A_h+\mathsf V_h^n)y^n,p^{n-1}\bigr)_{X_h}\\
&=(y^{n-1},p^{n-1})_{X_h}
+\delta t(B_hu^n,p^{n-1})_{X_h}.
\end{align*}
Summing over $n=1,\ldots,N$ gives \cref{eq:trueD}.
\end{proof}

\subsection{Relaxed observability with exponentially small terminal remainder}\label{sec:observability}

We now derive the relaxed observability estimate from the discrete Carleman estimate. Fix the constants $\kappa,h_0,\Delta_0$ supplied by \cref{cor:scales}. Recall that $S_{h,{\delta t}}$ is defined in \cref{eq:mainS}. For $h\leq h_0$ and ${\delta t}\leq \Delta_0$, the discrete Carleman estimate is valid whenever $q_{\max}\leq S_{h,{\delta t}}$.

The discrete Carleman estimate yields observability with an exponentially small terminal remainder.
\begin{theorem}\label{prop:physicalO}
There exist constants $C_{\mathrm b},d_0>0$, depending additionally on $\omega$ and on the mesh bounds in \cref{prop:p1,prop:fd} for the chosen realization, such that the following holds:

Given $T>0$ and $M\ge0$, set
$
q_{\mathrm b}=C_{\mathrm b}\bigl(q_c+T^{-1}+M^{2/3}\bigr)$.
Then there exists a constant $C(T,M)>0$ such that, for every spatial mesh size $h$ and time step $\delta t=T/N$ satisfying
\begin{align}\label{eq:physicalconditions}
h\le h_0,\qquad
\delta t\le\min\{\Delta_0,T/12\},\qquad
\delta t M\le\tfrac12,\qquad
16q_{\mathrm b}\le S_{h,\delta t}\le\frac{q_{\mathrm b}T}{4\delta t},
\end{align}
every solution of \cref{eq:adjoint} satisfies
\begin{align}\label{eq:mainO}
\|p^0\|_{X_h}^2
\le C(T,M)\Bigl(
\delta t\sum_{n=1}^N\|B_h^*p^{n-1}\|_{U_h}^2
+e^{-2d_0S_{h,\delta t}}\|p^N\|_{X_h}^2\Bigr).
\end{align}
\end{theorem}
\begin{proof}
We first fix the constants, then argue in three steps: we introduce a periodic Carleman weight and cut off the adjoint, we absorb the potential and estimate the endpoint contribution, and finally we recover the initial adjoint norm.

With the Carleman and load constants fixed, we choose $C_{\mathrm b}$ sufficiently large, independently of $T,M,h,\delta t$, so that the weight bounds in Step~1 and the potential absorption condition in Step~2 hold. This is possible because
$$
q_{\mathrm b}\ge C_{\mathrm b}q_c,\qquad
q_{\mathrm b}T\ge C_{\mathrm b},\qquad
q_{\mathrm b}^3\ge C_{\mathrm b}^3M^2.
$$
We fix this choice before taking $h,\delta t$ satisfying \cref{eq:physicalconditions}.

\medskip

\textbf{Step 1: Choice of the periodic weight and cutoff of the adjoint.}
Put $\delta=q_{\mathrm b}/S_{h,\delta t}\le1/16$ and define
\begin{align}\label{eq:physical-weight}
q(t)=\frac{q_{\mathrm b}}{\sqrt{\sin^2(\pi t/T)+\delta^2}}.
\end{align}
Then
$$
\min_{[0,T]}q=\frac{q_{\mathrm b}}{\sqrt{1+\delta^2}},
\qquad
\max_{[0,T]}q=S_{h,\delta t},
$$
and differentiation gives
$$
\frac{|q'|}{q^2}\le\frac{C}{q_{\mathrm b}T},\qquad
\frac{|q''|}{q^3}\le\frac{C}{q_{\mathrm b}^2T^2}.
$$
The choice of $C_{\mathrm b}$ ensures \cref{eq:qclass}. By \cref{eq:scalerule,eq:mainS}, the mesh conditions also give
$$
hS_{h,\delta t}\le\delta_x,\qquad
\delta t S_{h,\delta t}^2\le\delta_{\mathrm{time}}.
$$
Let $\chi$ be a smooth function with values in $[0,1]$ that vanishes at distance at most $\delta T$ from $\{0,T\}$, equals one at distance at least $2\delta T$, and satisfies $|\chi'|\le C/(\delta T)$. For the cutoff sequence and its loads, the discrete weighted norms below use the periodic index set $\{0,\ldots,N-1\}$. Define the periodic sequence
$$
z^j=\chi(t_j)p^j,\qquad 0\le j\le N-1,\qquad z^N=z^0=0,
$$
and set
$$
r^j=\frac{\chi(t_{j+1})-\chi(t_j)}{\delta t}\,p^{j+1},\qquad 0\le j\le N-1.
$$
Since $\chi(t_N)p^N=0$, \cref{eq:adjoint} gives
\begin{align}\label{eq:cutoffresidual}
D_-z^j+A_hz^j
=-\mathsf V_h^{j+1}z^j-r^j.
\end{align}
Define the potential load by
$$
f_{\mathrm{pot}}^j=
\begin{cases}
-b^{j+1}R_hz^j, &\text{for finite elements},\\[2pt]
(-b_i^{j+1}z_i^j)_i, &\text{for finite differences}.
\end{cases}
$$
In both cases, $\Gamma_hf_{\mathrm{pot}}^j=-\mathsf V_h^{j+1}z^j$, and the load satisfies
$$
\|f_{\mathrm{pot}}^j\|_{\mathfrak F_h,\theta_j}\le CM\|z^j\|_{X_h,\theta_j}.
$$
Set $f^j=f_{\mathrm{pot}}^j-\mathcal I_hr^j$. By \cref{ass:loads,eq:cutoffresidual},
$$
D_-z^j+A_hz^j=\Gamma_hf^j,\qquad 0\le j\le N-1.
$$
Applying \cref{eq:finalCE} to the periodic pair $(z,f)$, with $o=3$, and using \cref{eq:stateload}, we obtain
\begin{align}\label{eq:cutoff-carleman}
\mathcal E_{q,h,\delta t}(z)
&\le C\bigl(\|f\|_{\ell_{\delta t}^2(\mathfrak F_h),\theta}^2
+\|q^{3/2}B_h^*z\|_{\ell_{\delta t}^2(U_h),\theta}^2\bigr)\notag\\
&\le CM^2\|z\|_{\ell_{\delta t}^2(X_h),\theta}^2
+C\|r\|_{\ell_{\delta t}^2(X_h),\theta}^2
+C\|q^{3/2}B_h^*z\|_{\ell_{\delta t}^2(U_h),\theta}^2.
\end{align}

\medskip

\textbf{Step 2: Absorption of the potential and estimate of the endpoint contribution.}
By \cref{eq:discreteenergy} and $q\ge q_{\mathrm b}/\sqrt{1+\delta^2}$, our choice of $C_{\mathrm b}$ gives
\begin{align*}
CM^2\|z\|_{\ell_{\delta t}^2(X_h),\theta}^2
&\le\frac{CM^2(1+\delta^2)^{3/2}}{q_{\mathrm b}^3}
\|q^{3/2}z\|_{\ell_{\delta t}^2(X_h),\theta}^2 \le\frac12\|q^{3/2}z\|_{\ell_{\delta t}^2(X_h),\theta}^2
\le\frac12\mathcal E_{q,h,\delta t}(z).
\end{align*}
Since $z^j=\chi(t_j)p^j$, $0\le\chi\le1$, and $\rho\le-\rho_*$,
\begin{align*}
\|q^{3/2}B_h^*z\|_{\ell_{\delta t}^2(U_h),\theta}^2
&=\delta t\sum_{j=0}^{N-1}q_j^3\chi(t_j)^2\|B_h^*p^j\|_{U_h,\theta_j}^2 \le\Bigl(\sup_{s\ge q_c}s^3e^{-2\rho_*s}\Bigr)
\delta t\sum_{j=0}^{N-1}\|B_h^*p^j\|_{U_h}^2 \le C\delta t\sum_{j=0}^{N-1}\|B_h^*p^j\|_{U_h}^2.
\end{align*}

Let $\mathcal J_\chi$ denote the set of indices where the cutoff difference does not vanish:
$$
\mathcal J_\chi:=\{j\in\{0,\ldots,N-1\}:\chi(t_{j+1})\ne\chi(t_j)\}.
$$
Each corresponding cell meets $[\delta T,2\delta T]$ or $[T-2\delta T,T-\delta T]$. Counting these cells and using $\delta t\le\delta T/4$, we obtain
$$
\delta t\sum_{j\in\mathcal J_\chi}1\le2\delta T+4\delta t\le3\delta T.
$$
Moreover, $\min\{t_j,T-t_j\}\le3\delta T$ for every $j\in\mathcal J_\chi$.
Define $c_{\mathrm{end}}:=2\rho_*/\sqrt{1+9\pi^2}$, a positive constant depending only on $\rho_*$. Since $|\sin(\pi t_j/T)|\le3\pi\delta$ at these indices, \cref{eq:physical-weight} gives
$$
q_j\ge\frac{S_{h,\delta t}}{\sqrt{1+9\pi^2}},\qquad
\theta_j(x)^2\le e^{-c_{\mathrm{end}}S_{h,\delta t}},
\qquad x\in\Omega,\quad j\in\mathcal J_\chi.
$$
Using the derivative bound on $\chi$, \cref{eq:adjointenergy}, and $\delta=q_{\mathrm b}/S_{h,\delta t}$, we conclude
\begin{align}\label{eq:endpointbudget}
\|r\|_{\ell_{\delta t}^2(X_h),\theta}^2
&=\delta t\sum_{j\in\mathcal J_\chi}
\Bigl|\frac{\chi(t_{j+1})-\chi(t_j)}{\delta t}\Bigr|^2
\|p^{j+1}\|_{X_h,\theta_j}^2\notag\\
&\le\frac{C\delta t\sum_{j\in\mathcal J_\chi}1}{(\delta T)^2}
e^{-c_{\mathrm{end}}S_{h,\delta t}}e^{4MT}\|p^N\|_{X_h}^2\notag\\
&\le\frac{Ce^{4MT}}{q_{\mathrm b}T}S_{h,\delta t}e^{-c_{\mathrm{end}}S_{h,\delta t}}\|p^N\|_{X_h}^2\notag\\
&\le C(T,M)S_{h,\delta t}e^{-c_{\mathrm{end}}S_{h,\delta t}}\|p^N\|_{X_h}^2.
\end{align}
After absorption, \cref{eq:cutoff-carleman,eq:endpointbudget} therefore give
\begin{align}\label{eq:cutoff-energy}
\mathcal E_{q,h,\delta t}(z)
\le C\delta t\sum_{j=0}^{N-1}\|B_h^*p^j\|_{U_h}^2
+C(T,M)S_{h,\delta t}e^{-c_{\mathrm{end}}S_{h,\delta t}}\|p^N\|_{X_h}^2.
\end{align}

\medskip

\textbf{Step 3: Recovery of the initial adjoint norm.}
On $[T/3,2T/3]$, we have $\chi=1$, $q\asymp q_{\mathrm b}$, and $\theta\ge e^{-Cq_{\mathrm b}}$. Since $\delta t\le T/12$, the nodes in this interval have total discrete measure at least $T/6$. By \cref{eq:discreteenergy,eq:adjointenergy},
\begin{align}\label{eq:initial-adjoint-recovery}
\mathcal E_{q,h,\delta t}(z)
&\ge cq_{\mathrm b}^3e^{-Cq_{\mathrm b}}
\delta t\sum_{t_j\in[T/3,2T/3]}\|p^j\|_{X_h}^2 \ge c(T,M)\|p^0\|_{X_h}^2.
\end{align}
Choose $d_0>0$ so that $2d_0<c_{\mathrm{end}}$. Then
$$
S_{h,\delta t}e^{-c_{\mathrm{end}}S_{h,\delta t}}
\le Ce^{-2d_0S_{h,\delta t}},
$$
and \cref{eq:cutoff-energy,eq:initial-adjoint-recovery} imply \cref{eq:mainO}.
\end{proof}

\subsection{Uniform approximate null controllability with exponentially small terminal error}\label{sec:approximate-control}

By duality, \cref{prop:physicalO} yields approximate null controls with a cost uniform in $h$ and ${\delta t}$.
\begin{theorem}\label{prop:physicalC}
Under the hypotheses of \cref{prop:physicalO}, there exists $C(T,M)>0$ such that, for every $y_{0,h}\in X_h$, there is a control $u\in\ell_{\delta t}^2(U_h)$ satisfying
\begin{align}\label{eq:mainC}
\|u\|_{\ell_{\delta t}^2(U_h)}\leq C(T,M)\|y_{0,h}\|_{X_h}.
\end{align}
The corresponding solution of \cref{eq:forward} satisfies
\begin{align}\label{eq:main-terminal}
\|y^N\|_{X_h}\leq C(T,M)e^{-d_0S_{h,{\delta t}}}\|y_{0,h}\|_{X_h},
\end{align}
with $d_0$ as in \cref{prop:physicalO}.
\end{theorem}
\begin{proof}
We use the standard penalized HUM argument; see \cite[Propositions~2.2 and~2.4]{Boyer2013}. For completeness, we present the details in four steps: we introduce a strictly convex functional, take its unique minimizer, use the optimality condition to identify the resulting control and terminal state, and finally estimate them.

\medskip

For each $p_T\in X_h$, let $p$ solve \cref{eq:adjoint}, and define
$$
J(p_T)=\frac{\delta t}{2}\sum_{n=1}^N\|B_h^*p^{n-1}\|_{U_h}^2
+\frac{e^{-2d_0S_{h,\delta t}}}{2}\|p_T\|_{X_h}^2
-(y_{0,h},p^0)_{X_h}.
$$
For fixed $h,\delta t$, the positive penalty coefficient makes $J$ strictly convex and coercive. Hence $J$ admits a unique minimizer $\widehat p_T\in X_h$. Let $\widehat p$ denote the corresponding  solution of  \cref{eq:adjoint}.

\smallskip

Since $\widehat p_T$ minimizes $J$, for every $p_T\in X_h$ and its associated adjoint solution $p$, we have
\begin{align}\label{eq:hum-optimality}
\delta t\sum_{n=1}^N(B_h^*\widehat p^{n-1},B_h^*p^{n-1})_{U_h}
+e^{-2d_0S_{h,\delta t}}(\widehat p_T,p_T)_{X_h}
=(y_{0,h},p^0)_{X_h}.
\end{align}

Choose
$
u^n=-B_h^*\widehat p^{n-1}$, $1\le n\le N$,
and let $y$ solve \cref{eq:forward}. Combining the duality identity \cref{eq:trueD} with the optimality condition \cref{eq:hum-optimality} yields
$$
(y^N,p_T)_{X_h}=e^{-2d_0S_{h,\delta t}}(\widehat p_T,p_T)_{X_h},
\qquad p_T\in X_h.
$$
Since $p_T\in X_h$ is arbitrary, it follows that
$
y^N=e^{-2d_0S_{h,\delta t}}\widehat p_T$.

Taking $p_T=\widehat p_T$ in \cref{eq:hum-optimality} and applying \cref{eq:mainO} to $\widehat p$, we obtain
\begin{align*}
\|u\|_{\ell_{\delta t}^2(U_h)}^2+e^{-2d_0S_{h,\delta t}}\|\widehat p_T\|_{X_h}^2
&=(y_{0,h},\widehat p^0)_{X_h} \le\sqrt{C(T,M)}\|y_{0,h}\|_{X_h}
\bigl(\|u\|_{\ell_{\delta t}^2(U_h)}^2
+e^{-2d_0S_{h,\delta t}}\|\widehat p_T\|_{X_h}^2\bigr)^{1/2}.
\end{align*}
Dividing by the square-root factor gives
$$
\|u\|_{\ell_{\delta t}^2(U_h)}^2+e^{-2d_0S_{h,\delta t}}\|\widehat p_T\|_{X_h}^2
\le C(T,M)\|y_{0,h}\|_{X_h}^2.
$$
In particular, this yields \cref{eq:mainC}. Finally, combining the identity $y^N=e^{-2d_0S_{h,\delta t}}\widehat p_T$ with the preceding bound gives
$$
\|y^N\|_{X_h}\le C(T,M)e^{-d_0S_{h,\delta t}}\|y_{0,h}\|_{X_h},
$$
which is \cref{eq:main-terminal}.
\end{proof}

\begin{remark}\label{rem:fully-discrete-comparison}
Fix $T,M$ and consider the parabolic scaling $\delta t\asymp h^2$ with $\delta t\le T^{-2}h^2$. Under this scaling, \cref{eq:mainS,eq:main-terminal} yield the terminal-error decay factor $e^{-ch^{-4/5}}$. For the same scaling, \cite[Proposition~3.1 and Theorem~1.10]{GCH} gives the factor $e^{-ch^{-1/2}}$ for one-dimensional finite differences with the backward Euler scheme. Thus, in this common setting, \cref{prop:physicalC} improves the decay exponent and, in addition, applies to the higher-dimensional realizations in \cref{prop:p1,prop:fd}. Since both bounds decay faster than any fixed algebraic power of $h$, they are consistent with the practical accuracy criterion discussed in \cite[Remark~4.4]{Boyer2014}.
\end{remark}
\begin{remark}
\Cref{prop:physicalC} yields $\phi$-null controllability in the sense of \cite[Definition~2.1 and Theorem~3.2]{Boyer2013}. Indeed, taking
$\phi(h,\delta t)=e^{-2d_0S_{h,\delta t}},$
we obtain a control cost that is uniform in $h,\delta t$, together with a terminal error bounded by $C(T,M)\sqrt{\phi(h,\delta t)}\,\|y_{0,h}\|_{X_h}$.
\end{remark}

\begin{remark}
For fixed $T,M$, the conditions in \cref{eq:physicalconditions} hold for all sufficiently small $h,\delta t$ when $h,\delta t\to0$ independently, with $\delta t=T/N$. Indeed, \cref{eq:mainS} gives
\begin{align*}
\lim_{h,\delta t\to0}S_{h,\delta t}=+\infty.
\end{align*}
Moreover, $0\leq\delta t S_{h,\delta t}\leq\kappa(\delta t)^{3/5}$ implies
\begin{align*}
\lim_{h,\delta t\to0}\delta t S_{h,\delta t}=0.
\end{align*}
Thus the two-sided bound on $S_{h,\delta t}$ is eventually satisfied. If, in addition, one imposes the optional scaling $\delta t\le ch^2$, then $S_{h,\delta t}\asymp h^{-4/5}$ is retained.
\end{remark}

\subsection{Uniform observability under coarse-space filtering}\label{sec:filtering}

Suitable filtering makes it possible to control the retained component of the terminal state exactly to zero, with a cost that is bounded independently of $h,\delta t$ and a full terminal state that is exponentially small \cite{Zuazua2005}. In our setting, the terminal remainder in \cref{eq:mainO} can be absorbed in an analogous way, provided that the terminal adjoint data have a sufficiently small discrete Rayleigh quotient.

Throughout this subsection, we assume \cref{eq:physicalconditions}, fix $0<\vartheta<1$, write $S=S_{h,\delta t}$, and set
\begin{align}\label{9.26-eq15}
C_0=\max\{1,C(T,M)\},
\end{align}
with $C(T,M)$ as in \cref{prop:physicalO}. For a terminal subspace $F\subset X_h$, define the discrete Rayleigh quotient\vspace{-2mm}
\begin{align}\label{eq:filter-rayleigh}
\Lambda_F=\sup_{0\ne v\in F}\frac{(A_hv,v)_{X_h}}{\|v\|_{X_h}^2},
\end{align}
with $\Lambda_{\{0\}}=0$. We then set\vspace{-2mm}
$$
D_{T,M}=MT+\tfrac12M^2T^2,
\qquad
K_F=e^{T\Lambda_F+D_{T,M}}.
$$
The restriction is imposed only on $p^N$, while the fine-grid equations remain unchanged. In particular, neither $A_h$ nor the potential is required to preserve $F$.

We first bound the terminal norm by the initial adjoint norm.
\begin{lemma}\label{lem:filter-reverse}
Every solution of \cref{eq:adjoint} with $p^N\in F$ satisfies
\begin{align}\label{eq:filter-reverse}
\|p^N\|_{X_h}\leq K_F\|p^0\|_{X_h}.
\end{align}
\end{lemma}
The proof of \cref{lem:filter-reverse} is given in \cref{sec:filter-reverse-proof}.
It shows that, at each backward-Euler step, the Rayleigh quotient can increase by at most $\delta t\,M^2$. Importantly, the resulting constants are independent of both $\dim X_h$ and $\|A_h\|$.

\begin{theorem}\label{thm:filtered-observability}
Recall $C_0$ is given by \cref{9.26-eq15}. If
\begin{align}\label{eq:filter-condition}
T\Lambda_F\leq\vartheta d_0S,
\qquad (1-\vartheta)d_0S\geq D_{T,M}+\tfrac12\log(4C_0),
\end{align}
then every solution of \cref{eq:adjoint} with $p^N\in F$ satisfies
\begin{align}\label{eq:filtered-observability}
\|p^0\|_{X_h}^2
\leq2C_0\,\delta t\sum_{n=1}^N\|B_h^*p^{n-1}\|_{U_h}^2.
\end{align}
\end{theorem}
\begin{proof}
By \cref{eq:filter-condition}, $4C_0e^{-2d_0S}K_F^2\leq1$. Inserting \cref{eq:filter-reverse} into \cref{eq:mainO} bounds the remainder by $\tfrac14\|p^0\|_{X_h}^2$. Absorption proves \cref{eq:filtered-observability}.
\end{proof}

A penalty on the complementary terminal space gives one control that annihilates the retained component and leaves an exponentially small full terminal state. Let $\Pi_F$ denote the $X_h$-orthogonal projection onto $F$.
\begin{corollary}\label{cor:filtered-control}
Assume \cref{eq:filter-condition}. Then, for every $y_{0,h}\in X_h$, there exists a control $u\in\ell_{\delta t}^2(U_h)$ whose corresponding solution of \cref{eq:forward} satisfies $\Pi_Fy^N=0$.
Moreover, the control cost satisfies
\begin{align}\label{eq:filtered-control}
\|u\|_{\ell_{\delta t}^2(U_h)}\le\sqrt{2C_0}\|y_{0,h}\|_{X_h},
\end{align}
while the full terminal state fulfills
\begin{align}\label{eq:filtered-terminal}
\|y^N\|_{X_h}\le C(T,M)e^{-(1-\vartheta)d_0S}\|y_{0,h}\|_{X_h},
\end{align}
where the constant $C(T,M)$ is as in \cref{prop:physicalO}, which is independent of $h,\delta t$ and $F$.
\end{corollary}
The proof of \cref{cor:filtered-control} is given in \cref{sec:filter-control-proof}.

Coarse-space realizations are motivated by the two-grid strategies studied for numerical wave controllability in \cite{NegreanuZuazua2004,IgnatZuazua2009}, where interpolation from a coarser grid yields adjoint data with mesh-uniform observability. In the present parabolic setting, the admissible coarse scale follows from \cref{eq:filter-condition}.

For consistent-mass finite elements, a nested coarse space gives the required Rayleigh bound.
\begin{proposition}\label{prop:filter-p1}
Let $V_h$ be the space in \cref{prop:p1}. Suppose a conforming, shape-regular, quasi-uniform coarse triangulation of maximum diameter $H_c$ has a $P_1$ space $V_{H_c}\subset V_h$. Then $F=V_{H_c}$ satisfies
\begin{align}\label{eq:filter-p1-bound}
\Lambda_F\leq CH_c^{-2},
\end{align}
where $C$ depends only on the coarse mesh regularity bounds and the dimension.
\end{proposition}
\begin{proof}
The coarse-space inverse estimate gives $\|\nabla v\|_{L^2(\Omega)}^2\leq CH_c^{-2}\|v\|_{L^2(\Omega)}^2$ for $v\in V_{H_c}$. Since the physical inner product is the consistent $L^2$ product and $(A_hv,v)_{X_h}=\|\nabla v\|_{L^2(\Omega)}^2$, this proves \cref{eq:filter-p1-bound}. The fine mesh need only satisfy \cref{prop:p1}.
\end{proof}

For Cartesian differences, tensor-product linear interpolation gives the corresponding bound in the original diagonal mass.
\begin{proposition}\label{prop:filter-fd}
Consider \cref{eq:physicalFD} and a nested Cartesian coarse grid with steps $H_r=m_rh_r$, where $m_r\in\mathbb N$ divides $N_r$, $1\le r\le d$. Let $H_c=\max_rH_r$. Let $I_{H_c}^h$ prolong zero-boundary coarse nodal values by tensor-product piecewise linear interpolation, and set $F=\operatorname{Ran}I_{H_c}^h$. Then the discrete Rayleigh quotient of $F$ satisfies\vspace{-2mm}
\begin{align}\label{eq:filter-fd-bound}
\Lambda_F\le12\sum_{r=1}^dH_r^{-2}.
\end{align}
\end{proposition}
\begin{proof}
We first consider the one-dimensional case on $(0,L)$, and write $H$ for the coarse step. Let $v_H$ be the continuous coarse piecewise linear function, and let $z_h$ denote its vector of values on the fine grid. Since the endpoint values vanish, summation over the fine intervals gives\vspace{-2mm}
$$
(A_hz_h,z_h)_{X_h}=\|v_H'\|_{L^2(0,L)}^2.\vspace{-2mm}
$$
The induced mass satisfies\vspace{-2mm}
$$
\|z_h\|_{X_h}^2=\|v_H\|_{L^2(0,L)}^2+\tfrac16h^2\|v_H'\|_{L^2(0,L)}^2.\vspace{-2mm}
$$
By the linear-element inverse estimate\vspace{-2mm}
$$
\|v_H'\|_{L^2(0,L)}^2\le12H^{-2}\|v_H\|_{L^2(0,L)}^2,\vspace{-2mm}
$$
we therefore obtain\vspace{-2mm}
$$
(A_hz_h,z_h)_{X_h}\le12H^{-2}\|z_h\|_{X_h}^2.\vspace{-2mm}
$$
This proves \cref{eq:filter-fd-bound} in one dimension.

In dimension $d$, the induced mass is the tensor product of the one-dimensional induced masses. Tensorizing the preceding quadratic-form bound with these masses in the other directions, and summing the resulting directional stiffness terms, yields \cref{eq:filter-fd-bound}.
\end{proof}

When the Cartesian coarse steps are comparable, both realizations have $\Lambda_F\leq C H_c^{-2}$. Thus, on mesh hierarchies with compatible coarse levels, \cref{eq:mainS} gives the sufficient choice
\begin{align}\label{eq:filter-coarse-scale}
H_c\geq C(T,\vartheta)\max\{h^{2/5},(\delta t)^{1/5}\},
\end{align}
with the constant sufficiently large and $S_{h,\delta t}$ satisfying the second condition in \cref{eq:filter-condition}. This sufficient scale is generally larger than a fixed multiple of the fine scale.
\begin{remark}
The spectral space spanned by eigenvectors of $A_h$ with eigenvalues at most $\vartheta d_0S/T$ also satisfies the first condition in \cref{eq:filter-condition}. The coarse-space constructions above avoid computing these eigenvectors.
\end{remark}

\subsection{Weak convergence of reconstructed controls to continuous null controls}\label{sec:control-convergence}

We finally show that the discrete controls constructed above converge, along weakly convergent subsequences, to continuous null controls.

Fix $y_0\in H$ and a real potential $b\in L^\infty(\mathcal Q)$ with $\|b\|_{L^\infty(\mathcal Q)}\leq M$. We discretize these data by projection and averaging. For $P_1$ finite elements, take\vspace{-2mm}
\begin{align}\label{eq:conv-data-fe}
y_{0,h}=\Pi_hy_0,\qquad b^n(x)=\frac1{\delta t}\int_{t_{n-1}}^{t_n}b(t,x)\,dt.
\end{align}
For differences, define the interior dual cells\vspace{-2mm}
\begin{align}\label{eq:conv-dual-cells}
D_i=\prod_{r=1}^d(x_{r,i_r}-h_r/2,x_{r,i_r}+h_r/2),
\qquad |D_i|=\omega_i,
\end{align}
and use\vspace{-2mm}
\begin{align}\label{eq:conv-data-fd}
(y_{0,h})_i=\frac1{\omega_i}\int_{D_i}y_0(x)\,dx,\qquad
b_i^n=\frac1{{\delta t}\omega_i}\int_{t_{n-1}}^{t_n}\int_{D_i}b(t,x)\,dx\,dt.
\end{align}
Define $E_h:X_h\to H$ by $E_h=R_h$ for finite elements. For differences, let $E_hz=z_i$ on $D_i$ and zero on the uncovered boundary strip. Define $K_h:U_h\to H$ as zero extension from $\omega$ for finite elements and as $E_hB_h$ for differences. These reconstructions preserve the physical norms:
\begin{align}\label{eq:conv-pairings}
(E_hz,E_hw)_H=(z,w)_{X_h},\qquad
\|K_hv\|_{L^2(\Omega)}=\|v\|_{U_h}.
\end{align}
They also preserve the control pairing $
(K_hv,E_hz)_H=(v,B_h^*z)_{U_h}$.
For $t\in(t_{n-1},t_n]$, set\vspace{-3mm}
$$
U_{h,{\delta t}}(t)=K_hu^n,\qquad
Y_{h,{\delta t}}(t)=R_hy^n,\qquad
\widetilde Y_{h,{\delta t}}(t)=E_hy^n.\vspace{-2mm}
$$
For differences, fix an open set $\omega_\sharp$ with $\omega_0\Subset\omega_\sharp\Subset\omega$ and restrict the active nodes to $\omega_\sharp$, extending the controls by zero to $U_h$. Their cellwise reconstructions are then supported in $\omega$ for sufficiently small $h$.

Denote by $V'=H^{-1}(\Omega)$ the dual of $V=H_0^1(\Omega)$. Let $A:V\to V'$ be the linear operator defined by
$$
\langle Av,w\rangle_{V',V}=a(v,w)=\int_\Omega\nabla v\cdot\nabla w\,dx,
\qquad v,w\in V,
$$
where $\langle\cdot,\cdot\rangle_{V',V}$ denotes the duality pairing. Since $H$ is embedded in $V'$ through the $L^2(\Omega)$ inner product, this operator extends the Dirichlet Laplacian $A=-\Delta_D$ with domain $D(A)=H^2(\Omega)\cap H_0^1(\Omega)$. Consequently, the energy solution belongs to $L^2(0,T;V)\cap H^1(0,T;V')\subset C([0,T];H)$.

With this functional setting in place, we may pass from the discrete controls to the continuous equation. The reconstructed discrete controls have weakly convergent subsequences whose limits drive the continuous state to zero at time $T$. This is made precise in the following corollary.
\begin{corollary}\label{prop:control-convergence}
Fix a realization in \Cref{sec:realizations}, with the data in \cref{eq:conv-data-fe} or \cref{eq:conv-data-fd}. Then there exists $C(T,M)>0$ such that the following holds. For every sequence of admissible mesh sizes $h_j$ and time steps $\delta t_j=T/N_j$ satisfying $\lim\limits_{j\to\infty}h_j=\lim\limits_{j\to\infty}\delta t_j=0$, one can choose, for all sufficiently large $j$, controls satisfying \cref{eq:mainC,eq:main-terminal} whose reconstructions $U_{h_j,\delta t_j}$ vanish outside $\omega$ and form a relatively weakly compact family in $L^2(0,T;H)$.
Moreover, every weak accumulation point of this family is the zero extension of a control $u\in L^2(\mathcal Q_\omega)$ for which the energy solution of
\begin{align}\label{eq:conv-continuous-system}
\partial_t y+Ay+b(t,\cdot)y=\mathbf1_\omega u\quad\text{in }V',
\qquad y(0)=y_0,
\end{align}
satisfies $y(T)=0$. The limiting control satisfies
\begin{align}\label{eq:conv-control-bound}
\|u\|_{L^2(\mathcal Q_\omega)}\le C(T,M)\|y_0\|_{L^2(\Omega)}.
\end{align}
\end{corollary}
\begin{proof}
We use the standard compactness argument for control approximation; see \cite[Theorem~2.5]{Boyer2013} for the semidiscrete setting. The proof proceeds in four steps: uniform bounds for the controls, weak compactness of the states, identification of the limit equation, and identification of the endpoint values together with the control bound.

\smallskip

Since the conditions in \cref{eq:physicalconditions} hold eventually along the given sequence, \cref{prop:physicalC} applies. In the finite difference case, restricting the controls to $\omega_\sharp$ preserves \cref{ass:observation} by the local norm comparison in \cref{sec:fd-verification}. Consequently, \cref{eq:mainC,eq:conv-pairings,eq:conv-data-consistency} show that the controls $U_{h,\delta t}$ are uniformly bounded in $L^2(0,T;H)$.

\smallskip

Next, we test \cref{eq:forward} with $\delta t y^n$, sum over $n$, and use \cref{eq:forward-energy-bound,eq:conv-stiffness}. This yields a uniform bound for $Y_{h,\delta t}$ in $L^2(0,T;V)$. Therefore, starting from any weakly convergent control subsequence and passing to a further subsequence if necessary, relabeled by $j$, we obtain
\begin{align*}
\lim_{j\to\infty}U_{h_j,\delta t_j}=\mathbf1_\omega u
\quad\text{in the weak topology of }L^2(0,T;H).
\end{align*}
Along the same subsequence, the states satisfy
\begin{align*}
\lim_{j\to\infty}Y_{h_j,\delta t_j}=y
\quad\text{in the weak topology of }L^2(0,T;V).
\end{align*}
By \cref{eq:conv-consistency}, the modified states satisfy
\begin{align*}
\lim_{j\to\infty}\widetilde Y_{h_j,\delta t_j}=y
\quad\text{in the weak topology of }L^2(0,T;H).
\end{align*}

\smallskip

It remains to identify the limit equation and the endpoint values. For the smooth tests $\phi$ and their interpolants $\phi_h^n$ from \cref{sec:control-consistency}, we test \cref{eq:forward} with $\delta t\phi_h^{n-1}$ and sum by parts. Since $\lim\limits_{h,\delta t\to0}S_{h,\delta t}=+\infty$, \cref{eq:main-terminal} and the uniform bounds for $y_{0,h}$ and $\phi_h^N$ give
\begin{align*}
\lim_{h,\delta t\to0}(y^N,\phi_h^N)_{X_h}=0,
\qquad N=T/\delta t.
\end{align*}
This limit holds in particular along the chosen subsequence.
Using the consistency properties from \cref{sec:control-consistency}, we therefore obtain
\begin{align}\label{eq:conv-limit-identity}
-\int_0^T(y,\partial_t\phi)_H\,dt
+\int_0^T\bigl(a(y,\phi)+(by,\phi)_H\bigr)\,dt
=(y_0,\phi(0))_H+\int_0^T\int_\omega u\phi\,dx\,dt.
\end{align}

\smallskip

Interior time tests give
\begin{equation}\label{9.26-eq16}
\partial_t y=\mathbf1_\omega u-Ay-by\in L^2(0,T;V').
\end{equation}
To see this, let $\psi\in C_c^\infty(0,T)$ and $w\in V$, and take $\phi(t,x)=\psi(t)w(x)$ in \eqref{eq:conv-limit-identity}. Since $\psi$ vanishes near $0$ and $T$, the endpoint term $(y_0,\phi(0))_H$ drops out, and we obtain
$$
-\int_0^T\psi'(t)\,(y(t),w)_H\,dt
+\int_0^T\psi(t)\bigl(a(y(t),w)+(b(t)y(t),w)_H\bigr)\,dt
=\int_0^T\psi(t)\int_\omega u(t,x)w(x)\,dx\,dt.
$$
This is the distributional formulation of \eqref{9.26-eq16}. Indeed, each term on the right-hand side of \eqref{9.26-eq16} lies in $L^2(0,T;V')$: the control term $\mathbf1_\omega u$ lies in $L^2(0,T;H)\subset L^2(0,T;V')$, the term $Ay$ lies in $L^2(0,T;V')$ since $y\in L^2(0,T;V)$ and $A:V\to V'$ is bounded, and the potential term $by$ lies in $L^2(0,T;V')$ by the boundedness of $b$ and the embedding $H\hookrightarrow V'$. Hence \eqref{9.26-eq16} holds.

Consequently, $y\in C([0,T];H)$, so the unrestricted endpoint values of $\phi$ in \cref{eq:conv-limit-identity} identify $y(0)=y_0$ and $y(T)=0$. Finally, the control bound in \cref{eq:conv-control-bound} follows from weak lower semicontinuity.
\end{proof}

\begin{remark}\label{rem:spectral-comparison}
The independent refinement in \cref{prop:control-convergence} follows from the transferred observability estimate for bounded real space-time potentials. To relate this to the spectral framework in~\cite{BHLR11}, suppose that the uniform partial spectral inequality holds up to frequencies of order $h^{-\beta_{\mathrm{sp}}}$, with $\beta_{\mathrm{sp}}>0$, and set
$
\eta=\max\{h,(\delta t)^{1/\beta_{\mathrm{sp}}}\}$.
With this choice, the same discrete system satisfies the spectral hypothesis on the smaller range of order $\eta^{-\beta_{\mathrm{sp}}}$, and moreover $\delta t\le\eta^{\beta_{\mathrm{sp}}}$. Thus, in the autonomous backward-Euler case, Theorems~3.5 and~3.8 of \cite{BHLR11} apply with the effective cutoff parameter $\eta$, upon choosing $\gamma=\beta_{\mathrm{sp}}$ and $C_T=1$. This yields an independent-refinement consequence of their results. We emphasize that $\eta$ only changes the cutoff used in the argument; it does not alter the underlying spatial discretization.
\end{remark}
\begin{remark}\label{rem:filtered-control-limit}
For fixed $0<\vartheta<1$, the compactness argument in \cref{prop:control-convergence} applies equally well to the controls constructed in \cref{cor:filtered-control}, under the same data approximations and consistency assumptions. Indeed, the bounds in \cref{eq:filtered-control,eq:filtered-terminal} provide uniform control norms together with
\begin{align*}
\lim_{h,\delta t\to0}\|y^N\|_{X_h}=0,
\qquad N=T/\delta t.
\end{align*}
Consequently, every weak accumulation point of their reconstructions is a continuous null control.
\end{remark}

\vspace{-2mm}

\section{Numerical experiments for heat control}\label{sec:numerics}

We now illustrate the  control construction of \cref{prop:physicalC} on the consistent-mass $P_1$/backward Euler scheme. The experiments examine three aspects of the construction: control effectiveness on nonuniform meshes, independent spatial and temporal refinement, and the balance between terminal accuracy and iterative cost. In each case, the computed control is also applied, without re-optimization, in an independently refined approximation of the continuous heat equation.

\subsection{Model and numerical method}

We take $\Omega=(0,1)^2$, $T=0.1$, and $\omega=(0.4,0.6)^2$, and impose homogeneous Dirichlet boundary conditions on\vspace{-2mm}
$$
\partial_t y-\Delta y+b(t,x)y=\mathbf1_\omega U.
$$
The main experiments use $b=0$ and two initial data: the sine datum (S), $y_{0,\mathrm S}(x)=2\sin(\pi x_1)\sin(\pi x_2)$, and the localized datum (L),\vspace{-2mm}
$$
y_{0,\mathrm L}(x)=c_{\mathrm L}
\begin{cases}
\exp\bigl(-1/(1-|x-x_c|^2/r^2)\bigr),& |x-x_c|<r,\\
0,& |x-x_c|\geq r.
\end{cases}\vspace{-2mm}
$$
Here $x_c=(0.25,0.70)$, $r=0.1$, and $c_{\mathrm L}$ normalizes the $L^2(\Omega)$ norm to one. Both initial data therefore have unit norm, and we use their $L^2$ projections onto $V_h$.

For the nodal basis $\{\phi_i\}_{i=1}^{N_h}$, $N_h=\dim V_h$, define
\begin{align*}
\mathsf M_{ij}=\int_\Omega\phi_i\phi_j\,dx,\qquad
\mathsf K_{ij}=\int_\Omega\nabla\phi_i\cdot\nabla\phi_j\,dx,\qquad
\mathsf W_{ij}=\int_\omega\phi_i\phi_j\,dx.
\end{align*}
The observation matrix is assembled on the exact intersections of the triangles with $\omega$. We represent the state by $y_h^n=\sum\limits_{i=1}^{N_h}\mathbf y_i^n\phi_i$ and the control by $u_h^n=\big(\sum\limits_{i=1}^{N_h}\mathbf u_i^n\phi_i\big)\big|_\omega$, and write $\|\mathbf v\|_{\mathsf M}^2=\mathbf v^T\mathsf M\mathbf v$. For $\varepsilon>0$, the penalized problem minimizes\vspace{-2mm}
$$
\frac{\delta t}{2}\sum_{n=1}^N(\mathbf u^n)^T\mathsf W\mathbf u^n
+\frac{1}{2\varepsilon}\|\mathbf y^N\|_{\mathsf M}^2,\vspace{-2mm}
$$
subject to \cref{intro:discrete-model}. We solve the dual problem by mass-preconditioned conjugate gradients, with sparse direct solves for the time steps. With the sign convention of \cref{prop:physicalC}, $u_h^n=-p_h^{n-1}|_\omega$, where $p_h^n=\sum\limits_{i=1}^{N_h}\mathbf p_i^n\phi_i$ is the computed adjoint. From the forward state driven by the computed control, we evaluate the primal--dual gap\vspace{-2mm}
$$
\mathrm{gap}=\frac{\|\mathbf y^N-\varepsilon\mathbf p_T\|_{\mathsf M}^2}{2\varepsilon},\vspace{-2mm}
$$
where $\mathbf p_T=\mathbf p^N$. We stop the iterations when\vspace{-2mm}
\begin{align}\label{eq:num-gap}
\sqrt T e^{2MT}\sqrt{2\,\mathrm{gap}}\leq10^{-9},
\end{align}
where $M$ is the potential bound. Work is measured by the number of forward and adjoint time-step linear solves, including stopping checks.

Let $U_{h,\delta t}$ be the zero extension of $u_h^n$ on $(t_{n-1},t_n]$. We compute a reference terminal state $y_{\rm ref}(T;U_{h,\delta t},y_0)$ by sine-Galerkin propagation of the same control. The discrete terminal error is\vspace{-2mm}
$$
E_d=\frac{\|y_h^N\|_{L^2(\Omega)}}{\|y_0\|_{L^2(\Omega)}},\vspace{-2mm}
$$
and the reference terminal error is\vspace{-3mm}
$$
E_r=\frac{\|y_{\rm ref}(T;U_{h,\delta t},y_0)\|_{L^2(\Omega)}}{\|y_0\|_{L^2(\Omega)}}.\vspace{-2mm}
$$
We also report the control cost\vspace{-3mm}
$$
C_u =\frac{\|U_{h,\delta t}\|_{L^2((0,T)\times\omega)}}{\|y_0\|_{L^2(\Omega)}}.\vspace{-2mm}
$$
The control is projected onto tensor-product sine modes by integration over the intersections with $\omega$ and is kept constant on its original time intervals. For $b=0$, modal time propagation is exact. We compare 12, 24, 48, and, when needed, 96 modes per coordinate, together with increased quadrature orders. In the three main figures, the largest empirical refinement indicator, measured by the change in the reference terminal state relative to $E_r$, is $2.27\%$.

With $h=\max_{K\in\mathcal T_h}\operatorname{diam}K$, we choose, unless another value is specified, \vspace{-2mm}
\begin{equation}\label{eq:num-penalty}
\varepsilon_{\exp}(h,\delta t)
=\exp\!\bigl[-\min\{h^{-4/5},(\delta t)^{-2/5}\}\bigr].\vspace{-2mm}
\end{equation}
This choice follows the scale of the theoretical penalty $e^{-2d_0S_{h,\delta t}}$, with the experimental coefficient fixed at one after preliminary runs.

\subsection{Control effectiveness on nonuniform meshes}

We first examine whether the control remains effective on nonuniform meshes. To this end, we use two families of Delaunay meshes: perturbed Cartesian sites (F2) and a fixed smooth coordinate grading with density $1+2\exp(-((x-1/2)/0.18)^2)$, followed by the same perturbation (F3). The perturbations are at most $0.18$ of the coordinate spacing. The four nominal spatial resolutions are $n_{\mathrm{sp}}\in\{8,16,32,64\}$, with $N=128$ time steps. On these meshes, $\ell_K=h_K$ satisfies \cref{eq:grading}, with vertex-neighbor ratios below $2.04<3$.
\begin{figure}[tb]
\centering\includegraphics[width=\linewidth]{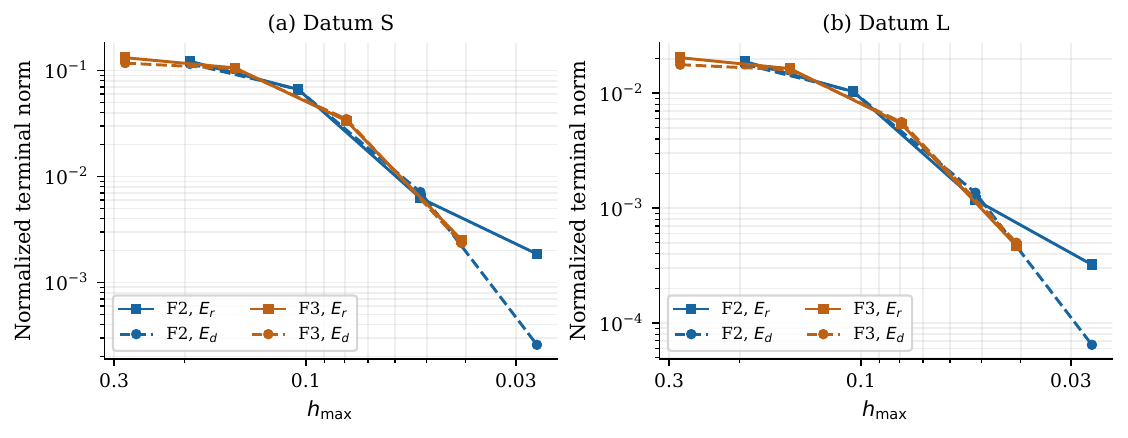}
\caption{\small  Control effectiveness on F2 and F3 meshes for (a) the sine datum S and (b) the localized datum L, with $T=0.1$, $N=128$, $b=0$, and \cref{eq:num-penalty}. Solid squares show the reference terminal norm $E_r$; dashed circles show the discrete terminal norm $E_d$. The horizontal coordinate is the measured maximum element diameter.}
\label{fig:num-effectiveness}
\end{figure}
Both terminal norms decrease along the reported refinements in \cref{fig:num-effectiveness}. To compare the controlled state with natural heat decay, let $E_{\rm free}=\|y_{\rm ref}(T;0,y_0)\|_{L^2(\Omega)}/\|y_0\|_{L^2(\Omega)}$. Its value is $e^{-2\pi^2T}=0.1389111$ for S and $0.0214074$ for L. At the finest reported F2/F3 levels, $E_r/E_{\rm free}$ lies between $0.0134$ and $0.0219$, while the control norms range from $0.708$ to $5.314$; see \cref{tab:num-effectiveness}. At a common nominal resolution, F2 and F3 have different measured diameters and hence different penalties, so this comparison concerns refinement rather than equal-accuracy efficiency.
\begin{table}[tb]
\centering\small
\caption{Terminal accuracy and control cost at the finest F2/F3 levels with $N=128$, and at the finest slow-time level with $N=53$. All rows use \cref{eq:num-penalty}; slow denotes F2 with \cref{eq:num-slow}.}
\label{tab:num-effectiveness}
\begin{tabular}{lrrrrr}
\toprule
Mesh/data & $h_{\max}$ & $N$ & $E_r$ & $C_u$ & $E_r/E_{\rm free}$ \\
\midrule
F2/L & 0.0266 & 128 & $3.22\times10^{-4}$ & 1.024 & $1.50\times10^{-2}$ \\
F2/S & 0.0266 & 128 & $1.86\times10^{-3}$ & 5.314 & $1.34\times10^{-2}$ \\
F3/L & 0.0410 & 128 & $4.69\times10^{-4}$ & 0.708 & $2.19\times10^{-2}$ \\
F3/S & 0.0410 & 128 & $2.53\times10^{-3}$ & 4.148 & $1.82\times10^{-2}$ \\
slow/L & 0.0266 & 53 & $8.65\times10^{-4}$ & 0.718 & $4.04\times10^{-2}$ \\
slow/S & 0.0266 & 53 & $5.11\times10^{-3}$ & 4.186 & $3.68\times10^{-2}$ \\
\bottomrule
\end{tabular}
\end{table}

A complementary experiment with local refinement in dyadic rings keeps $h_{\max}=0.0625$, $N=256$, and $\varepsilon=1.02097\times10^{-4}$ fixed, while increasing the global diameter contrast from $1.41$ to $90.51$, with vertex grading ratio two. The measured control norms vary only from $2.9233$ to $2.9254$, and the reference terminal norms from $0.01478$ to $0.01483$. Thus, when the mesh is refined locally at a fixed outer scale, both the measured control cost and the terminal norm remain nearly constant.

\subsection{Independent spatial and temporal refinement}

We next test whether the control remains accurate under independent spatial and temporal refinement. To this end, we combine the four F2 spatial levels with $N=16,32,64,128,256$. \Cref{fig:num-independent} compares $E_r$ for S with a fixed penalty $\varepsilon=10^{-6}$ and with \cref{eq:num-penalty}.
\begin{figure}[tb]
\centering\includegraphics[width=0.9\linewidth]{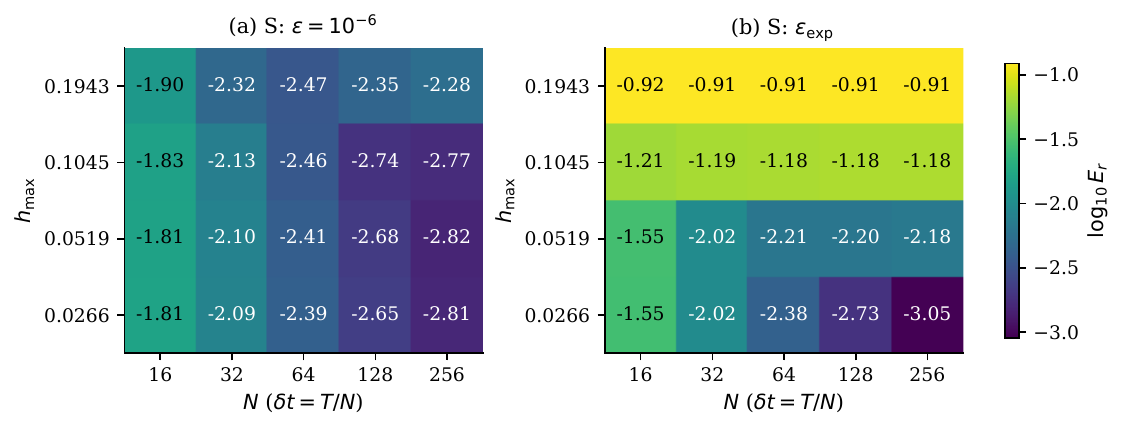}
\caption{\small  Independent spatial and temporal refinement for S on F2, with $T=0.1$ and $b=0$: (a) fixed $\varepsilon=10^{-6}$; (b) \cref{eq:num-penalty}. The entries are $\log_{10}E_r$ on a common color scale. Rows correspond to measured $h_{\max}$ and columns to $N$, with $\delta t=0.1/N$.}
\label{fig:num-independent}\vspace{-2mm}
\end{figure}
For the fixed penalty on the finest spatial mesh, increasing $N$ from 16 to 256 reduces $E_r$ from $1.55\times10^{-2}$ to $1.56\times10^{-3}$, whereas $E_d$ changes from $2.17\times10^{-3}$ to $1.46\times10^{-3}$. Thus, a small discrete terminal state can coexist with a substantially larger terminal state in the reference PDE. On the coarsest mesh, $E_r$ is not monotone in $N$. For the exponential penalty, refinement also changes the optimization problem, so the observed trends combine discretization and regularization effects.

We also consider the slower temporal refinement\vspace{-2mm}
\begin{align}\label{eq:num-slow}
N(h)=\max\{12,\lceil4[\log(1/h)]^2\rceil\},
\end{align}
using the natural logarithm and the measured diameter. This gives $N=12,21,36,53$. At the finest level, the S reference terminal norm is $5.11\times10^{-3}$, compared with $1.86\times10^{-3}$ for $N=128$. This illustrates independent refinement along a path distinct from the parabolic balance, with a different accuracy at the tested resolutions.

\vspace{-2mm}

\subsection{Penalty selection and computational cost}

Finally, we examine how the penalty parameter governs the trade-off between terminal accuracy and computational cost. We take $(n_{\mathrm{sp}},N)=(16,32)$ and $(32,64)$, and vary $\varepsilon\in\{10^{-2},10^{-4},\ldots,10^{-12}\}$. \Cref{fig:num-penalty} compares the discrete and reference terminal norms and the number of time-step solves for S. Across the S/L scans, six configurations at the smallest penalties hit the iteration limit before satisfying \cref{eq:num-gap}; the unconverged S iterates are marked by crosses.
\begin{figure}[tb]
\centering\includegraphics[width=0.9\linewidth]{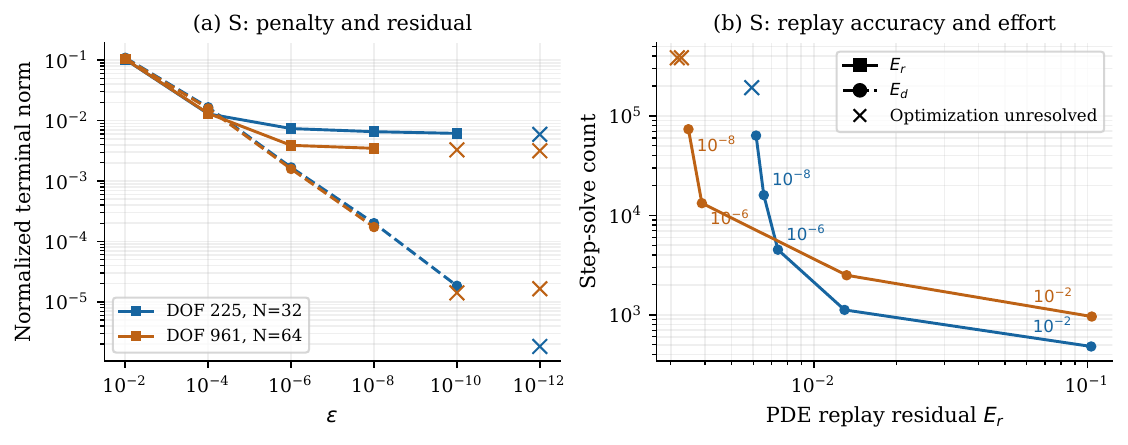}
\caption{\small  Penalty scans for S on F2: 225 state degrees of freedom with $N=32$, and 961 with $N=64$. (a) Discrete terminal norm $E_d$ (dashed circles) and reference terminal norm $E_r$ (solid squares); the penalty decreases to the right. (b) Number of forward and adjoint time-step solves versus $E_r$. Crosses mark iterates that did not satisfy \cref{eq:num-gap} and are excluded from the connecting curves.}
\label{fig:num-penalty}
\end{figure}

On the 225-degree-of-freedom mesh, reducing $\varepsilon$ from $10^{-6}$ to $10^{-10}$ decreases $E_d$ by a factor of $91.3$, but decreases $E_r$ only from $7.392\times10^{-3}$ to $6.151\times10^{-3}$, an improvement of $16.8\%$. In contrast, the number of time-step solves rises from 4512 to 63392. On the 961-degree-of-freedom mesh, the change from $10^{-6}$ to $10^{-8}$ decreases $E_d$ by a factor of $9.18$ and $E_r$ by $10.7\%$, while the solve count rises from 13248 to 73408. We note that these counts measure iterative effort rather than equal-cost work across different meshes.
The penalty therefore controls the discrete target more directly than the accuracy attained by the resulting control in the continuous equation. Once the latter changes only slightly, further reductions of the penalty mainly increase the iterative cost. The experiments distinguish this effect from the improvements obtained by spatial and temporal refinement.
\FloatBarrier

\appendix

\section{Verification of the continuous analytic assumptions}\label[appendix]{sec:continuous}

We first give a sufficient class of spatial operators for arbitrary $m\geq1$, and then verify all continuous analytic assumptions for the Dirichlet Laplacian used in \Cref{sec:realizations}.

\subsection{Symmetric coercive spatial operators}

Throughout this subsection, weighted estimates are understood at each fixed time.

We begin with a sufficient class of spatial operators. The following proposition shows that symmetric coercive differential forms with full elliptic regularity satisfy the spatial analytic assumptions.
\begin{proposition}\label{prop:spatial-operators}
Let $m\ge1$ be a fixed integer, set $V=H_0^m(\Omega)$, and let
\begin{align}\label{eq:coefficientform}
a(u,v)=\sum_{|\alpha|,|\beta|\le m}
\int_\Omega a_{\alpha\beta}(x)D^\beta u\,D^\alpha v\,dx,
\end{align}
where the coefficients are real, satisfy $a_{\alpha\beta}=a_{\beta\alpha}$, and belong to $W^{2m,\infty}(\Omega)$. Suppose that $a$ is bounded and coercive on $H_0^m(\Omega)$ with uniform positive constants, and that its associated operator has domain $D(A)=H^{2m}(\Omega)\cap H_0^m(\Omega)$ and satisfies\vspace{-2mm}
$$
\|v\|_{H^{2m}(\Omega)}\le C\|Av\|_{L^2(\Omega)}.
$$
Then \cref{ass:form,ass:localization,ass:weighted} hold for every fixed smooth spatial weight $\rho$ and every fixed finite $\tau_*$. The constants depend additionally on the $W^{2m,\infty}(\Omega)$ norms of $a_{\alpha\beta}$.
\end{proposition}
\begin{proof}
The representation theorem for closed symmetric forms, together with the assumed graph estimate, gives \cref{ass:form}. Smooth multiplication preserves $V$ and $D(A)$. By Leibniz' rule, every term of $[A,\sigma]\sigma^{-1}v$ contains a derivative of the weight; the coefficient of $D^jv$, $j\leq2m-1$, is bounded by $Cq^{2m-j}$ by \cref{eq:allowed-weight-derivatives}. This proves \cref{ass:localization}.

The Cauchy--Schwarz inequality in \cref{eq:coefficientform} gives \cref{eq:unshiftedformbound}. For the lower bound, set $w=\sigma v$. Leibniz' rule and Sobolev interpolation give, for every $\eta>0$,\vspace{-2mm}
\begin{align*}
|a(v,\sigma^2v)-a(w,w)|
&\leq C\sum_{\substack{0\leq r,s\leq m\\r+s<2m}}
q^{2m-r-s}\|D^rw\|_{L^2(\Omega)}\|D^sw\|_{L^2(\Omega)}\\
&\leq\eta\|w\|_{H^m(\Omega)}^2+C(\eta)q^{2m}\|w\|_{L^2(\Omega)}^2.
\end{align*}
Choose $\eta>0$ smaller than the coercivity constant of $a$. Coercivity, together with the product rule and interpolation bound\vspace{-2mm}
$$
\mathcal N_{\sigma,\Omega}(v)^2
\leq C\bigl(\|w\|_{H^m(\Omega)}^2+q^{2m}\|w\|_{L^2(\Omega)}^2\bigr),\vspace{-2mm}
$$
then yields \cref{eq:unshiftedcoercive}.
\end{proof}

For variable-coefficient forms, the spatial analytic assumptions are supplied by \cref{prop:spatial-operators}, whereas the Carleman estimate and numerical stiffness compatibility must be verified separately. Weight preservation also depends on the choice of boundary condition. In particular, a general smooth weight does not preserve a Neumann domain, since $\partial_\nu(\sigma v)=\sigma\partial_\nu v+v\partial_\nu\sigma$.

The Dirichlet Laplacian satisfies the spatial analytic assumptions on convex polytopes and boxes.
\begin{corollary}\label{prop:analytic-realizations}
\Cref{ass:form,ass:localization,ass:weighted} hold for the Dirichlet Laplacian on a fixed bounded convex polytope in $\mathbb R^d$, $d=2,3$, or on a box in any fixed dimension, with $m=1$, $V=H_0^1(\Omega)$, and $a(v,w)=(\nabla v,\nabla w)_{L^2(\Omega)}$.
\end{corollary}
\begin{proof}
Grisvard's convex-domain regularity theorem~\cite[Theorem~3.2.1.2]{Grisvard2011} gives $D(A)=H^2(\Omega)\cap H_0^1(\Omega)$ and \cref{eq:ellreg}. Poincar\'e's inequality and \cref{prop:spatial-operators} complete the verification.
\end{proof}

\subsection{Continuous Carleman estimate with matching time endpoints}

We now verify the continuous Carleman assumption for the heat operator. Since the auxiliary estimates are applied to periodic-in-time functions with matching endpoints, we state the estimate in that setting. In this subsection, integrals are over $\mathcal Q$ unless another domain is indicated.
\begin{proposition}\label{prop:continuousCE}
Consider the Dirichlet Laplacian on any domain in \cref{prop:analytic-realizations}, and fix a nonempty \(\omega_0\Subset\Omega\). Then there exist a negative smooth weight \(\rho\) and constants \(0<\beta\le1\), \(q_c\ge1\), and \(C>0\), determined only by \(d,\Omega,\omega_0\), such that \cref{eq:continuousabstract} holds for every \(T>0\), every \(q\) satisfying \cref{eq:qclass}, both signs of \(\epsilon\), and every \(v\in H^1(0,T;H)\cap L^2(0,T;D(A))\) with \(v(0)=v(T)\).

Here $m=1$, $\mathcal D_1v=\nabla v$, and the Carleman exponents are $(c_0,c_1;o)=(3,1;3)$.
\end{proposition}
\begin{proof}
We prove the estimate by choosing a suitable weight, specializing the pointwise Carleman inequality of \cite[Corollary~3.1]{fu2019carleman}, and then absorbing the lower-order terms.

Choose a ball $B_*\Subset\omega_0$ centred at $x_*$, and set
\[
\psi=-|x-x_*|^2,
\qquad
\phi=e^{\lambda\psi},
\qquad
\rho=\phi-2.
\]
We prove the estimate for $\epsilon=1$ by specializing the pointwise inequality of \cite[Corollary~3.1]{fu2019carleman} to the periodic weight $\ell=q\rho$.

First take $v$ smooth and periodic in time with values in $D(A)$, and set $z=\theta v$. Apply the cited inequality to $v$, with weighted unknown $z$, choosing
$$
p^{jk}=\delta_{jk},
\qquad
\Psi=-\Delta\ell-8q\lambda\phi,
\qquad
\Phi=0.
$$
The total time derivative integrates to zero by periodicity. The spatial flux on each face is $2q\partial_\nu\phi|\partial_\nu z|^2\le0$ by convexity, and the face integrals are justified by the Green identities on $H^2(\Omega)\cap H_0^1(\Omega)$.

The matrix $(c^{jk})$ in (3.5) of \cite{fu2019carleman} equals
$$
4q\lambda\phi\,\mathrm{Id}_{\mathbb R^d}
+2q\lambda^2\phi\,\nabla\psi\otimes\nabla\psi.
$$
Substituting this into (3.3)--(3.5) of \cite{fu2019carleman}, and then applying Young's inequality together with \cref{eq:qclass}, yields, with $C>0$ independent of $\lambda$,
\begin{align}\label{eq:heatcross}
\int q\lambda\phi|\nabla z|^2
+\int q^3\lambda^3\phi^3|\nabla\psi|^2
(\lambda|\nabla\psi|^2-C)z^2
\le C\|\theta P_1v\|_{L^2(\mathcal Q)}^2
+C(\lambda)\int\bigl(q^2+(\beta+\beta^2)q^3\bigr)z^2.
\end{align}
The constants in this estimate are independent of $\beta,q_c,T$, and the dependence on $\lambda$ is displayed explicitly.

Choose $\lambda$ so that $\lambda|\nabla\psi|^2\ge2C$ outside $B_*$. Then choose $\beta$ sufficiently small and $q_c$ sufficiently large to absorb the right-hand remainder outside $B_*$. The remaining zeroth-order terms on $B_*$ are bounded by
$
C\int_{\mathcal Q_{\omega_0}}q^3z^2$.
Returning from $z$ to $v$ proves \cref{eq:continuousabstract} with the exponents $(c_0,c_1;o)=(3,1;3)$.

Density extends the estimate to the stated space, and time reflection gives the case $\epsilon=-1$.
\end{proof}
\begin{remark}\label{rem:zeroth-order-example}
Let $b\in C^\infty(\overline\Omega)$ be nonnegative and set $A_b=-\Delta_D+b$ on a convex polytope as above. Bounded perturbation, coercivity, and the Dirichlet regularity estimate give $D(A_b)=H^2(\Omega)\cap H_0^1(\Omega)$ and \cref{eq:ellreg}, so \cref{prop:spatial-operators} applies. In \cref{prop:continuousCE}, the additional residual $bv$ is absorbed by the $q^3$ term after increasing $q_c$, with the same Carleman exponents. Taking the exact stiffness form $a_b(u,v)=(\nabla u,\nabla v)_{L^2(\Omega)}+(bu,v)_{L^2(\Omega)}$ for consistent $P_1$ elements then gives the same transfer result.
\end{remark}

\section{Cartesian difference schemes and numerical assumptions}\label[appendix]{sec:fd-verification}

We complete the proof of \cref{prop:fd}. All norms in this appendix are spatial.

{\bf Step 1.} In this step, we verify the moments and load lift.

The basis on the infinite uniform Kuhn grid is translation invariant and forms a partition of unity. Each interior nodal star is complete and centrally symmetric, including those adjacent to the boundary. Hence
\begin{align}\label{eq:fdmoments}
\int_\Omega\phi_i\,dx=\omega_i,\qquad
\int_\Omega(x-x_i)\phi_i\,dx=0.
\end{align}
In particular, $(J_hp)_i=p(x_i)$ for affine $p$.

On each simplex $K$, with barycentric coordinates $\lambda_i$, define
\begin{align*}
(\mathcal L_hf)|_K
=\sum_{i:\,x_i\in K\cap\Omega}f_i\big((d+2)\lambda_i-1\big).
\end{align*}
For each interior vertex $x_j$, the identities
\begin{align*}
\int_K\big((d+2)\lambda_i-1\big)\lambda_j\,dx
=\frac{\delta_{ij}|K|}{d+1},\qquad
\sum_{K\ni x_j}\frac{|K|}{d+1}=\omega_j
\end{align*}
give $J_h\mathcal L_h=\mathrm{Id}_{X_h}$. Local polynomial norm equivalence and \cref{eq:element-weight-comparison} give the load stability bound
\begin{align*}
\|\sigma\mathcal L_hf\|_{L^2(\Omega)}\leq C\|f\|_{X_h,\sigma}.
\end{align*}
The reconstruction satisfies the norm equivalence
\begin{align}\label{eq:fd-local-stability}
\|\sigma R_hz\|_{L^2(\Omega)}\asymp\|z\|_{X_h,\sigma}.
\end{align}
The same argument gives the local stability of $Q_h$:
\begin{align}\label{eq:fd-projection-stability}
\|Q_hv\|_{L^2(K)}\leq C\|v\|_{L^2(B_K)},
\end{align}
where $B_K$ is the smallest axis-aligned rectangular block containing $K$ and the supports of its active basis functions. The blocks have diameter $O(h)$, bounded aspect ratios, and bounded overlap. Together with \cref{eq:fdstiffness}, these estimates prove \cref{ass:interfaces,ass:loads}.

\smallskip

{\bf Step 2.}
In this step, we establish the weighted approximation estimates.

The first lemma gives a local affine approximation that respects the physical boundary, including at corners. Let $\mathcal P_1(B_K)$ be the space of affine polynomials on $B_K$, let $\mathcal F_K$ be the set of faces of $B_K$ contained in $\partial\Omega$, and define\vspace{-2mm}
$$
\mathcal P_{1,\mathcal F_K}(B_K)
:=\{p\in\mathcal P_1(B_K):p|_F=0\text{ for every }F\in\mathcal F_K\}.
$$
\begin{lemma}\label{lem:fd-boundary-affine}
For every $v\in H^2(B_K)$ with zero trace on each $F\in\mathcal F_K$, there exists $p_K\in\mathcal P_{1,\mathcal F_K}(B_K)$ such that\vspace{-2mm}
\begin{align}\label{eq:fd-local-approximation}
\|v-p_K\|_{L^2(B_K)}+h\|\nabla(v-p_K)\|_{L^2(B_K)}
\leq Ch^2\|D^2v\|_{L^2(B_K)},
\end{align}
where the constant $C$ depends only on the dimension and the grid-step comparability bound.
\end{lemma}
\begin{proof}
Map $B_K$ to the unit cube $\widehat B$, and denote the transformed function and physical faces by $\widehat v$ and $\widehat{\mathcal F}_K$. By the Bramble-Hilbert lemma~\cite[Lemma~4.3.8]{Brenner2008}, there exists $\widetilde p\in\mathcal P_1(\widehat B)$ such that\vspace{-2mm}
\begin{align*}
\|\widehat v-\widetilde p\|_{H^1(\widehat B)}
\le C\|D^2\widehat v\|_{L^2(\widehat B)}.
\end{align*}
In the finite-dimensional space $\mathcal P_1(\widehat B)$, the kernel of the face-trace map is $\mathcal P_{1,\widehat{\mathcal F}_K}(\widehat B)$. Norm equivalence on its orthogonal complement therefore gives\vspace{-2mm}
\begin{align*}
\inf_{\widehat p\in\mathcal P_{1,\widehat{\mathcal F}_K}(\widehat B)}
\|\widetilde p-\widehat p\|_{H^1(\widehat B)}
&\le C\Bigl(\sum_{\widehat F\in\widehat{\mathcal F}_K}
\|\widetilde p\|_{L^2(\widehat F)}^2\Bigr)^{1/2}
\le C\|\widehat v-\widetilde p\|_{H^1(\widehat B)},
\end{align*}
where the last inequality uses the zero traces of $\widehat v$ together with the trace theorem. For an interior block, we may take $\widehat p=\widetilde p$. For a boundary block, we choose a minimizer $\widehat p$ in the preceding estimate. Applying the triangle inequality and scaling back to $B_K$ proves \cref{eq:fd-local-approximation}. The constants are uniform because the reference cube has only finitely many subsets of faces and the block side lengths are comparable to $h$.

Finally, if $B_K$ meets two nonparallel physical boundary faces, then the only affine function vanishing on both is zero. Hence $p_K=0$, and the same argument shows that the zero traces eliminate the affine kernel of the Hessian seminorm.
\end{proof}

The moment reconstruction then gives a global weighted estimate without using point values of $v$.
\begin{lemma}\label{lem:fd-weighted-approximation}
For every $v\in H^2(\Omega)\cap H_0^1(\Omega)$ and every allowed weight $\sigma$ with $hq\leq1$ at the fixed time,
\begin{align}\label{eq:fdapproxallD}
\|\sigma(v-Q_hv)\|_{L^2(\Omega)}+h\|\sigma\nabla(v-Q_hv)\|_{L^2(\Omega)}
\leq Ch^2\|\sigma D^2v\|_{L^2(\Omega)},
\end{align}
where the constant is independent of $h,q,\sigma$ and $v$.
\end{lemma}
\begin{proof}
Fix $K$. Choose $p_K$ as in \cref{lem:fd-boundary-affine}. The moment identities \cref{eq:fdmoments} reproduce $p_K$ at every interior vertex of $K$, while the boundary vertices have value zero by the choice of $p_K$. Hence $Q_hp_K=p_K$ on $K$, and therefore\vspace{-2mm}
$$
(v-Q_hv)|_K=(v-p_K)|_K-Q_h(v-p_K)|_K.
$$
Since only values on $B_K$ enter this local action of $Q_h$, \cref{eq:fd-projection-stability}, the finite-element inverse estimate, and \cref{eq:fd-local-approximation} give\vspace{-2mm}
\begin{align*}
\|v-Q_hv\|_{L^2(K)}+h\|\nabla(v-Q_hv)\|_{L^2(K)}
&\le C\bigl(\|v-p_K\|_{L^2(B_K)}
+h\|\nabla(v-p_K)\|_{L^2(B_K)}\bigr)\\
&\le Ch^2\|D^2v\|_{L^2(B_K)}.
\end{align*}
It remains to pass from this local estimate to the global one. Since $hq\le1$, \cref{eq:allowed-weight-derivatives} makes $\sigma$ uniformly comparable to any of its values on each $B_K$. Multiplying the local estimate by such a value, squaring, and summing over $K$ therefore yields \cref{eq:fdapproxallD}, by the bounded overlap of the blocks. No point values of $v$ are used, so the argument applies in every fixed dimension.
\end{proof}

Since $hq\leq1$, \cref{lem:fd-weighted-approximation} proves \cref{Papprox} with $a_E=Ch$ and $a_0=Ch^2$.

\smallskip

{\bf Step 3.}
In this step, we verify the weighted test and observation estimates.

The diagonal map $\mathscr T_{\sigma,h}$ is $T$-periodic, positive and self-adjoint, and \cref{eq:allowed-weight-derivatives,eq:fd-local-stability} give \cref{eq:test-norm,eq:testtime}. Moreover,
$R_h\mathscr T_{\sigma,h}z=I_h(\sigma^2R_hz)$, where $I_h$ is nodal interpolation. On each element, normalize the weight by $\sigma_K$ from \cref{eq:element-weight-comparison}. For $hq\leq\delta_{\rm num}$, this normalized weight is uniformly comparable to one and its derivatives of order $j\leq2$ are $O(q^j)$. The local superapproximation argument in \cite[Theorem~2.1]{Demlow2010}, rescaled to a unit-size element with weight scale $(h_Kq)^{-1}$, yields
\begin{align*}
&\|\sigma^{-1}\!(R_h\mathscr T_{\sigma,h}z\!-\!\sigma^2R_hz)\|_{L^2(\Omega)}\! +h\|\sigma^{-1}\!\nabla(R_h\mathscr T_{\sigma,h}z\!-\!\sigma^2R_hz)\|_{L^2(\Omega)} \!\leq\! Ch^2\big(q\|\sigma\nabla R_hz\|_{L^2(\Omega)}\!
+\!q^2\|\sigma R_hz\|_{L^2(\Omega)}\big).
\end{align*}
Here the interpolated product is smooth on each simplex, so the estimate holds in every fixed dimension. Since $hq\leq1$, it implies \cref{eq:testsuper}. Finally, for small $h$, every vertex of an element meeting $\omega_0$ lies in $\omega$; local norm equivalence gives \cref{eq:observationinterface}.

\section{Consistency for the control limit}\label[appendix]{sec:control-consistency}

We record the consistency properties used in \cref{prop:control-convergence}. Throughout this appendix, $T$, $y_0$, and $b$ are fixed; spatial limits are taken over the admissible meshes, and temporal limits over $\delta t=T/N$ with $N\in\mathbb N$. Joint limits impose no coupling between $h$ and $\delta t$. The definitions give \cref{eq:conv-pairings}, and the reconstructions satisfy
\begin{align}\label{eq:conv-consistency}
\|E_hz-R_hz\|_{L^2(\Omega)}\le Ch\|R_hz\|_{H^1(\Omega)}.
\end{align}
For finite differences, this estimate follows by scaling on each dual-cell patch, since $z_i-R_hz$ vanishes for constant nodal data; the same bound on the uncovered boundary strip of width $O(h)$ follows from the zero-trace Poincar\'e estimate. For finite elements, we simply have $E_h=R_h$. Finally, in both realizations the exact stiffness identity\vspace{-2mm}
\begin{equation}\label{eq:conv-stiffness}
(A_hz,w)_{X_h}=a(R_hz,R_hw)\vspace{-2mm}
\end{equation}
holds, either by definition or by \cref{eq:fdstiffness}.

We now record the consistency properties of the discretized data. For finite elements, let $b_{h,\delta t}=b^n$ on each time cell; for finite differences, let $b_{h,\delta t}=b_i^n$ on $(t_{n-1},t_n]\times D_i$, extended by zero. Standard projection and cell-average approximation then yield\vspace{-2mm}
$$
\|y_{0,h}\|_{X_h}\le\|y_0\|_H\vspace{-2mm}
$$
and\vspace{-2mm}
\begin{align}\label{eq:conv-data-consistency}
\lim_{h\to0}\|E_hy_{0,h}-y_0\|_H=0.\vspace{-1mm}
\end{align}
The averaged potential satisfies $\|b_{h,\delta t}\|_{L^\infty(\mathcal Q)}\le M$ and
\begin{align}\label{eq:conv-potential-consistency}
\lim_{h,\delta t\to0}\|b_{h,\delta t}-b\|_{L^2(\mathcal Q)}=0,
\end{align}
and, on the $n$th time cell, we have the exact pairing
$$
(\mathsf V_h^nz,w)_{X_h}=(b_{h,\delta t}E_hz,E_hw)_H.
$$

We conclude by recording the consistency of the test functions. For a fixed finite sum $\phi(t,x)=\sum_\ell\eta_\ell(t)\psi_\ell(x)$ with $\eta_\ell\in C^\infty([0,T])$ and $\psi_\ell\in C_c^\infty(\Omega)$, set $\phi_h^n=I_h\phi(t_n)$ and $\Phi_{h,\delta t}^-(t)=\phi_h^{n-1}$ on $(t_{n-1},t_n]$. Here $I_h$ denotes nodal interpolation for finite elements and nodal sampling for finite differences. Spatial interpolation and time sampling yield
\begin{align}\label{eq:conv-test-consistency}
\lim_{h,\delta t\to0}\|R_h\Phi_{h,\delta t}^--\phi\|_{L^2(0,T;V)}=0.
\end{align}
The reconstruction using $E_h$ satisfies\vspace{-2mm}
\begin{align*}
\lim_{h,\delta t\to0}\|E_h\Phi_{h,\delta t}^--\phi\|_{L^\infty(\mathcal Q)}=0.\vspace{-2mm}
\end{align*}
The time difference quotient, defined piecewise on $(t_{n-1},t_n]$ by the expression below, satisfies\vspace{-2mm}
\begin{align*}
\lim_{h,\delta t\to0}
\biggl\|E_h\biggl(\frac{\phi_h^n-\phi_h^{n-1}}{\delta t}\biggr)
-\partial_t\phi\biggr\|_{L^2(0,T;H)}=0.\vspace{-2mm}
\end{align*}
At the endpoints, spatial interpolation gives
\begin{align*}
\lim_{h\to0}\|E_hI_h\phi(\tau)-\phi(\tau)\|_H=0,
\qquad \tau\in\{0,T\}.
\end{align*}
Finally, the potential and test-function limits give
\begin{align*}
\lim_{h,\delta t\to0}\|b_{h,\delta t}E_h\Phi_{h,\delta t}^--b\phi\|_{L^2(\mathcal Q)}=0.
\end{align*}

\section{Proofs of the filtering estimates}\label{sec:filter-proofs}

We use the notation and hypotheses of \Cref{sec:filtering}. The proofs are given in two parts: the reverse norm estimate in \Cref{sec:filter-reverse-proof}, and the complementary-penalty construction in \Cref{sec:filter-control-proof}.

\subsection{Reverse norm estimate}\label{sec:filter-reverse-proof}

We first prove that terminal adjoint data in the filtered subspace remain controlled by the initial adjoint norm.
\begin{proof}[Proof of \cref{lem:filter-reverse}]
All norms and inner products in this proof are those of $X_h$. For $z\ne0$, put
$$
\nu(z)=\frac{(A_hz,z)}{\|z\|^2}.
$$
\textbf{One resolvent step.}
Let $V=V^*$ with $\|V\|\le M$, and let
$$
x=[\mathrm{Id}+\delta t(A_h+V)]z.
$$
Normalize $\|z\|=1$, set $r=\nu(z)$ and $b=(Vz,z)$, and define the components orthogonal to $z$ by
$$
v=(A_h-r\mathrm{Id})z,\qquad w=(V-b\mathrm{Id})z.
$$
Put $\xi=v+w$ and $\gamma=1+\delta t(r+b)$. Then $\xi\perp z$, $\|w\|\le M$, $x=\gamma z+\delta t\,\xi$, and
$$
\|x\|^2=\gamma^2+(\delta t)^2\|\xi\|^2.
$$
Since $A_h\ge0$, completion of the square gives
\begin{align}\label{9.26-eq17}
	(A_hx,x)-r\|x\|^2
	&=2\gamma\delta t(v,\xi)+(\delta t)^2((A_h-r\mathrm{Id})\xi,\xi)\notag\\
	&\ge\delta t(2\gamma-\delta t\,r)\|\xi\|^2-2\gamma\delta t(w,\xi)\\
	&\ge-\frac{\delta t\,\gamma^2M^2}{2\gamma-\delta t\,r},\notag
\end{align}
where $2\gamma-\delta t\,r\ge2(1-\delta t M)>0$. Dividing \eqref{9.26-eq17} by $\|x\|^2$ and using $\|x\|^2\ge\gamma^2$, we obtain
\begin{align}\label{eq:filter-quotient-step}
	\nu(z)\le\nu(x)+\frac{\delta t\,M^2}{2(1-\delta t M)}
	\le\nu(x)+\delta t\,M^2.
\end{align}

\textbf{Iteration along the backward Euler recursion.}
Put $z_j=p^{N-j}$ and $V_j=\mathsf V_h^{N-j+1}$, so that
\begin{equation}\label{9.26-eq18}
[\mathrm{Id}+\delta t(A_h+V_j)]z_j=z_{j-1}.
\end{equation}
If $p^N\ne0$, all $z_j$ are nonzero, since each resolvent is invertible. Writing $\nu_j=\nu(z_j)$, \cref{eq:filter-quotient-step} gives
$$
\nu_j\le\nu_0+j\delta t M^2.
$$

Applying \(S_0=(\mathrm{Id}+\delta t A_h)^{-1}\) to both sides of \eqref{9.26-eq18} gives
\[
z_j=S_0z_{j-1}-\delta t\,S_0V_jz_j.
\]
Taking norms and using the reverse triangle inequality
$
\|a-b\|\ge \|a\|-\|b\|$,
we obtain
\begin{equation}\label{9.26-eq19}
\|z_j\|
=\|S_0z_{j-1}-\delta t\,S_0V_jz_j\|
\ge \|S_0z_{j-1}\|-\delta t\,\|S_0V_jz_j\|.
\end{equation}
Since $\|S_0\|\le1$ and $\|V_j\|\le M$, we have
$$
\|S_0V_jz_j\|
\le \|S_0\|\,\|V_jz_j\|
\leq M\|z_j\|.
$$
This, together with \eqref{9.26-eq19}, implies\vspace{-2mm}
\begin{equation}\label{9.26-eq20}
\|z_j\|
\ge \|S_0z_{j-1}\|-\delta t\,M\|z_j\|.\vspace{-2mm}
\end{equation}
Adding $\delta t\,M\|z_j\|$ to both sides yields\vspace{-2mm}
$$
(1+\delta t\,M)\|z_j\|
\ge \|S_0z_{j-1}\|.\vspace{-2mm}
$$
Thus\vspace{-2mm}
\begin{equation}\label{9.26-eq21}
\|z_j\|
\ge \frac{\|S_0z_{j-1}\|}{1+\delta t\,M}.\vspace{-2mm}
\end{equation}

It remains to estimate $\|S_0z_{j-1}\|$ from below. By the spectral lower bound for $S_0$,
$$
\|S_0x\|\ge \frac{\|x\|}{1+\delta t\,\nu(x)},
\qquad x\ne0.
$$
Taking $x=z_{j-1}$ gives\vspace{-2mm}
\begin{equation}\label{9.26-eq22}
\|S_0z_{j-1}\|
\ge \frac{\|z_{j-1}\|}{1+\delta t\,\nu(z_{j-1})}
=
\frac{\|z_{j-1}\|}{1+\delta t\,\nu_{j-1}}.
\end{equation}

Combining \cref{9.26-eq21,9.26-eq22}, we obtain\vspace{-2mm}
\begin{equation}\label{9.26-eq23}
\|z_j\|\ge\frac{\|z_{j-1}\|}{(1+\delta t M)(1+\delta t\,\nu_{j-1})}.
\end{equation}

\textbf{Bounding the norm ratio.}
Taking logarithms of \cref{9.26-eq23} and summing, we obtain\vspace{-2mm}
\begin{align*}
	\log\frac{\|z_N\|}{\|z_0\|}
	&\ge-MT-\delta t\sum_{j=1}^N\nu_{j-1}
	\ge-MT-T\nu_0-\frac12M^2T^2.
\end{align*}
Since $z_0=p^N\in F$ gives $\nu_0\le\Lambda_F$, this proves \cref{eq:filter-reverse}. The zero terminal datum is immediate.
\end{proof}

\subsection{Complementary terminal penalty and filtered control}\label{sec:filter-control-proof}

We next prove the filtered-control result by combining the reverse estimate with the relaxed observability inequality.

\begin{proof}[Proof of \cref{cor:filtered-control}]
Write $P=\Pi_F$, $Q=\mathrm{Id}_{X_h}-P$, and $\varepsilon=e^{-d_0S}$. The proof is divided into four steps: we first derive a complementary observability inequality, then introduce a penalized functional and identify its minimizer, next construct the control and estimate the terminal state, and finally estimate the penalization parameter $\eta_F$.

\medskip

\textbf{Step 1: A complementary observability inequality.}
For a terminal datum $z\in X_h$, let $p[z]$ solve \cref{eq:adjoint}, and define
$$
\mathscr Rz=p[z]^0,\qquad
\mathscr Oz=(B_h^*p[z]^{n-1})_{n=1}^N.
$$
By \cref{eq:adjointenergy}, the operator $\mathscr R$ satisfies
$$
\|\mathscr R\|\le L_{T,M}:=e^{2MT}.
$$
Moreover, \cref{eq:filter-reverse} applied to $Pz\in F$ gives
\begin{equation}\label{9.26-eq24}
\|Pz\|_{X_h}\le K_F\|\mathscr RPz\|_{X_h}
\le K_F\bigl(\|\mathscr Rz\|_{X_h}+L_{T,M}\|Qz\|_{X_h}\bigr).
\end{equation}
This, together with  the orthogonality of $P$ and $Q$, implies
\begin{equation}\label{9.26-eq25}
\|z\|_{X_h}^2
\le 2K_F^2\|\mathscr Rz\|_{X_h}^2
+(1+2L_{T,M}^2K_F^2)\|Qz\|_{X_h}^2.
\end{equation}
Set
$
\eta_F^2=\varepsilon^2(1+2L_{T,M}^2K_F^2).
$
Substituting \eqref{9.26-eq25} into \cref{eq:mainO} and absorbing the term $2C_0\varepsilon^2K_F^2\le1/2$ yields the complementary observability inequality
\begin{align}\label{eq:filter-complementary-observability}
	\|\mathscr Rz\|_{X_h}^2
	&\le2C_0\bigl(\|\mathscr Oz\|_{\ell_{\delta t}^2(U_h)}^2+\eta_F^2\|Qz\|_{X_h}^2\bigr).
\end{align}

\medskip

\textbf{Step 2: The penalized functional and its minimizer.}
Consider on $X_h$ the functional
\begin{align}\label{eq:filter-functional}
	J_F(z)=\tfrac12\|\mathscr Oz\|_{\ell_{\delta t}^2(U_h)}^2
	+\tfrac12\eta_F^2\|Qz\|_{X_h}^2-(y_{0,h},\mathscr Rz)_{X_h}.
\end{align}
Its quadratic part is positive definite. Indeed, if both quadratic terms vanish, then \cref{eq:filter-complementary-observability} gives $\mathscr Rz=0$, and the invertibility of the resolvents implies $z=0$. Since $X_h$ is finite-dimensional, $J_F$ therefore has a unique minimizer $\widehat z\in X_h$. Its optimality condition reads
\begin{align*}
	(\mathscr O\widehat z,\mathscr Oz)_{\ell_{\delta t}^2(U_h)}
	+\eta_F^2(Q\widehat z,Qz)_{X_h}
	=(y_{0,h},\mathscr Rz)_{X_h},
	\qquad z\in X_h.
\end{align*}

\medskip

\textbf{Step 3: Construction of the control and estimate of the terminal state.}
Choose the control
$
u^n=-B_h^*p[\widehat z]^{n-1}$, $1\le n\le N$,
and let $y$ solve \cref{eq:forward}. Combining the optimality condition with the duality identity \cref{eq:trueD} gives
$
y^N=\eta_F^2Q\widehat z,
$
and hence $Py^N=0$. It remains to bound the control and the terminal state. Taking $z=\widehat z$ in the optimality condition yields
\begin{equation}\label{9.26-eq26}
\|\mathscr O\widehat z\|_{\ell_{\delta t}^2(U_h)}^2
+\eta_F^2\|Q\widehat z\|_{X_h}^2
=(y_{0,h},\mathscr R\widehat z)_{X_h}.
\end{equation}
By \cref{eq:filter-complementary-observability}, the right-hand side of \eqref{9.26-eq26} is bounded by
$
\|y_{0,h}\|_{X_h}
\bigl(2C_0(\|\mathscr O\widehat z\|_{\ell_{\delta t}^2(U_h)}^2
+\eta_F^2\|Q\widehat z\|_{X_h}^2)\bigr)^{1/2}.
$
Therefore,
\begin{equation}\label{9.26-eq27}
\|\mathscr O\widehat z\|_{\ell_{\delta t}^2(U_h)}^2
+\eta_F^2\|Q\widehat z\|_{X_h}^2
\le2C_0\|y_{0,h}\|_{X_h}^2.
\end{equation}
Since $u^n=-B_h^*p[\widehat z]^{n-1}$, it follows from \eqref{9.26-eq27} that
$$
\|u\|_{\ell_{\delta t}^2(U_h)}^2+\eta_F^2\|Q\widehat z\|_{X_h}^2
\le2C_0\|y_{0,h}\|_{X_h}^2.
$$
In particular,
$$
\|y^N\|_{X_h}
=\eta_F^2\|Q\widehat z\|_{X_h}
\le\sqrt{2C_0}\,\eta_F\|y_{0,h}\|_{X_h}.
$$

\medskip

\textbf{Step 4: Estimate of $\eta_F$.}
Finally, since $K_F\ge1$ and \cref{eq:filter-condition} holds, we have
$$
\eta_F\le\sqrt{1+2L_{T,M}^2}\,e^{D_{T,M}}e^{-(1-\vartheta)d_0S}.
$$
Combining this with the bounds obtained in Step 3 proves
\cref{eq:filtered-control,eq:filtered-terminal}.
\end{proof}

\bibliographystyle{siam-paper-no}
\bibliography{references}
\end{document}